\documentclass{compositio}

\usepackage{lmodern}
\usepackage{dsfont}
\usepackage[centertags]{amsmath}
\usepackage[usenames,dvipsnames]{xcolor}
\usepackage[colorlinks=true,linkcolor=Red,urlcolor=NavyBlue,citecolor=blue]{hyperref}
\usepackage{amsfonts}
\usepackage{MnSymbol}
\usepackage{amsthm}
\usepackage{newlfont}
\usepackage{amscd}
\usepackage{amsmath}
\usepackage{enumerate}

\usepackage[all,2cell]{xy}
\usepackage{tikz-cd}
\usetikzlibrary{positioning, calc}
\usepackage{mathrsfs}
\UseAllTwocells
\input xy
\xyoption{2cell}
\xyoption{all}
\usepackage{verbatim}
\usepackage{eucal}
\usepackage{graphicx}

\usepackage[T1]{fontenc}

\usepackage{setspace}

\usepackage{pict2e}

\makeatletter
\newcommand{\adjunction}{\@ifstar\named@adjunction\normal@adjunction}
\newcommand{\normal@adjunction}[4]{%
	#1\colon #2%
	\mathrel{\vcenter{%
			\offinterlineskip\m@th
			\ialign{%
				\hfil$##$\hfil\cr
				\longrightharpoonup\cr
				\noalign{\kern-.3ex}
				\smallbot\cr
				\longleftharpoondown\cr
			}%
	}}%
	#3 \noloc #4%
}
\newcommand{\named@adjunction}[4]{%
	#2%
	\mathrel{\vcenter{%
			\offinterlineskip\m@th
			\ialign{%
				\hfil$##$\hfil\cr
				\scriptstyle#1\cr
				\noalign{\kern.1ex}
				\longrightharpoonup\cr
				\noalign{\kern-.3ex}
				\smallbot\cr
				\longleftharpoondown\cr
				\scriptstyle#4\cr
			}%
	}}%
	#3%
}
\newcommand{\longrightharpoonup}{\relbar\joinrel\rightharpoonup}
\newcommand{\longleftharpoondown}{\leftharpoondown\joinrel\relbar}
\newcommand\noloc{%
	\nobreak
	\mspace{6mu plus 1mu}
	{:}
	\nonscript\mkern-\thinmuskip
	\mathpunct{}
	\mspace{2mu}
}
\newcommand{\smallbot}{%
	\begingroup\setlength\unitlength{.15em}%
	\begin{picture}(1,1)
		\roundcap
		\polyline(0,0)(1,0)
		\polyline(0.5,0)(0.5,1)
	\end{picture}%
	\endgroup
}
\makeatother

\makeatletter
\newcommand*{\doublerightarrow}[2]{\mathrel{
		\settowidth{\@tempdima}{$\scriptstyle#1$}
		\settowidth{\@tempdimb}{$\scriptstyle#2$}
		\ifdim\@tempdimb>\@tempdima \@tempdima=\@tempdimb\fi
		\mathop{\vcenter{
				\offinterlineskip\ialign{
					\hbox to\dimexpr\@tempdima+1.5em{##}\cr
					\rightarrowfill\cr\noalign{\kern.5ex}
					\rightarrowfill\cr
		}}}%
		\limits^{\!#1}_{\!#2}}}
\newcommand*{\triplerightarrow}[1]{\mathrel{
		\settowidth{\@tempdima}{$\scriptstyle#1$}
		\mathop{\vcenter{
				\offinterlineskip\ialign{\hbox to\dimexpr\@tempdima+1em{##}\cr
					\rightarrowfill\cr\noalign{\kern.5ex}
					\rightarrowfill\cr\noalign{\kern.5ex}
					\rightarrowfill\cr}}}\limits^{\!#1}}}
\makeatother

\newcommand{\NN}{\mathbb{N}}

\newcommand{\ZZ}{\mathbb{Z}}
\newcommand{\LL}{\mathbb{L}}
\newcommand{\RR}{\mathbb{R}}

\def\sE{E}

\def\QQ{\mathbb{Q}}
\def\C{\mathcal{C}}
\def\D{\mathcal{D}}
\def\G{\mathcal{G}}

\def\F{\mathcal{F}}
\def\SS{\mathcal{S}}

\def\I{\mathcal{I}}
\def\cS{\mathcal{S}}
\def\T{\mathcal{T}}
\def\P{\mathcal{P}}
\def\Q{\mathcal{Q}}

\def\R{\mathcal{R}}
\renewcommand{\L}{\mathcal{L}}

\def\O{\mathcal{O}}

\def\wtl{\widetilde}

\def\K{\mathcal{K}}

\def\alp{{\alpha}}

\def\eps{{\varepsilon}}

\def\Om{{\Omega}}

\def\vphi{\varphi}

\def\lrarsimeq{\overset{\simeq}{\lrar}}

\def\l{\langle}
\def\r{\rangle}

\newcommand{\HSwarrow}{\kern0.05ex\vcenter{\hbox{\Huge\enemath{\Swarrow}}}\kern0.05ex}
\newcommand{\hSwarrow}{\kern0.05ex\vcenter{\hbox{\huge\ensuremath{\Swarrow}}}\kern0.05ex}
\newcommand{\LLSwarrow}{\kern0.05ex\vcenter{\hbox{\LARGE\ensuremath{\Swarrow}}}\kern0.05ex}
\newcommand{\LSwarrow}{\kern0.05ex\vcenter{\hbox{\Large\ensuremath{\Swarrow}}}\kern0.05ex}

\newcommand{\HSearrow}{\kern0.05ex\vcenter{\hbox{\Huge\ensuremath{\Searrow}}}\kern0.05ex}
\newcommand{\hSearrow}{\kern0.05ex\vcenter{\hbox{\huge\ensuremath{\Searrow}}}\kern0.05ex}
\newcommand{\LLSearrow}{\kern0.05ex\vcenter{\hbox{\LARGE\ensuremath{\Searrow}}}\kern0.05ex}
\newcommand{\LSearrow}{\kern0.05ex\vcenter{\hbox{\Large\ensuremath{\Searrow}}}\kern0.05ex}

\newcommand{\HDownarrow}{\kern0.05ex\vcenter{\hbox{\Huge\ensuremath{\Downarrow}}}\kern0.05ex}
\newcommand{\hDownarrow}{\kern0.05ex\vcenter{\hbox{\huge\ensuremath{\Downarrow}}}\kern0.05ex}
\newcommand{\LLDownarrow}{\kern0.05ex\vcenter{\hbox{\LARGE\ensuremath{\Downarrow}}}\kern0.05ex}
\newcommand{\LDownarrow}{\kern0.05ex\vcenter{\hbox{\Large\ensuremath{\Downarrow}}}\kern0.05ex}

\newcommand{\HUparrow}{\kern0.05ex\vcenter{\hbox{\Huge\ensuremath{\Uparrow}}}\kern0.05ex}
\newcommand{\hUparrow}{\kern0.05ex\vcenter{\hbox{\huge\ensuremath{\Uparrow}}}\kern0.05ex}
\newcommand{\LLUparrow}{\kern0.05ex\vcenter{\hbox{\LARGE\ensuremath{\Uparrow}}}\kern0.05ex}
\newcommand{\LUparrow}{\kern0.05ex\vcenter{\hbox{\Large\ensuremath{\Uparrow}}}\kern0.05ex}

\newcommand\restr[2]{{
		\left.\kern-\nulldelimiterspace 
		#1 
		\vphantom{\big|} 
		\right|_{#2} 
}}

\newtheorem{thm}{Theorem}[subsection]
\newtheorem{cor}[thm]{Corollary}
\newtheorem{lem}[thm]{Lemma}
\newtheorem{prop}[thm]{Proposition}
\newtheorem{pro}[thm]{Proposition}

\makeatletter
\@addtoreset{thm}{section}
\makeatother

\numberwithin{thm}{subsection}
\numberwithin{equation}{subsection}

\theoremstyle{definition}
\newtheorem{define}[thm]{Definition}

\newtheorem{example}[thm]{Example}

\newtheorem{dfn}[thm]{Definition}
\newtheorem{cons}[thm]{Construction}

\newtheorem{notn}[thm]{Notation}

\newtheorem{obs}[thm]{Observation}

\newtheorem{conv}[thm]{Convention}

\newtheorem{rem}[thm]{Remark}

\DeclareMathOperator{\holim}{holim}
\DeclareMathOperator{\hocolim}{hocolim}

\DeclareMathOperator{\Com}{Com}

\DeclareMathOperator{\Id}{Id}
\DeclareMathOperator{\sD}{D}
\DeclareMathOperator{\id}{id}
\DeclareFontFamily{OT1}{pzc}{}
\DeclareFontShape{OT1}{pzc}{m}{it}{<-> s * [1.10] pzcmi7t}{}
\DeclareMathAlphabet{\mathpzc}{OT1}{pzc}{m}{it}
\DeclareMathOperator{\cof}{cof}

\DeclareMathOperator{\Alg}{Alg}
\DeclareMathOperator{\Ass}{Ass}
\DeclareMathOperator{\Free}{Free}
\DeclareMathOperator{\LMod}{LMod}

\DeclareMathOperator{\RMod}{RMod}
\DeclareMathOperator{\ModCat}{ModCat}

\DeclareMathOperator{\Rect}{Rect}

\DeclareMathOperator{\op}{op}
\DeclareMathOperator{\rE}{E}
\DeclareMathOperator{\Map}{Map}

\DeclareMathOperator{\red}{red}
\DeclareMathOperator{\inj}{inj}
\DeclareMathOperator{\proj}{proj}

\DeclareMathOperator{\Sets}{Sets}
\DeclareMathOperator{\Top}{Top}

\DeclareMathOperator{\Cat}{Cat}

\DeclareMathOperator{\CAlg}{CAlg}
\DeclareMathOperator{\Mod}{Mod}

\DeclareMathOperator{\Fun}{Fun}

\DeclareMathOperator{\Ob}{Ob}
\DeclareMathOperator{\Sp}{Sp}
\DeclareMathOperator{\Ho}{Ho}

\DeclareMathOperator{\Op}{Op}

\DeclareMathOperator{\der}{h}
\DeclareMathOperator{\BMod}{BMod}
\DeclareMathOperator{\aug}{aug}

\DeclareMathOperator{\Aut}{Aut}

\DeclareMathOperator{\h}{h}

\DeclareMathOperator{\sMod}{sMod}

\DeclareMathOperator{\Der}{Der}

\DeclareMathOperator{\hocofib}{hocofib}

\DeclareMathOperator{\Sec}{Sec}
\DeclareMathOperator{\coc}{coc}

\DeclareMathOperator{\Def}{Def}
\DeclareMathOperator{\defi}{def}

\DeclareMathOperator{\art}{art}

\DeclareMathOperator{\IbMod}{IbMod}
\DeclareMathOperator{\ILMod}{I\ell Mod}
\DeclareMathOperator{\Hom}{Hom}
\DeclareMathOperator{\rL}{L}

\DeclareMathOperator{\sH}{H}

\DeclareMathOperator{\sN}{N}

\DeclareMathOperator{\st}{t}

\DeclareMathOperator{\sB}{b}

\DeclareMathOperator{\Coll}{Coll}

\DeclareMathOperator{\Set}{Set}

\DeclareMathOperator{\E}{E}

\DeclareMathOperator{\End}{End}
\DeclareMathOperator{\Fin}{Fin}

\DeclareMathOperator{\sS}{S}
\DeclareMathOperator{\HHH}{HH}
\DeclareMathOperator{\HHQ}{HQ}

\DeclareMathOperator{\sr}{r}
\DeclareMathOperator{\si}{i}
\DeclareMathOperator{\ir}{ir}

\DeclareMathOperator{\Tan}{Tan}

\def\x{\overset}

\def\Hom{\textrm{Hom}}
\def\End{\textrm{End}}

\newcommand{\tgpd}{\kern0.05ex\vcenter{\hbox{\footnotesize\ensuremath{2}}}\kern0.05ex\mathcal{G}pd} 

\def\rar{\rightarrow}

\def\lrar{\longrightarrow}
\def\llar{\longleftarrow}

\def\ovl{\overline}

\title[Deformation theory and cotangent complex of dg operads]%
{Deformation theory and cotangent complex of \\  dg operads}

\author{\textbf{\texttt{Yonatan Harpaz}}\,}
\address{Institut de Mathématiques de Jussieu-Paris Rive Gauche \newline Email: yonatan.harpaz@imj-prg.fr}

\author{\,\textbf{\texttt{Truong Hoang}}}
\address{Department of Mathematics, Hanoi FPT University \newline Email: truonghm@fe.edu.vn}

\usepackage{fancyhdr}
\usepackage{lipsum} 
\usepackage{xcolor}

\begin{document}

\begin{abstract} In the first part, we give an explicit description of the cotangent complex of differential graded (dg) operads, modeled as an operadic infinitesimal bimodule. This leads to a uniform formula for the Quillen cohomology of their associated algebras. We further show that the cotangent complex of the dg $\E_\infty$-operad is represented by the Pirashvili functor, while that of the dg $\E_n$-operad is conveniently described via its Hochschild complex. In the second part, we establish an explicit relation between deformation theory and (spectral) Quillen cohomology for various types of algebraic objects. Combining these results, we obtain a formulation of the space of first-order deformations of dg operads, which is particularly convenient in the case of dg $\E_n$-operads.  
\end{abstract} 

\maketitle

\tableofcontents

\bigskip

\fontsize{10.5}{13}\selectfont

\medskip

\section{Introduction}\label{s:introduction}

Any \textbf{operad} determines a category of (\textbf{operadic}) \textbf{algebras} over it. In this way, operads encode familiar algebraic structures such as \textbf{associative}, \textbf{commutative}, and \textbf{Lie algebras}, as well as more intricate structures arising in a topological setting, such as $\E_n$-\textbf{algebras}. A natural principle suggests that, for a given type of operadic algebras, any property shared by a sufficiently broad class of algebras can be interpreted in terms of intrinsic features of the operad itself. Below, we present some illustrative examples.

\begin{enumerate}[(1)]
	\item Let $\P$ be an operad in a (sufficiently nice) \textit{symmetric monoidal model category} $\cS$. It has been observed that whenever $\P$ is $\Sigma$-\textbf{cofibrant}, the \textit{category of $\P$-algebras} admits a \textit{semi-model structure} transferred from the model structure on $\cS$. Moreover, the resulting semi-model category is of the right type: any weak equivalence between $\Sigma$-cofibrant operads induces a \textit{Quillen equivalence} between the corresponding categories of algebras (see e.g., \cite{Fresse1, Spitzweck}). This demonstrates that the \textbf{little $n$-discs operad}  provides a right model for the homotopy theory of  $\E_n$-algebras for every $1\leq n < \infty$. On the other hand, it highlights the role of $\E_\infty$-\textbf{algebras} and $\rL_\infty$-\textbf{algebras} as (homotopically) refined versions of ordinary commutative and Lie algebras, respectively.   
	
	\item As in the classical algebraic setting, the (co)homology theories of operadic algebras are of central significance. Initially, the (\textit{operadic}) \textit{cohomology of $\P$-algebras} was established for \textbf{quadratic} (\textbf{dg}) \textbf{operads}, where one can see explicitly how the associated cochain complexes are controlled by the operations of $\P$. In addition, if $\P$ is \textbf{Koszul}, the resulting cohomology agrees with the \textbf{André-Quillen cohomology}. (See \cite[$\S$12.4]{Loday} for more details.)
	
	\item In \cite{Hoang1}, a formalism of \textbf{Quillen}  (and \textbf{Hochschild}) \textbf{cohomology} for operadic algebras has been developed in a highly general setting, which employs Lurie's \textit{cotangent complex formalism} \cite{Lurieha} together with the foundational results on \textit{operadic tangent categories} due to Harpaz-Nuiten-Prasma \cite{Yonatan}. The main results of \cite{Hoang1} show how the Quillen and Hochschild cohomology of $\P$-algebras can be expressed in terms of the corresponding cohomology of $\P$ (once and for all such algebras). Furthermore, an explicit relation between these two cohomology theories for $\E_n$-\textbf{spaces} is obtained, based solely on the corresponding known relation for the $\E_n$-operad itself, as established in \cite{Hoang}. 
\end{enumerate}

The above discussion illustrates the role of operads in the study of their associated algebras, particularly in the context of cohomology theories. Nevertheless, operads are of independent interest, even apart from operadic algebras: $\E_n$-operads enrich the homotopy type of \textbf{configuration spaces of discs}, and enable the study of the homotopy of \textbf{long embedding spaces} via their associated \textit{infinitesimal bimodule structure} (cf. \cite{Turchin}), while also having a wide range of other applications.

With these motivations in mind, we focus in this paper on the Quillen cohomology of (\textbf{dg}) \textbf{operads}, as well as on the relation between Quillen cohomology and \textbf{deformation theory}. Together with \cite{Hoang} and \cite{Hoang1}, this paper sheds further light on cohomology theories for operads and their associated algebras, particularly in the simplicial and differential graded contexts. Furthermore, these works, along with the related series \cite{YonatanCotangent, Yonatan, YonatanBundle}, lay the foundations for the study of tangent categories and cotangent complexes in a general setting, focusing on \textit{enriched categories}, \textit{enriched operads}, and operadic algebras.

\subsection{Some terminology and notations}

For convenience in summarizing the main results, we briefly present some relevant terminology and notations.

\begin{enumerate}
	
	\item Let $\textbf{M}$ be a suitable model (or semi-model) category, and let $A \in \textbf{M}$ be an object. We denote by $\mathcal{T}_A\textbf{M}:=\Sp(\textbf{M}_{A//A})$ the \textbf{stabilization} of the \textit{category of objects over and under $A$}, and refer to it as the \textbf{tangent category to $\textbf{M}$ at $A$}. This comes equipped with a Quillen adjunction 
	$$\Sigma^{\infty}_+ : \adjunction*{}{\textbf{M}_{/A}}{\mathcal{T}_A\textbf{M}}{} : \Omega^{\infty}_+$$
	of \textbf{suspension-desuspension spectrum functors}. The \textbf{cotangent complex} of $A$ is defined to be $$\rL_A:=\mathbb{L}\Sigma^{\infty}_+(A) \in  \mathcal{T}_A\textbf{M}.$$
	The \textbf{Quillen cohomology of} $A$ with coefficients in some $M\in\mathcal{T}_A\textbf{M}$ is the space
	$$ \HHQ^\star(A ; M) :=  \Map^{\der}_{\mathcal{T}_A\textbf{M}} (\rL_A , M),$$
	and the \textbf{$n$-th Quillen cohomology group} is defined to be
	$$ \HHQ^n(A  ; M) := \pi_0 \Map^{\der}_{\mathcal{T}_A\textbf{M}} (\rL_A , M[n]).$$
	(See $\S$\ref{s:notcotan} for more details about these constructions.)
	
	\item Let $\textbf{k}$ be a fixed commutative ring, and let $\C(\textbf{k})$ denote the \textbf{category of dg $\textbf{k}$-modules}. For a \textit{set of colors} $C$, we denote by $\Op_C(\C(\textbf{k}))$ the \textbf{category of $C$-colored operads in $\C(\textbf{k})$}, and by $\Op(\C(\textbf{k}))$ the \textbf{category of operads with non-fixed sets of colors} (cf. $\S$\ref{s:opmod}).
	
	\item  An operad $\P \in \Op_C(\C(\textbf{k}))$ will typically be referred to as a (\textbf{$C$-colored}) \textbf{dg operad}. Moreover, we denote by $\IbMod(\P)$ the \textbf{category of infinitesimal $\P$-bimodules}. (See $\S$\ref{s:opinfi} for more details.) 
	
	\item We are mainly interested in the Quillen cohomology of $\P$ regarded as an object of $\Op(\C(\mathbf{k}))$. Accordingly, we write $\rL_\P \in \T_{\P}\Op(\C(\mathbf{k}))$ for the (\textit{proper}) \textit{cotangent complex of} $\P$, and denote its corresponding Quillen cohomology by $\HHQ^\star(\P; -)$.
	
	\item The \textbf{Hochschild cohomology of $\P$} with coefficients in an object $M \in \IbMod(\P)$ is the space
	   $$ \HHH^\star(\P;M) := \Map^{\h}_{\IbMod(\P)}(\P,M),$$ 
	   and the \textbf{$n$-th Hochschild cohomology group} is given by
	   $$ \HHH^n(\P;M) := \pi_0\Map^{\h}_{\IbMod(\P)}(\P,M[n]).$$
	   
	   (See \cite[$\S$5.1]{Hoang1} for more details.)
	
	\item By convention, an \textit{augmented commutative $\textbf{k}$-algebra} $\eps:R \lrar \textbf{k}$ is \textbf{artinian} if the underlying complex of $R$ is finite-dimensional and  $\eps$ makes $\pi_0(R)$ (i.e. the zeroth homology of $R$) a \textit{local} $\textbf{k}$-\textit{algebra}. We denote by $\CAlg^{\art}$ the \textbf{category of artinian $\textbf{k}$-algebras}. (See $\S$\ref{sub:fmc} for more details.)

\end{enumerate}

\subsection{Cotangent complex of dg operads}

Let $\P \in \Op_C(\C(\textbf{k}))$ be a $C$-colored operad in $\C(\textbf{k})$. We assume that $\P$ is $\Sigma$-\textit{cofibrant} and \textit{connective} (i.e., every complex of  operations is concentrated in non-negative degrees).    

According to \cite{Hoang}, there is a natural Quillen equivalence 
$$	\adjunction*{\simeq}{\IbMod(\P)}{\T_{\P}\Op(\C(\textbf{k}))}{}  $$
between the category of infinitesimal $\P$-bimodules and the \textit{tangent category to $\Op(\C(\textbf{k}))$ at} $\P$. Our first main result is stated as follows.
\begin{thm}\label{t:main1}\textup{(\ref{t:cotan})} Under the equivalence $\T_\P\Op(\C(\textbf{k})) \simeq \IbMod(\P)$, the cotangent complex $\rL_\P$ is identified with $\ovl{\rL}_\P[-1] \in \IbMod(\P)$, where $\ovl{\rL}_\P$ is defined, for each tuple $(c_1,\cdots,c_m;c)$ of colors of $\P$, by
	$$ \ovl{\rL}_\P(c_1,\cdots,c_m;c) = \P(c_1,\cdots,c_m;c)^{\oplus \, m}$$
the $m$-fold direct sum of the complex of operations $\P(c_1,\cdots,c_m;c)$. Consequently, the Quillen cohomology of $\P$ with coefficients in some $M\in\IbMod(\P)$ is computed by the formula
$$ \HHQ^\star(\P ; M)  \simeq \Map^{\h}_{\IbMod(\P)}(\ovl{\rL}_\P[-1],M),$$
and the $n$-th Quillen cohomology group is formulated as
$$  \HHQ^n(\P ; M)  \cong \pi_0\Map^{\h}_{\IbMod(\P)}(\ovl{\rL}_\P,M[n+1]).$$ 
\end{thm}

For more information, $\ovl{\rL}_\P$ can be represented as $$\ovl{\rL}_\P \cong \P\circ_{(1)}\I_C$$ the \textbf{infinitesimal composite product} of $\P$ with the \textit{initial $C$-colored operad} $\I_C$ (cf. $\S$\ref{s:opinfi}). This is equipped with an infinitesimal $\P$-bimodule structure as described in Remark~\ref{r:stlp}.

In addition, we provide a reformulation of the Quillen cohomology of $\P$-algebras as follows.
\begin{cor}\textup{(\ref{co:QPalg})} Let $A$ be a cofibrant $\P$-algebra, and $N$ an $A$-module over $\P$. The Quillen cohomology of $A$ with coefficients in $N$ can be computed by the formula
	\begin{equation}\label{eq:QPalg}
		\HHQ^\star(A ; N) \simeq \Omega \HHQ^\star(\P;\End_{A,N}) \simeq \Map^{\h}_{\IbMod(\P)}(\ovl{\rL}_\P,\End_{A,N})
	\end{equation}
where $\End_{A,N}$ is the infinitesimal $\P$-bimodule associated with the pair $(A,N)$ (see $\S$\ref{s:endo}). Consequently, the $n$-th Quillen cohomology group is given by
	$$  \HHQ^n(A ; N) \cong \pi_0\Map^{\h}_{\IbMod(\P)}(\ovl{\rL}_\P,\End_{A,N}[n]).$$
\end{cor}

\begin{rem} This is traditionally computed by
	 $$  \HHQ^\star(A ; N) \simeq \Map^{\h}_{\Mod_{\P}^A}(\Om_A, N) $$
in which $\Mod_{\P}^A$ denotes the \textit{category of 	$A$-modules over $\P$}, and $\Om_A$ denotes the \textbf{module of Kähler differentials} of $A$ (see $\S$\ref{s:cotandg}). Clearly, computations in the category $\Mod_{\P}^A$ are  less convenient than in $\IbMod(\P)$ (which is only parametrized over $\P$), and $\Omega_A$ does not admit as simple a description as $\ovl{\rL}_\P$. Moreover, formula \eqref{eq:QPalg} provides a uniform way to compute the Quillen cohomology for all $\P$-algebras.  
\end{rem}

The next theorem provides an explicit model for the cotangent complex of the \textbf{dg $\E_\infty$-operad}. 

\begin{thm}\label{t:main2}\textup{(\ref{t:LEinfty})} Up to a zigzag of Quillen equivalences, the tangent category $\T_{\E_\infty}\Op(\C(\textbf{k}))$ is equivalent to the projective model category $\Fun(\Fin_*^{\op} , \C(\textbf{k}))$, where $\Fin_*$ denotes the skeleton of the category of finite pointed sets (cf. $\S$\ref{s:opinfi}). Moreover, the cotangent complex $\rL_{\E_\infty} \in \T_{\E_\infty}\Op(\C(\textbf{k}))$ is then identified with $t[-1] : \Fin_*^{\op} \lrar \C(\textbf{k})$, i.e., the desuspension of the  Pirashvili functor $t$.
\end{thm}

\begin{rem} For more details, the  \textbf{Pirashvili functor} $t : \Fin_*^{\op} \lrar \C(\textbf{k})$ sends each pointed set $\left \langle m \right \rangle := \{0,1,\cdots,m\}$ to 
	$$t(\l m \r) := [\l m \r, \textbf{k}]_*$$ the $\textbf{k}$-module of based maps $\l m \r \lrar \textbf{k}$, where $\textbf{k}$ has base point $0_\textbf{k}$. In light of Theorem \ref{t:main2}, the Quillen cohomology of the dg $\E_\infty$-operad agrees with the \textbf{stable cohomotopy} introduced by Pirashvili \cite{Pirash}. Moreover, it leads to an elegant formulation of the \textbf{Quillen cohomology of dg $\E_\infty$-algebras} (cf. Corollary \ref{co:QEinftyalg}). 
\end{rem}

Next, we discuss the cotangent complex of the \textbf{dg version of the $\E_n$-operad}, denoted $\mathbb{E}_n$. Denote by $\mathcal{E}_*$ the $\Sigma_*$-\textit{object} that is given by $\textbf{k}$ concentrated in level $0$, and by $ \Free_{\mathbb{E}_n}^{\si}(\mathcal{E}_*)$ the \textit{free infinitesimal $\mathbb{E}_n$-bimodule} generated by $\mathcal{E}_*$.

\begin{thm}\label{t:main3}\textup{(\ref{t:QprinEn})} There is a cofiber sequence in $\IbMod(\mathbb{E}_n)$ of the form
	$$\Free_{\mathbb{E}_n}^{\si}(\mathcal{E}_*) \lrar \mathbb{E}_n^{\si} \lrar \ovl{\rL}_{\mathbb{E}_n}[n]$$
in which $\mathbb{E}_n^{\si}$ refers to $\mathbb{E}_n$ regarded as an infinitesimal bimodule over itself, and the first map is induced by the identification $\textbf{k} \overset{\cong}{\lrar} \mathbb{E}_n(0)$. Consequently, for $M\in\IbMod(\mathbb{E}_n)$, there is a fiber sequence 
$$ \Om^{n+1}\HHQ^\star(\mathbb{E}_n ; M) \lrar \HHH^\star(\mathbb{E}_n ; M) \lrar \Map^{\h}_{\C(\textbf{k})}(\textbf{k},M(0)).$$
\end{thm}

\begin{rem} Theorem \ref{t:main3} exhibits a noteworthy property of the operad $\mathbb{E}_n$, namely the connection between its Quillen and Hochschild cohomology. It may be viewed as the dg version of the \textit{Quillen principle for $\E_n$-operads} (see \cite{Hoang, Hoang1}). Moreover,  it reinterprets a significant phenomenon first observed by Quillen \cite{Quillen2} for associative algebras and later developed by Francis \cite{Francis} and Lurie \cite{Lurieha} for $\E_n$-algebras (see also \cite[$\S$1.2.3]{Hoang} for more details).
\end{rem}

An interesting corollary of Theorem \ref{t:main3} is as follows.

\begin{cor}\label{co:main3}\textup{(\ref{co:QprinEnm})} For $m\geq n$ and each $k \neq 0$, there is a canonical isomorphism 
	$$  \HHQ^{-k-n-1}(\mathbb{E}_n ; \mathbb{E}_m) \cong \HHH^{-k}(\mathbb{E}_n ; \mathbb{E}_m),$$
where $\mathbb{E}_m$ is regarded as an infinitesimal $\mathbb{E}_n$-bimodule via the canonical embedding $\mathbb{E}_n \lrar \mathbb{E}_m$. Moreover, the zeroth Hochschild cohomology group decomposes as
	$$ \HHH^{0}(\mathbb{E}_n ; \mathbb{E}_m) \cong \HHQ^{-n-1}(\mathbb{E}_n ; \mathbb{E}_m) \oplus \textbf{k}.$$
\end{cor}

\subsection{Deformation theory and Quillen cohomology}

We further assume that $\textbf{k}$ is a field of characteristic $0$, and let $R$ be an artinian $\textbf{k}$-algebra with augmentation map $\eps:R \lrar \textbf{k}$,  regarded as a ``small perturbation'' in deformation theory.   

The (\textbf{infinitesimal}) \textbf{deformation theory} studies ``infinitesimal variations'' of a given object that are structurally very similar to it: after base change along the augmentation, they recover the original object. Concretely, for an \textit{algebraic object} $X$ (e.g., a dg operadic algebra, a dg category, or a dg operad over $\textbf{k}$), a \textbf{deformation of $X$ over} $R$ is an object $Y$ of the same type, defined over $R$, and equipped with a weak equivalence $\eta : Y \otimes_{R} \textbf{k} \lrarsimeq X$. Moreover, an \textit{equivalence} $(Y,\eta) \lrarsimeq (Y',\eta')$ between deformations is a weak equivalence $Y \lrarsimeq Y'$ compatible with the maps $\eta$ and $\eta'$. (See $\S$\ref{sub:fmc} for more details.)

 We denote by $\underline{\Def}(X,R)$ the \textbf{set of deformations of $X$ over} $R$, modulo this notion of ``equivalence''. 

For $R = \textbf{k}[t]/(t^2)$ the \textbf{algebra of dual numbers}, one speaks of the \textbf{first-order deformations of $X$}.    

\begin{example}\label{ex:defclass} The deformation theory of various algebraic structures has been studied extensively over the past several decades (see, e.g. \cite{Gerten, Nijen, Loday, Lurieha}). Its early development focused on associative, commutative, and Lie algebras, and revealed a unifying feature: the equivalence classes of first-order deformations coincide with degree-two cohomology classes, i.e.,
 \begin{equation}\label{eq:defclass}
 	\underline{\Def}(A,\textbf{k}[t]/(t^2)) \cong \sH^2(A;A).
 \end{equation}
Here, the relevant cohomology theories are standard: \textbf{Hochschild}, \textbf{Harrison}, or \textbf{Lie cohomology} when $A$ is an associative, commutative, or Lie algebra, respectively. Nevertheless, this no longer holds in more general contexts (e.g., in the differential graded setting).
\end{example}

Our work refines and develops this theory in two directions. (The full results are formulated in a more abstract and general framework, using a construction called the \textbf{formal moduli context}, which is not convenient to present here. The statements below are therefore given in a simplified form.)

\smallskip
  
(1) Although these standard cohomology theories are computationally convenient, they do not yield the correct deformation invariant; instead, Quillen cohomology does. Indeed, the identification \eqref{eq:defclass} can equivalently be rewritten as
$$ \underline{\Def}(A,\textbf{k}[t]/(t^2)) \cong \HHQ^1(A;A),$$
and this formulation remains valid in far broader contexts. 

We can say even more. For an algebraic object $X$ and $R\in\CAlg^{\art}$ as above, one can organize all deformations of $X$ over $R$  into a category in which every morphism is an equivalence of deformations. We then take $\Def(X,R)$ to be the \textit{Kan replacement} of the \textit{nerve} of this category, and refer to it as the \textbf{space of deformations of $X$ over $R$}. Clearly,
$$ \pi_0\Def(X,R) \cong \underline{\Def}(X,R).$$ 
Moreover, we obtain the following simplified form of Theorem~\ref{t:main}.
\begin{thm} There is a weak equivalence
	$$ \Omega\Def(X,\textbf{k}[t]/(t^2)) \simeq \HHQ^\star(X ; X)  $$
between the loop space of the space of first-order deformations of $X$ and the Quillen cohomology of $X$ with ``coefficients in itself''. Consequently, there is a canonical isomorphism 
$$ \pi_0 \Def(X,\textbf{k}[t]/(t^2)) \cong \HHQ^1(X;X),$$ 
and for every $n\geq1$, there is a canonical isomorphism 
$$ \pi_n\Def(X,\textbf{k}[t]/(t^2)) \cong \HHQ^{1-n}(X;X).$$
\end{thm}

\smallskip

(2) Let $\textbf{S}$ denote the \textbf{$\infty$-category of spaces}. We also prove that the \textbf{deformation functor}
$$  \Def(X,-) : \CAlg^{\art}_\infty \lrar \textbf{S}, \; R \mapsto \Def(X,R)$$
forms a \textbf{formal moduli problem} in the sense of Lurie (see $\S$\ref{sub:fmc}). Thus, by the \textbf{Pridham-Lurie theorem}, the deformations of $X$ are governed by a single \textit{dg Lie algebra} (cf.  \cite{Pridham, Luriefmp}). 

\begin{example} Let $\P\in\Op(\C(\textbf{k}))$ be a (\textit{single-colored}) \textit{connective augmented dg operad}. We can then take a cofibrant resolution of $\P$ of the form $$\phi : \Omega(\C) \lrarsimeq \P$$ where $\C$ is a \textbf{dg cooperad} which models the \textbf{bar construction} of $\P$, and $\Omega(-)$ refers to the \textbf{cobar construction}. The map $\phi$ corresponds to a \textit{Maurer-Cartan element} of the \textbf{convolution dg Lie algebra} $\Hom_\Sigma(\C,\P)$. (See \cite[Chapter 6]{Loday} for more details.) Moreover, we can show that the \textbf{twisted dg Lie algebra} $\Hom_\Sigma^\phi(\C,\P)$ is precisely the dg Lie algebra classifying deformations of $\P$ (regarded as an object of $\Op(\C(\textbf{k}))$).
\end{example}

Finally, the first-order deformations of the operad $\mathbb{E}_n$ can be interpreted using the Hochschild cohomology, as stated below.

\begin{thm}\textup{(\ref{co:defEn})} There is a fiber sequence of spaces
	$$ \Om^{n+2}\Def(\mathbb{E}_n,\textbf{k}[t]/(t^2)) \lrar \HHH^\star(\mathbb{E}_n ; \mathbb{E}_n^{\si} ) \lrar \textbf{k}.$$
	Furthermore, for $k\geq -n-1$ and $k \neq 0$, there is a canonical isomorphism
	$$ \pi_{k+n+2}\Def(\mathbb{E}_n,\textbf{k}[t]/(t^2))  \cong \HHH^{-k}(\mathbb{E}_n ; \mathbb{E}_n^{\si}),$$
	and, in addition,
	$$  \HHH^{0}(\mathbb{E}_n ; \mathbb{E}_n^{\si}) \cong \pi_{n+2}\Def(\mathbb{E}_n,\textbf{k}[t]/(t^2))  \oplus \textbf{k}.$$
\end{thm}

\bigskip

\textbf{\underline{Acknowledgements}.} This paper forms part of the second author's PhD thesis, written under the supervision of the first author at Université Sorbonne Paris Nord. The authors acknowledge support from the French project ANR-16-CE40-0003 ChroK.

\section{Background  and conventions}\label{s:bnc}

We briefly recall some fundamental concepts and their notations, which will be used throughout the paper. 

\subsection{Operads and various types of modules over an operad}\label{s:opmod}

Let $(\SS, \otimes, 1_\cS)$ be a \textbf{symmetric monoidal category} such that the underlying category $\cS$ is cocomplete. We will henceforth refer to $\cS$ as the \textbf{base category}. 

\smallskip

Let $C$ be a set, regarded as the \textbf{set of colors}. We denote by $\Coll_C(\SS)$ the \textbf{category of (symmetric)} $C$-\textbf{collections} in $\cS$. For more details, an object of $\Coll_C(\SS)$ is a collection $$M = \{M(c_1,\cdots,c_n;c) \, | \, c_i,c \in C , n\geqslant0  \}$$ of objects in $\cS$ that is endowed, for each permutation $\sigma\in \Sigma_n$ and each $C$-\textbf{sequence} $(c_1,\cdots,c_n;c)$, with a map of the form $$\sigma^{*} : M(c_1,\cdots,c_n;c) \lrar M(c_{\sigma(1)},\cdots,c_{\sigma(n)};c).$$ Moreover, these maps together define a right action by $\Sigma_n$.

\begin{notn} \phantomsection\label{no:IE}
	\begin{enumerate}	 
\item We denote by $\I_C$ (resp. $\mathcal{E}_C$) the $C$-collection which is the initial object $\emptyset_\cS$  at all levels, except that $\I_C(c;c)=1_{\cS}$ (resp. $\mathcal{E}_C(c) = 1_\cS$) for every $c\in C$.

\item When $C$ is a singleton, a $C$-collection is called a $\Sigma_*$-\textbf{object}, and the corresponding category is denoted by $\Sigma_*(\cS)$. Moreover, in this case, the two above will be denoted by $\I_*$ and $\mathcal{E}_*$, respectively. 
\end{enumerate}
\end{notn}

 The object $\I_C$ represents the \textit{monoidal unit} of $\Coll_C(\SS)$ with respect to the  \textbf{composite product}
$$ -\circ- : \Coll_C(\SS) \times \Coll_C(\SS) \lrar \Coll_C(\SS)$$
(see, e.g., \cite{Yonatan}). One then has the following definition.

\begin{define}\phantomsection\label{d:opc} We denote by $\Op_C(\SS)$ the category of \textit{monoid objects} in $(\Coll_C(\SS),-\circ-,\I_C)$, and refer to it as the \textbf{category of (symmetric) $C$-colored operads in} $\cS$. Explicitly, a $C$-colored operad in $\cS$ is a $C$-collection $\P$ equipped with 
	
	\begin{enumerate} 
	\item  a \textbf{composition}, consisting of $\Sigma_*$-equivariant maps of the form 
		\begin{align*}
		& \P(c_1,\cdots,c_n;c) \otimes \P(c_{1,1},\cdots,c_{1,k_1};c_1) \otimes \cdots \otimes \P(c_{n,1},\cdots,c_{{n,k_n}};c_n) \\
		&\phantom{\P(c_1,\cdots,c_n;c) \otimes \P(c_{1,1},\cdots,c_{1,k_1};c_1)  {}} \lrar \P(c_{1,1},\cdots,c_{1,k_1}, \cdots, c_{n,1},\cdots,c_{{n,k_n}} ; c),
	\end{align*}	

	\item and for each $c \in C$, a \textbf{unit operation} $\id_c : 1_\cS \lrar \P(c;c)$.
\end{enumerate}

	\noindent These are required to satisfy the usual axioms of associativity and unitality. 
\end{define} 

We will write $\Op_*(\cS)$ to denote the \textbf{category of single-colored operads in} $\cS$.

\begin{example} \phantomsection\label{ex:ascom}
	\begin{enumerate}	 
		\item The \textbf{commutative operad} $\Com \in \Op_*(\cS)$ is defined by letting $\Com(n)=1_\cS$ for every $n\in\NN$, equipped with the trivial action by the symmetric groups. The composition is defined in an evident way.

		\item Suppose further that the monoidal product distributes over colimits in $\cS$. One can then define the \textbf{associative operad} $\Ass \in \Op_*(\cS)$ by setting $\Ass(n) = \underset{\Sigma_n}{\bigsqcup  } 1_\cS$ for every $n\in\NN$. The symmetric action is induced by the multiplication in $\Sigma_n$, and the composition is induced by the concatenation of linear orders.
	\end{enumerate}
\end{example}

\begin{define}\label{d:op} The \textbf{category of $\cS$}-\textbf{enriched operads}, denoted by $\Op(\SS)$, has as objects the pairs $(C,\P)$ where $C\in\Sets$ and $\P\in\Op_C(\SS)$. A morphism from $(C,\P)$ to $(D,\Q)$ is the choice of a map $\alpha : C \lrar D$ of sets, together with a map in $\Op_C(\SS)$ of the form $\P \lrar \alpha^*\Q$ where $\alpha^*\Q$ is the $C$-colored operad given at each level by
$$ \alpha^{*}\Q(c_1,\cdots,c_n;c) := \Q(\alpha(c_1),\cdots,\alpha(c_n);\alpha(c)).$$
\end{define}

\begin{rem}\label{r:opCop} For any $\P\in\Op_C(\SS)$, there is an adjunction  between \textit{under categories}
	$$ \adjunction*{}{\Op_C(\SS)_{\P/}}{\Op(\SS)_{\P/}}{} $$
in which the left adjoint is the canonical embedding, and the right adjoint is induced by the restriction of colors, as described in Definition \ref{d:op}. 
\end{rem}

\begin{rem}\label{r:catop} We denote by $\Cat(\cS)$ the \textbf{category of (small) categories enriched in $\cS$}. There is an obvious embedding $\iota : \Cat(\cS) \lrar \Op(\SS)$. Conversely, for an operad $\P \in \Op(\SS)$, we denote by $\P_1 \in \Cat(\cS)$ the category whose objects are the colors of $\P$ and whose mapping spaces are given by the collection of \textit{unary} (i.e., $1$-ary) operations. We refer to $\P_1$ as the \textbf{underlying category} of $\P$.  These constructions determine an adjunction
	$$ \iota : \adjunction*{}{\Cat(\cS)}{\Op(\SS)}{} : (-)_1.$$
\end{rem}

\smallskip

Now let $\P\in\Op_C(\SS)$ be a fixed $C$-colored operad. Various types of modules over $\P$ arise naturally by viewing $\P$ as a monoid object.

\begin{dfn} \label{d:opmod}
		(i) A \textbf{left $\mathcal{P}$-module} is a $C$-collection $M$ equipped with a map $\mathcal{P}\circ M\lrar M$, specified by a family of $\Sigma_*$-equivariant maps of the form
		\begin{align*}
			&\circ^\ell : \P(c_1,\cdots,c_n;c) \otimes M(d_{1,1},\cdots,d_{1,k_1};c_1) \otimes \cdots \otimes M(d_{n,1},\cdots,d_{n,k_n};c_n) \\
			&\phantom{\circ^\ell : \P(c_1,\cdots,c_n;c) \otimes M(d_{1,1},\cdots,d_{1,k_1};c_1)  {}} \lrar M(d_{1,1},\cdots,d_{1,k_1},\cdots,d_{n,1},\cdots,d_{n,k_n};c).
		\end{align*}
         Moreover, these maps are required to satisfy the associativity and unitality axioms for left modules.

		(ii) A \textbf{right $\mathcal{P}$-module} is a $C$-collection $M$ equipped with a map $M \circ \P \lrar M$, specified by a family of $\Sigma_*$-equivariant maps of the form
			\begin{align*}
			&\circ^{\sr} : M(c_1,\cdots,c_n;c) \otimes \P(d_{1,1},\cdots,d_{1,k_1};c_1) \otimes \cdots \otimes \P(d_{n,1},\cdots,d_{n,k_n};c_n) \\
			&\phantom{\circ^{\sr} : M(c_1,\cdots,c_n;c) \otimes \P(d_{1,1},\cdots,d_{1,k_1};c_1)  {}} \lrar M(d_{1,1},\cdots,d_{1,k_1},\cdots,d_{n,1},\cdots,d_{n,k_n};c)
		\end{align*}
			satisfying the axioms of associativity and unitality for right modules.
		
		(iii) A $\P$\textbf{-bimodule} is a $C$-collection $M$ equipped with both a left and a right $\mathcal{P}$-module structure, satisfying the compatibility axiom between the two actions.
\end{dfn}  

We denote by $\LMod(\P)$, $\RMod(\P)$, and $\BMod(\P)$ the categories of \textbf{left} $\P$-\textbf{modules}, \textbf{right} $\P$-\textbf{modules}, and $\P$-\textbf{bimodules}, respectively.

\begin{rem}\label{r:biopC} As usual, a map $\P \lrar \Q$ in $\Op_C(\SS)$ endows $\Q$ with the structure of a $\P$-bimodule (under $\P$). We thus obtain a \textit{restriction functor} $\Op_C(\SS)_{\P/} \lrar \BMod(\P)_{\P/}$. In practice, this typically arises as part of an adjunction of \textit{induction-restriction functors}:
	$$ \adjunction*{}{\BMod(\P)_{\P/}}{\Op_C(\SS)_{\P/}}{}.$$
\end{rem}

Next, a $\P$-\textit{algebra} is simply a left $\P$-module concentrated in level $0$. It is described explicitly as follows.

\begin{define}\label{d:alg} A $\P$-\textbf{algebra} $A$ consists of a collection $\{A(c)\}_{c\in C}$ of objects in $\cS$ equipped with action maps of the form
		$$ \P(c_1,\cdots,c_n;c)\otimes A(c_1)\otimes\cdots\otimes A(c_n) \lrar A(c)$$
which satisfy the usual  axioms of associativity, unitality, and equivariance. We denote by $\Alg_\P(\SS)$ the \textbf{category of $\P$-algebras}.
\end{define}

Furthermore, the notion of $A$-\textit{modules over} $\P$ provides \textit{representations} for the $\P$-algebra $A$.

\begin{define}\label{d:opmodule} An $A$\textbf{-module} (\textbf{over} $\P$) is a collection  $M=\{M(c)\}_{c\in C}$ of objects in $\cS$ endowed, for each $C$-sequence $(c_1,\cdots,c_n;c)$ and $k\in\{1,\cdots,n\}$, with an action map of the form
	$$  \P(c_1,\cdots,c_n;c) \otimes \;  \bigotimes _{i \in \{1,\cdots,n\} \setminus \{k\}} \; A(c_i) \otimes M(c_k) \lrar M(c).$$
	These maps must satisfy the usual associativity, unitality, and equivariance axioms. We let $\Mod_\P^{A}$ denote the \textbf{category of $A$-modules}.
\end{define}

\begin{rem}\label{r:modAalg} A map $A \lrar B$ of $\P$-algebras endows $B$ with the structure of an $A$-module (under $A$). This procedure yields a \textit{restriction functor} $\Alg_\P(\SS)_{A/}\lrar (\Mod_\P^{A})_{A/}$. Again, this usually arises as part of an adjunction of \textit{induction-restriction functors}:
	$$ \adjunction*{}{(\Mod_\P^{A})_{A/}}{\Alg_\P(\SS)_{A/}}{}.$$
\end{rem}

\subsection{Infinitesimal bimodules over an operad}\label{s:opinfi}

Throughout this subsection, we assume that $(\SS, \otimes, 1_\cS)$ is a \textbf{closed symmetric monoidal category} such that the underlying category $\cS$ is bicomplete. Let $\P\in\Op_C(\cS)$ be a $C$-colored operad in $\cS$, for some set of colors $C$. 

\textit{Operadic infinitesimal bimodules}, as introduced by Merkulov-Vallette \cite{Vallette}, will be of particular interest in this paper. To begin with, we write ``$\circ_{(1)}$'' for the \textbf{infinitesimal composite product} $$ - \circ_{(1)} - : \Coll_C(\cS) \times \Coll_C(\cS) \lrar \Coll_C(\cS)$$ 
that is characterized as the \textit{right linearization} of the usual composite product $- \circ -$ (see \cite{Loday, Vallette}). In particular, $- \circ_{(1)} -$ distributes over colimits.

\begin{define}\label{d:infbi} (i) An \textbf{infinitesimal left $\P$-module} is a $C$-collection $M$ equipped with an action map $\mathcal{P}\circ_{(1)}M \lrar M$ whose data consist of  $\Sigma_*$-equivariant maps of the form
	$$ \circ^{i\ell} : \P(c_1,\cdots,c_n;c) \otimes M(d_1,\cdots,d_m;c_i) \lrar M(c_1,\cdots,c_{i-1},d_1,\cdots,d_m,c_{i+1},\cdots,c_n;c),$$
subject to the usual associativity and unitality axioms for left modules.

(ii) An \textbf{infinitesimal} $\P$-\textbf{bimodule} is a $C$-collection $M$ equipped with both an infinitesimal left $\P$-module and a right $\P$-module structure, which are subject to the usual compatibility axiom. (See \cite{Hoang, Loday, Vallette} for more details.)
\end{define}

\begin{rem}\label{r:infright} Giving $M\in\Coll_C(\cS)$ the structure of a right $\P$-module is equivalent to endowing it with an \textbf{infinitesimal right $\P$-module} structure. Explicitly, the latter is given  by a map $M \circ_{(1)} \P \lrar M$ whose data consist of $\Sigma_*$-equivariant maps of the form
	$$ \circ^{\ir} : M(c_1,\cdots,c_n;c) \otimes \P(d_1,\cdots,d_m;c_j) \lrar M(c_1,\cdots,c_{j-1},d_1,\cdots,d_m,c_{j+1},\cdots,c_n;c).$$
\end{rem}

We denote by $\ILMod(\P)$ (resp. $\IbMod(\P)$) the \textbf{category of infinitesimal left $\P$-modules} (resp. \textbf{infinitesimal $\P$-bimodules}).

\begin{notn}\label{no:freeib} We will write 
	\begin{gather*} \Free_\P^{i\ell} : \Coll_C(\cS) \lrar \ILMod(\P),  \; \text{and} \\
		\Free_\P^{\si} : \Coll_C(\cS) \lrar \IbMod(\P)
	\end{gather*}
	for the \textit{free functors} that are left adjoints to the associated forgetful functors.
\end{notn}

\begin{rem}\label{r:inffr} As in \cite[$\S$2.2]{Hoang}, for each $C$-collection $M$ we have 
	$$ \Free_\P^{i\ell}(M) \cong \P\circ_{(1)}M \;\; \text{and}  \;\; \Free_\P^{\si}(M) \cong \mathcal{P}\circ_{(1)}( M \circ \mathcal{P}).$$
 In particular, when $C = \{*\}$  and $M=\mathcal{E}_*$ (see Notation \ref{no:IE}), we obtain  $$ \Free_\P^{\si}(\mathcal{E}_*) \cong  \mathcal{P}\circ_{(1)}(\mathcal{E}_*\circ\P) \cong \mathcal{P}\circ_{(1)}\mathcal{E}_* \cong \Free_\P^{i\ell}(\mathcal{E}_*).$$
Concretely, this is given at each level by $(\mathcal{P}\circ_{(1)}\mathcal{E}_*)(n)= \mathcal{P}(n+1)$ with the infinitesimal $\P$-bimodule structure canonically  induced by the composition in $\P$.
\end{rem}

\begin{rem}\label{r:infbibi} A map $\P \lrar M$ in $\BMod(\P)$ canonically  endows $M$ with the structure of an infinitesimal $\P$-bimodule (under $\P$). We thus obtain a \textit{restriction functor} $\BMod(\P)_{\P/} \lrar \IbMod(\P)_{\P/}$. Moreover, one typically has an adjunction of \textit{induction-restriction functors}:
	$$ \adjunction*{}{\BMod(\P)_{\P/}}{\IbMod(\P)_{\P/}}{}.$$
\end{rem}

Each of the categories $\IbMod(\P)$, $\RMod(\P)$ and $\ILMod(\P)$  can be represented as a category of $\cS$-valued enriched functors  (cf. \cite{Hoang1}). Here we recall the construction for $\IbMod(\P)$. 

\begin{notn} We let $\Fin_*$ denote the skeleton of the \textbf{category of finite pointed sets}. Its objects are the pointed sets $\left \langle m \right \rangle := \{0,1,\cdots,m\}$ for  $m\geqslant0$, with $0$ as the base point.
\end{notn}

\begin{cons} \label{infbimod} The category encoding infinitesimal $\P$-bimodules is denoted by $\textbf{Ib}^{\mathcal{P}}$. Its objects are precisely the $C$-sequences, and the mapping spaces are defined as follows. First, for each map $f : \l m \r \lrar \l n \r$ in $\Fin_*$, we define
	$$  \Map^f_{\textbf{Ib}^{\mathcal{P}}}((c_1,\cdots,c_n;c) , (d_1,\cdots,d_m;d)) := \mathcal{P}\left (c,\{d_j\}_{j\in f^{-1}(0)};d \right ) \otimes \bigotimes_{i=1,\cdots,n} \mathcal{P} \left (\{d_j\}_{j\in f^{-1}(i)};c_i \right ).$$
	Then we define
	$$ \Map_{\textbf{Ib}^{\mathcal{P}}}((c_1,\cdots,c_n;c) , (d_1,\cdots,d_m;d)) := \bigsqcup_{\left \langle m \right \rangle  \overset{f }{\rar} \left \langle n \right \rangle} \Map^f_{\textbf{Ib}^{\mathcal{P}}}((c_1,\cdots,c_n;c) , (d_1,\cdots,d_m;d) ) $$
	where the coproduct ranges over $\Hom_{\Fin_*}(\l m \r, \l n \r)$. The unit morphisms of $\textbf{Ib}^{\mathcal{P}}$ are defined via the unit operations of $\mathcal{P}$, and moreover, the categorical structure maps are canonically induced by the composition in $\mathcal{P}$.
\end{cons}

\begin{prop}\textup{(\cite{Hoang}, $\S$2.2)}\label{p:ibfunc} There is a categorical isomorphism
	$$ \Fun(\textbf{Ib}^{\mathcal{P}}, \cS) \cong \IbMod(\P)$$
	between $\cS$-enriched functors on $\textbf{Ib}^{\mathcal{P}}$ and infinitesimal $\P$-bimodules.
\end{prop}

\begin{notn}\label{no:csing} When $C$ is a singleton, the objects of $\textbf{Ib}^{\mathcal{P}}$ can be identified with the non-negative integers, denoted $\{\underline{0},\underline{1},\underline{2},\cdots\}$. For each map $f : \l m \r \lrar \l n \r$ in $\Fin_*$, we have
	$$ \Map^f_{\textbf{Ib}^{\mathcal{P}}}(\underline{n},\underline{m}) = \P(k_{f,0})\otimes\P(k_{f,1})\otimes\cdots\otimes\P(k_{f,n}) $$
	in which $k_{f,i} := |f^{-1}(i)|$ for $i=0,\cdots,n$.
\end{notn}

We now describe some examples, with a view toward $\S$$\S$\ref{s:Qdg}-\ref{s:Qprin}, assuming henceforth that $C = \{*\}$.

\begin{notn} We denote by $*_\cS$ the $\cS$-enriched category that has a single object, with mapping object given by $1_\cS$. 
\end{notn}

\begin{example}\label{ex:inffree1}  We will write $$ \underline{0}_![1_\cS] : \textbf{Ib}^{\mathcal{P}} \lrar \cS$$ for the left Kan extension of the functor $*_\cS \x{\{1_\cS\}}{\lrar} \cS$ along the embedding $*_\cS \x{ \{\underline{0}\} }{\lrar} \textbf{Ib}^{\mathcal{P}}$. Clearly, under the identification $\Fun(\textbf{Ib}^{\mathcal{P}}, \cS) \cong \IbMod(\P)$, the functor $ \underline{0}_![1_\cS]$  is identified with  $\Free_\P^{\si}(\mathcal{E}_*)$ (see Remark \ref{r:inffr}).
\end{example}

\begin{example}\label{ex:rightkan} For an object $X\in\cS$, we let $\underline{1}_*[X] : \textbf{Ib}^{\mathcal{P}} \lrar \cS$ denote the right Kan extension of the functor $*_\cS \x{\{X\}}{\lrar} \cS$ along $*_\cS \x{\{\underline{1}\}}{\lrar} \textbf{Ib}^{\mathcal{P}}$. This is given on objects by 
	$$  \underline{1}_*[X](\underline{m}) \cong \Map_{\cS}(\Map_{\textbf{Ib}^{\mathcal{P}}}(\underline{m},\underline{1}), X).$$
	We write $\Hom_{\Fin_*}(\l 1 \r,\l m \r) = \{\rho_i \, | \, 0 \leq i\leq m\}$ where $\rho_i : \l 1 \r\lrar\l m \r$ is defined by  $\rho_i(1)=i$, and write
	$$  \underline{1}_*[X](\underline{m}) \; \cong \; \bigsqcap_{i=0}^m T_m^i $$
in which $T_m^i := \Map_{\cS}(\Map^{\rho_i}_{\textbf{Ib}^{\mathcal{P}}}(\underline{m},\underline{1}), X)$. Finally, notice that $$\Map^{\rho_0}_{\textbf{Ib}^{\P}}(\underline{m},\underline{1}) \cong \P(2)\otimes \P(0)^{\otimes \, m} \;\; \text{and}  \;\; \Map^{\rho_i}_{\textbf{Ib}^{\P}}(\underline{m},\underline{1}) \cong \P(1)^{\otimes \, 2} \otimes \P(0)^{\otimes \, (m-1)} $$
	for every $1\leq i \leq m$.
\end{example}

\begin{example}\label{ex:rightkan1} We also write $\underline{0}_*[X] : \textbf{Ib}^{\mathcal{P}} \lrar \cS$ for the right Kan extension of $*_\cS \x{\{X\}}{\lrar} \cS$ along the embedding $*_\cS \x{\{\underline{0}\}}{\lrar} \textbf{Ib}^{\mathcal{P}}$. This is given on objects by 
	$$  \underline{0}_*[X](\underline{m}) \cong \Map_{\cS}(\Map_{\textbf{Ib}^{\mathcal{P}}}(\underline{m},\underline{0}), X) \cong \Map_{\cS}(\mathcal{P}(1)\otimes\P(0)^{\otimes \, m},X).$$
	In particular, if  $\P(0) \cong \P(1) \cong 1_\cS$, then $\underline{0}_*[X](\underline{m}) \cong X$ for every $m\geq0$, and hence $\underline{0}_*[X]$ may be regarded   as the \textit{constant functor} with value $X$.
\end{example}

\begin{example}\label{ex:tfunctor} Observe that $\textbf{Ib}^{\Com}$ is isomorphic to (the $\cS$-enriched version of) $\Fin_*^{\op}$. Therefore, each object of $\IbMod(\Com)$ can be represented as a functor $\Fin_*^{\op} \lrar \cS$. In the case $\cS = \C(\textbf{k})$ (see $\S$\ref{s:coeinfty}), the \textbf{Pirashvili functor} $t : \Fin_*^{\op} \lrar \C(\textbf{k})$ is defined on objects by $$t(\l m \r) := [\l m \r, \textbf{k}]_*$$ i.e., the $\textbf{k}$-module of based maps $\l m \r \lrar \textbf{k}$, where $\textbf{k}$ has base point $0_\textbf{k}$ (see \cite{Pirash}). As we will see in $\S$\ref{s:cotandg}, up to a shift, this  functor $t$ models the cotangent complex of the \textbf{dg $\E_\infty$-operad}. 
\end{example}

\subsection{Endomorphism constructions}\label{s:endo}

In this subsection, we assume that $(\SS, \otimes, 1_\cS)$ is a \textbf{closed symmetric monoidal category} whose underlying category $\cS$ admits coequalizers. Let $\P$ be a $C$-colored operad in $\cS$ for some set of colors $C$, and let $A$ be a $\P$-algebra. 

We will discuss various \textit{endomorphism constructions} that naturally relate $\P$-algebras (resp. $A$-modules) to $\P$-bimodules (resp. infinitesimal $\P$-bimodules).

For each pair $(M,N) \in \RMod(\P) \times \LMod(\P)$, the \textbf{relative composite product} $M\circ_\P N \in \Coll_C(\cS)$ is determined by the coequalizer  diagram of $C$-collections: $$M\circ \P \circ N \doublerightarrow{}{} M\circ N \lrar M\circ_\P N.$$
In particular, taking relative composite product with $A$ (considered as a left $\P$-module concentrated in level $0$) yields a functor
$$ (-)\circ_\P A : \RMod(\P) \lrar  \cS^{\times C}.$$
Moreover, this functor descends to the functors
\begin{gather*} (-)\circ_\P A : \BMod(\P) \lrar \Alg_{\P}(\cS),  \; \text{and} \\
	(-)\circ_\P A : \IbMod(\P) \lrar \Mod_\P^A.
\end{gather*}
We now construct the corresponding right adjoints to these functors.

\begin{cons} For a pair $(X,Y)$ of objects in $\cS^{\times C}$, the \textbf{endomorphism object} $\End_{X,Y}$ is  the $C$-collection defined by setting
	$$ \End_{X,Y}(c_1,\cdots,c_n;c) := \Map_{\cS}(X(c_1)\otimes\cdots\otimes X(c_n),Y(c)).$$ 
	The right $\Sigma_n$-action arises naturally by permuting the factors $X(c_i)$. In particular, the \textbf{endomorphism operad} associated to $X$ is given by $\End_X:=\End_{X,X}$. For more details, we refer the reader to \cite{Fresse1, Rezk}.
\end{cons}

\begin{prop}\label{p:ends} The constructions $(-)\circ_\P A$ and  $\End_{A,-}$ determine the adjunctions:	
	
\onehalfspacing

	\;  \textup{(i)} $(-)\circ_\P A : \adjunction*{}{\RMod(\P)}{\cS^{\times C}}{} : \End_{A,-} , $
	
	\onehalfspacing	
	
	\;  \textup{(ii)} $(-)\circ_\P A : \adjunction*{}{\BMod(\P)}{\Alg_\P(\cS)}{} : \End_{A,-} , \;\; \text{and} $	
	
	\onehalfspacing
	
	\;  \textup{(iii)} $	(-)\circ_\P A : \adjunction*{}{\IbMod(\P)}{\Mod_\P^A}{} : \End_{A,-}.$	
	\begin{proof} The first two adjunctions were established in \cite{Fresse1, Rezk}, while the other is included in \cite{Hoang1}.
	\end{proof}
\end{prop}

\begin{rem}\label{r:prekey} Due to the identification $\P\circ_\P A \cong A$, we may combine the above adjunctions into a commutative diagram of adjunctions
	\begin{equation}\label{eq:prekey}
		\begin{tikzcd}[row sep=3.5em, column sep =3.5em]
			\RMod(\P)_{\P/}	 \arrow[r, shift left=1.5] \arrow[d, shift right=1.5, "(-)\circ_\P A"'] & \IbMod(\P)_{\P/} \arrow[l, shift left=1.5, "\perp"'] \arrow[d, shift left=1.5, "(-)\circ_\P A"] \arrow[r, shift left=1.5]  & \BMod(\P)_{\P/}  \arrow[d, shift left=1.5, "(-)\circ_\P A"] \arrow[l, shift left=1.5, "\perp"']   \\
			(\cS^{\times C})_{A/} \arrow[r, shift right=1.5] \arrow[u, shift right=1.5, "\dashv", "\End_{A,-}"'] & (\Mod_\P^A)_{A/}  \arrow[l, shift right=1.5, "\downvdash"] \arrow[u, shift left=1.5, "\vdash"', "\End_{A,-}"] \arrow[r, shift right=1.5] & (\Alg_\P(\cS))_{A/}  \arrow[u, shift left=1.5, "\vdash"', "\End_{A,-}"] \arrow[l, shift right=1.5, "\downvdash"]
		\end{tikzcd}
	\end{equation}
	such that the horizontal pairs are given by the adjunctions of induction-restriction functors (see Remarks \ref{r:modAalg} and \ref{r:infbibi}).
\end{rem}

\begin{rem}\label{r:freeAmod} The \textit{free $A$-module functor}, denoted $ \Free_A : \cS^{\times C} \lrar \Mod_\P^A $, is described as follows. Let $X\in\cS^{\times C}$ be given, which we also regard as a $C$-collection concentrated in level $0$. Combining the adjunction of Proposition \ref{p:ends}(iii) with Remark \ref{r:inffr}, we obtain a chain of isomorphisms in $\Mod_\P^A$:
	$$ \Free_A(X) \cong   \Free_\P^{\si}(X)\circ_\P A  \cong (\mathcal{P}\circ_{(1)}X) \circ_\P A.$$
Moreover, when $C = \{*\}$ and $X=1_\cS$, we obtain $$\Free_A(1_\cS) \cong (\P\circ_{(1)}\mathcal{E}_*)\circ_\P A.$$	
This provides a model for the \textbf{universal enveloping algebra} of $A$. (See also \cite[$\S$5.2]{Hoang1} for more details.)
\end{rem}

\subsection{Tangent categories and Quillen cohomology}\label{s:notcotan}

We briefly outline the procedures for establishing the \textit{cotangent complex formalism} and \textit{Quillen cohomology}. For further details, we refer the reader to \cite{Hoang, Hoang1, YonatanBundle, YonatanCotangent}. 

Let $\textbf{M}$ be a \textbf{semi-model category} (see \cite{Spitzweck, Fresse1}). By definition, $\textbf{M}$ is \textbf{weakly pointed} if it contains a  \textit{weak zero object} $0$ that is both homotopy initial and terminal.

\begin{dfn}\label{d:stable} We say that $\textbf{M}$ is \textbf{stable} if the following equivalent conditions hold:
\begin{enumerate}
	
	\item The \textbf{underlying} $\infty$-\textbf{category} $\textbf{M}_\infty$ is stable in the sense of \cite{Lurieha}.
	
	\item $\textbf{M}$ is weakly pointed, and a square in $\textbf{M}$ is homotopy coCartesian if and only if it is homotopy Cartesian.
	
	\item $\textbf{M}$ is weakly pointed, and the adjunction $$\Sigma:  \adjunction*{}{\Ho(\textbf{M})}{\Ho(\textbf{M})}{} :\Omega$$ of \textit{suspension} and  \textit{desuspension} functors is an adjoint equivalence.
		\end{enumerate}
\end{dfn}

Next, we assume  that $\textbf{M}$ is a weakly pointed semi-model category. For an ($\mathbb{N}\times\mathbb{N} $)-diagram $X\in\textbf{M}^{\mathbb{N}\times\mathbb{N}}$, we will refer to squares of the form
$$ \xymatrix{
	X_{n,n} \ar[r]\ar[d] & X_{n,n+1} \ar[d] \\
	X_{n+1,n} \ar[r] & X_{n+1,n+1} \\
}$$
as \textbf{diagonal squares}.

\begin{dfn}\label{dfnspectrumobj}   An  ($\mathbb{N}\times\mathbb{N} $)-diagram in $\textbf{M}$ is called
	
\begin{enumerate}[(1)]	
	
	\item a \textbf{prespectrum} if all its off-diagonal entries are weak zero objects in $\textbf{M}$,

	\item an \textbf{$\Omega$-spectrum} if it is a prespectrum and all its diagonal squares are homotopy Cartesian; and

	\item a \textbf{suspension spectrum} if it is a prespectrum and all its diagonal squares are homotopy coCartesian.
\end{enumerate}
\end{dfn}

We denote by $\textbf{M}^{\mathbb{N}\times\mathbb{N}}_{\proj}$  the \textit{projective semi-model category} of ($\mathbb{N}\times\mathbb{N} $)-diagrams in $\textbf{M}$.

\begin{dfn} A map $f : X\lrar Y$ in $\textbf{M}^{\mathbb{N}\times\mathbb{N}}$ is called a \textbf{stable equivalence} if for every $\Omega$-spectrum $Z$ the induced map between derived mapping spaces
	$$ \Map^{\der}_{\textbf{M}^{\mathbb{N}\times\mathbb{N}}_{\proj}}(Y,Z) \lrar  \Map^{\der}_{\textbf{M}^{\mathbb{N}\times\mathbb{N}}_{\proj}}(X,Z) $$
	is a weak equivalence. 
\end{dfn}

The procedure for constructing \textit{stable approximations} relies on the following definition.

\begin{dfn}\label{d:stabilization} Let $\textbf{M}$ be a weakly pointed combinatorial semi-model category in which the domains of generating cofibrations are cofibrant. The \textbf{stabilization} $\Sp(\textbf{M})$ is the left Bousfield localization of $\textbf{M}^{\mathbb{N}\times\mathbb{N}}_{\proj}$ with $\Omega$-spectra as the local objects. More explicitly, $\Sp(\textbf{M})$ is a cofibrantly generated semi-model category such that
	
	\begin{enumerate}
	\item weak equivalences are precisely the stable equivalences, and

	\item (generating) cofibrations are the same as those of $\textbf{M}^{\mathbb{N}\times\mathbb{N}}_{\proj}$.
\end{enumerate}

	In particular, $\Sp(\textbf{M})$ is a stable semi-model category in which the fibrant objects are precisely the levelwise fibrant $\Omega$-spectra.
\end{dfn}

\begin{rem} There is a Quillen adjunction $$\Sigma^{\infty} : \adjunction*{}{\textbf{M}}{\Sp(\textbf{M})}{} : \Omega^{\infty}$$ in which $\Omega^{\infty}(X)=X_{0,0}$ and $\Sigma^{\infty}(A)$ is the constant diagram with value $A \in \textbf{M}$.
\end{rem}

In what follows, we assume that $\textbf{M}$ satisfies the assumptions of Definition \ref{d:stabilization}, and let $A \in \Ob(\textbf{M})$ be a fixed object.

\begin{notn} We write $\textbf{M}_{A//A}$ for the \textbf{category of objects over and under} $A$. Concretely, an object of $\textbf{M}_{A//A}$ is a sequence $A \x{f}{\lrar} B \x{g}{\lrar} A$ of maps in $\textbf{M}$ such that $g\circ f = \Id_A$.
\end{notn}

We regard $\textbf{M}_{A//A}$ as a (strictly) pointed category, with the zero object given by the sequence of identity maps on $A$. It inherits a semi-model structure transferred from that on $\textbf{M}$.

\begin{dfn}\label{d:tangentcat}  The \textbf{tangent category} to $\textbf{M}$ at $A$ is the stabilization of $\textbf{M}_{A//A}$:
	$$ \mathcal{T}_A\textbf{M}:=\Sp(\textbf{M}_{A//A}).$$ 	
\end{dfn}

Let us denote by $ \Sigma^{\infty}_+ : \adjunction*{}{\textbf{M}_{/A}}{\mathcal{T}_A\textbf{M}}{} : \Omega^{\infty}_+ $ the Quillen adjunction in which, for  $B\in\textbf{M}_{/A}$ and $X \in \mathcal{T}_A\textbf{M}$, we have 
\begin{gather*} \Sigma^{\infty}_+(B) = \Sigma^{\infty}(A\longrightarrow A\sqcup B \longrightarrow A), \; \text{and}  \\
	\Omega^{\infty}_+(X) = (X_{0,0} \lrar A).
\end{gather*}

\begin{dfn}\label{d:cotan} The \textbf{cotangent complex} of $A$ is the derived suspension spectrum of $A\in\textbf{M}_{/A}$:
	$$\rL_A:=\mathbb{L}\Sigma^{\infty}_+(A) \in  \mathcal{T}_A\textbf{M}.$$
More generally, for a map $f : A \lrar B$ in $\textbf{M}$, the \textbf{relative cotangent complex} of $f$ is 
$$\rL_{B/A}:=\hocofib \, [\,\mathbb{L}\Sigma^{\infty}_+(f) \lrar \rL_B\,],$$
i.e., the homotopy cofiber of the map $\mathbb{L}\Sigma^{\infty}_+(f) \lrar \rL_B$ in $\mathcal{T}_B\textbf{M}$. 
\end{dfn}

\begin{rem}\label{r:relcotan} The map $\mathbb{L}\Sigma^{\infty}_+(f) \lrar \rL_B$ is weakly equivalent to the canonical map $f_!(\rL_A) \lrar \rL_B$, where $f_!$ refers to the induced left Quillen functor $f_! : \mathcal{T}_A\textbf{M} \lrar \mathcal{T}_B\textbf{M}$ (see \cite[$\S$2.5]{Hoang} for more details).
\end{rem}

We are now ready to state the main definition.

\begin{dfn}\label{d:Qcohom} Suppose that $A \in \textbf{M}$ is a fibrant object. The \textbf{Quillen cohomology} of $A$ with coefficients in an object $M\in\mathcal{T}_A\textbf{M}$ is the space
	$$ \HHQ^\star(A ; M) :=  \Map^{\der}_{\mathcal{T}_A\textbf{M}} (\rL_A , M).$$
Furthermore, the \textbf{$n$-th Quillen cohomology group} is
	$$ \HHQ^n(A  ; M) := \pi_0 \Map^{\der}_{\mathcal{T}_A\textbf{M}} (\rL_A , M[n])$$
	where $M[n]$ signifies the $n$-suspension of $M$ in $\mathcal{T}_A\textbf{M}$.
\end{dfn}

\begin{rem}\label{r:modelQ} Let $A \in \textbf{M}$ be fibrant. By the Quillen adjunction $\Sigma^{\infty}_+ : \adjunction*{}{\textbf{M}_{/A}}{\mathcal{T}_A\textbf{M}}{} : \Omega^{\infty}_+$, for every $M\in\mathcal{T}_A\textbf{M}$, there is a canonical weak equivalence
	$$ \Map^{\der}_{\textbf{M}_{/A}} (A , \RR\Omega^\infty_+ M) \simeq \Map^{\der}_{\mathcal{T}_A\textbf{M}} (\rL_A , M) \overset{\defi}{=} \HHQ^\star(A ; M).$$
\end{rem}

\subsection{Some base categories and operadic model structures}\label{s:bases}

We will assume that $(\cS, \otimes, 1_\cS)$ is a \textbf{symmetric monoidal model category} such that the model structure is cofibrantly generated and the monoidal unit $1_\cS$ is cofibrant.

 Let $\P$ be a $C$-colored operad in $\cS$ for some set of colors $C$, and let $A$ be a $\P$-algebra. We say that $\P$ is $\Sigma$-\textbf{cofibrant} if it is cofibrant in $\Coll_C(\cS)$ with respect to the (\textit{projective}) \textit{model structure} transferred from $\cS$.

\begin{define}\label{d:trans} The \textbf{transferred model (or semi-model) structure} on $\Alg_{\P}(\cS)$, if it exists, is the one transferred along the free-forgetful adjunction $\adjunction*{}{\cS^{\times C}}{\Alg_{\P}(\cS)}{}$. More concretely, it is a cofibrantly generated (semi-)model category in which the weak equivalences (resp. fibrations) are precisely the \textit{levelwise weak equivalences} (resp. \textit{levelwise fibrations}).
\end{define}

\begin{rem}\label{r:AA} We denote by 
	\begin{equation}\label{eq:AA}
		\mathbb{A} := \{\Op_C(\SS), \, \LMod(\P), \, \RMod(\P), \, \BMod(\P), \, \IbMod(\P), \, \Alg_{\P}(\cS) , \, \Mod_{\P}^A \}.
	\end{equation}
Each category in $\mathbb{A}$ can be represented as the category of algebras over an operad or $\cS$-valued functors on an enriched category (see \cite{Hoang} for a detailed discussion). According to \cite{Fresse1, Spitzweck},  the transferred semi-model structure on $\Op_C(\SS)$ exists. The same statement holds for the categories $\LMod(\P), \BMod(\P)$ and $\Alg_{\P}(\cS)$, provided that $\P$ is $\Sigma$-cofibrant. Meanwhile, $\RMod(\P)$ and $\IbMod(\P)$ admit a transferred (full) model structure when $\P$ is levelwise cofibrant. Moreover, the transferred model structure on  $\Mod_{\P}^A$ exists, as long as $\P$ is $\Sigma$-cofibrant and $A \in \Alg_{\P}(\cS)$ is cofibrant.          
\end{rem}

Next, we discuss the homotopy theory of $\cS$-enriched operads.
\begin{define}\label{d:levweak} A map $f : \P \lrar \Q$ in $\Op(\cS)$ is called a \textbf{levelwise weak equivalence} if for every sequence $(c_1,\cdots,c_n;c)$ of colors of $\P$, the induced  map $$\P(c_1,\cdots,c_n;c) \lrar \Q(f(c_1),\cdots,f(c_n);f(c))$$
	is a weak equivalence in $\cS$. Moreover, \textit{levelwise (co)fibrations} and \textit{levelwise trivial (co)fibrations} are defined analogously.
\end{define}

The \textbf{homotopy category} $\Ho(\P)$ is the (ordinary) category whose objects are the colors of $\P$,  and whose morphisms are given by
$$ \Hom_{\Ho(\P)}(c,d) := \Hom_{\Ho(\cS)}(1_\cS, \P(c;d)).$$

\begin{define}\label{d:dkequi} 
	\begin{enumerate}[(1)]
		
	\item A map $f : \P \lrar \Q$ in $\Op(\cS)$ is a \textbf{Dwyer-Kan equivalence} if it is a levelwise weak equivalence, and the induced functor $$\Ho(f) : \Ho(\P) \lrar \Ho(\Q)$$ between homotopy categories is essentially surjective.
	
	\item A map in $\Op(\cS)$ is a \textbf{Dwyer-Kan trivial fibration} if it is a levelwise trivial fibration and surjective on colors.
	 	\end{enumerate}
\end{define}

\begin{rem}\label{r:dk} In \cite{Caviglia}, Caviglia introduced the notion of the \textbf{Dwyer-Kan model structure} on $\Op(\cS)$, whose weak equivalences (resp. trivial fibrations) are precisely the Dwyer-Kan equivalences (resp. Dwyer-Kan trivial fibrations).  
\end{rem}

\begin{example}\label{ex:basecat} Here we introduce some base categories that will serve as the setting for our work. Let $\textbf{k}$ be a commutative ring.
	\begin{enumerate}[(1)]
		\item We endow the category $\Set_\Delta$ of \textbf{simplicial sets} with the Cartesian monoidal structure and the \textit{classical (Kan-Quillen) model structure}. 
		
		\item We denote by $\sMod(\textbf{k})$ the category of \textbf{simplicial} $\textbf{k}$-\textbf{modules}, equipped with the (degreewise) tensor product over $\textbf{k}$ and the model structure transferred from the classical one on $\Set_\Delta$.
		
		\item We let $\C(\textbf{k})$ (resp. $\C_{\geqslant0}(\textbf{k})$)  denote the category of \textbf{dg $\textbf{k}$-modules} (resp. \textbf{connective dg $\textbf{k}$-modules}, i.e., those  concentrated in non-negative degrees). Both categories are equipped with the usual tensor product of chain complexes and the \textit{projective model structure}, in which the weak equivalences are precisely the \textit{quasi-isomorphisms}, and the fibrations in $\C(\textbf{k})$ (resp. $\C_{\geqslant 0}(\textbf{k})$) are the maps that are degreewise surjective (resp. surjective in each positive degree).   
		
		\item Let $R$ be a commutative monoid in $\C_{\geqslant 0}(\textbf{k})$. We are also interested in the category $\Mod_R$ of $R$-\textbf{modules in} $\C_{\geqslant 0}(\textbf{k})$. This category is equipped with the usual tensor product $-\otimes_R-$ and the projective model structure transferred from that on $\C_{\geqslant 0}(\textbf{k})$.  
	\end{enumerate}
\end{example}

\begin{rem}\label{r:modelsimp} For $\cS = \Set_\Delta$ or $\sMod(\textbf{k})$, each category in  $\mathbb{A}$ \eqref{eq:AA} admits a transferred (full) model structure for any $\P\in\Op_C(\SS)$ and $A\in\Alg_{\P}(\cS)$. This follows, for instance, from the results of \cite{Pavlov}, together  with Remark \ref{r:AA}. Moreover, in this setting, the Dwyer-Kan model structure on $\Op(\cS)$ exists as well (see \cite{Caviglia}).
\end{rem}

\begin{rem}\label{r:modelchain} For $\cS$ as in Example \ref{ex:basecat} (3-4),  the transferred homotopy theory on each category in $\mathbb{A}$ exists (at least) as a semi-model category, under suitable assumptions as described in Remark \ref{r:AA}. Furthermore, there exists a cofibrantly generated semi-model structure on $\Op(\cS)$ whose weak equivalences and generating (trivial) cofibrations are those described in \cite[Theorem 8.6]{Caviglia}.  We refer to this as the \textit{Dwyer-Kan semi-model structure on} $\Op(\cS)$. Additionally, all of these exist as (full) model categories, provided that $\textbf{k}$ contains the field $\QQ$ of rational numbers.
\end{rem}

\subsection{More notations}\label{s:other}

 Let $(\cS, \otimes, 1_\cS)$ be a suitable symmetric monoidal model category. Let $\P$ be a $C$-colored operad in $\cS$, and let $A$ be a $\P$-algebra. In addition to those already introduced, we will use the following notations.

\begin{enumerate}

	\item For an object $X \in \cS$, we  denote by $|X| := \Map^{\h}_\cS(1_\cS,X)$ the derived mapping space from $1_\cS$ to $X$, and refer to it the \textbf{underlying space of $X$}.
	
	\item For an integer $n\geq0$, we  denote by $\sS^n := \Sigma^n(1_{\cS} \x{\h}{\sqcup} 1_{\cS})$, i.e., the $n$-suspension of the homotopy coproduct $1_{\cS} \x{\h}{\sqcup} 1_{\cS}\in\cS_{1_{\cS}//1_{\cS}}$, and refer to $\sS^n$ as the \textbf{(pointed) $n$-sphere in $\cS$}.

	\item We will write $\P^{\sB}$ and $\P^{\si}$ to denote $\P$ itself regarded as an object of $\BMod(\P)$ and $\IbMod(\P)$, respectively. 

 \item We will use the following notations for various cotangent complexes related to $\P$ and $A$:

\smallskip
	
 $\bullet$ $\rL_\P \in \T_\P\Op(\cS)$: the cotangent complex of $\P$ considered as an object of $\Op(\cS)$.
 
 \smallskip

$\bullet$  $\rL^{\red}_\P \in \T_\P\Op_C(\cS)$: the cotangent complex of $\P$ considered as an object of $\Op_C(\cS)$.

\smallskip

$\bullet$ $\rL_A \in \T_A\Alg_\P(\cS)$: the cotangent complex of $A$ considered as a $\P$-algebra.

\item We denote by
$$ \HHQ^\star(\P; -) \;\; \text{and} \;\; \HHQ^\star_{\red}(\P; -)$$
the Quillen cohomology of $\P$ regarded as an object in $\Op(\cS)$ and $\Op_C(\cS)$, respectively. The former is referred to as the \textbf{(proper) Quillen cohomology of}  $\P$ (classified by $\rL_\P$), and the latter as its \textbf{reduced Quillen cohomology} (classified by the \textit{reduced cotangent complex} $\rL^{\red}_\P$).

\smallskip

Moreover, we denote by $\HHQ^\star(A; -)$ the \textbf{Quillen cohomology of} $A$ as a $\P$-algebra.

\end{enumerate}

\section{Quillen cohomology of dg operads}\label{s:Qdg}

Let $\textbf{k}$ be a fixed commutative ring. In this section, we work over the base category $\cS = \C(\textbf{k})$ or $\C_{\geq 0}(\textbf{k})$ (cf. Example~\ref{ex:basecat}). The relevant operadic (semi-)model structures are as described in $\S$\ref{s:bases}.

   We will first give an explicit description of the cotangent complex of dg operads. We then show how the cotangent complexes of algebras over a dg operad are controlled by the cotangent complex of the operad itself. In particular, we also consider the case of the \textbf{dg $\E_\infty$-operad}.

\subsection{Cotangent complex of dg operads}\label{s:cotandg}

We begin by recalling the following fundamental theorem. 

\begin{thm}\textup{(\cite{Hoang})}\label{t:keylem} Suppose given a base category $\cS$ and a $C$-colored operad $\P \in \Op_C(\cS)$ for some set of colors $C$. 
\begin{enumerate}[(1)]
	\item Suppose that $\cS$ is \textbf{stably sufficient} (cf. \cite[$\S$3]{Hoang}) and that  $\P$ is $\Sigma$-cofibrant. There is a chain of natural Quillen equivalences
	\begin{equation}\label{eq:optan}
		\adjunction*{\simeq}{\IbMod(\P)}{\T_{\P^{\si}}\IbMod(\P)}{} \adjunction*{\simeq}{}{\T_{\P^{\sB}}\BMod(\P)}{} \adjunction*{\simeq}{}{\T_\P\Op_C(\cS)}{} \adjunction*{\simeq}{}{\T_\P\Op(\cS)}{}
	\end{equation}
	(see $\S$\ref{s:other} for notations). Here, the first adjunction is given by the composed Quillen equivalence
	$$ \adjunction*{(-) \oplus \P^{\si}}{\IbMod(\P)}{\IbMod(\P)_{\P^{\si}//\P^{\si}}}{\ker} \adjunction*{\Sigma^{\infty}}{}{\T_{\P^{\si}}\IbMod(\P)}{\Omega^{\infty}},$$
	while the others are induced by the adjunctions of induction-restriction functors (cf. Remarks \ref{r:opCop}, \ref{r:biopC} and \ref{r:infbibi}).
	
	\item  Suppose further that $\cS$ is \textbf{stably abundant} (cf. \cite[$\S$3]{Hoang}), and that $\P \in \Op_C(\cS)$ is cofibrant and is \textbf{good} in the sense of \cite[$\S$5.1]{Hoang}. Under the equivalence $\T_\P\Op(\cS) \simeq \IbMod(\P)$, the cotangent complex $\rL_\P \in \T_\P\Op(\cS)$ is identified with  $\ovl{\rL}_\P[-1] \in \IbMod(\P)$ in which $\ovl{\rL}_\P$ is given at each level by 
	\begin{equation}\label{eq:keyeq}
		\ovl{\rL}_\P(c_1,\cdots,c_m;c) \simeq \P(c_1,\cdots,c_m;c)\otimes \hocolim_n\Omega^{n} [\, (\sS^{n})^{\otimes \, m}  \times_{\textbf{k}}^{\h}  0\,].
	\end{equation}
	Here $\sS^{n}$ refers to the $n$-sphere in $\cS$ (see $\S$\ref{s:other}(ii)).  
\end{enumerate}
\end{thm}

Now let $\P \in \Op_C(\C(\textbf{k}))$ be a $C$-colored operad in $\C(\textbf{k})$. We outline the steps to deduce Theorem \ref{t:cotan} below.

\begin{define}\label{d:ade}
	An operad in $\C(\textbf{k})$ is said to be \textbf{adequate} if it is $\Sigma$-cofibrant and \textbf{connective}, in the sense that it comes from an operad in $\C_{\geq 0}(\textbf{k})$.
\end{define}

\begin{rem}\label{r:cotan} The statement of Theorem \ref{t:keylem}(2) already holds for adequate dg operads. Indeed, note first that the category $\C(\textbf{k})$ is stably abundant. Suppose that $\P$ is adequate, and consider a cofibrant resolution $\Q \lrarsimeq \P$ for $\P\in\Op_C(\C_{\geqslant0}(\textbf{k}))$. As discussed in \cite[$\S$5.1]{Hoang}, the operad $\Q$ is automatically \textit{good}. Moreover, we have a commutative square of Quillen equivalences	 
	$$\begin{tikzcd}[row sep=3.5em, column sep =3.5em]
		\IbMod(\mathcal{Q}) \arrow[r, shift left=1.5, "\simeq"] \arrow[d, shift right=1.5, "\simeq"'] & \IbMod(\mathcal{P}) \arrow[l, shift left=1.5, "\perp"'] \arrow[d, shift left=1.5, "\simeq"]  & \\
		\T_\Q\Op(\C(\textbf{k})) \arrow[r, shift right=1.5] \arrow[u, shift right=1.5, "\dashv"] & \T_\P\Op(\C(\textbf{k})).  \arrow[l, shift right=1.5, "\downvdash", "\simeq"'] \arrow[u, shift left=1.5, "\vdash"'] &
	\end{tikzcd}$$	
We deduce the claim by observing that the bottom Quillen equivalence identifies $\rL_\Q$ with $\rL_\P$, while the top one identifies $\ovl{\rL}_\Q$ with $\ovl{\rL}_\P$. 
\end{rem}

For a more concrete description of $\ovl{\rL}_\P$, we carry out an easy computation as follows.
\begin{lem}\label{l:cotan} For every $C$-sequence $(c_1,\cdots,c_m;c)$, there is a weak equivalence
	$$ \ovl{\rL}_\P(c_1,\cdots,c_m;c) \simeq \P(c_1,\cdots,c_m;c)^{\oplus \, m}.$$
\begin{proof}  First, we have that
	$$ (\sS^{n})^{\otimes \, m} \; \simeq \; (\textbf{k} \oplus \textbf{k}[n])^{\otimes \, m} \; \cong \; \bigoplus_{i=0}^m \binom{m}{i} \, \textbf{k}[in] ,$$
	from which we further obtain  that
	$$ \Omega^{n} [\, (\sS^{n})^{\otimes m}  \times_{\textbf{k}}^{\h}  0\,] \; \simeq \; \Omega^{n} \left( \bigoplus_{i=1}^m \binom{m}{i} \, \textbf{k}[in] \right) \; \simeq \; \bigoplus_{i=1}^m \binom{m}{i} \, \textbf{k}[(i-1)n].$$
	Thus we may write
	$$ \hocolim_n\Omega^{n} [\, (\sS^{n})^{\otimes m}  \times_{\textbf{k}}^{\h}  0\,] \; \simeq \; \bigoplus_{i=1}^m \binom{m}{i} \, \hocolim_n  \textbf{k}[(i-1)n].$$
	Note that the complex $\hocolim_n \textbf{k}[(i-1)n]$ vanishes for every $i>1$, due to the fact that the homology functor commutes with filtered colimits. For the case $i=1$, there is an evident quasi-isomorphism $\hocolim_n \textbf{k}[(i-1)n] \simeq \textbf{k}$. Accordingly, we obtain
	$$ \hocolim_n\Omega^{n} [\, (\sS^{n})^{\otimes m}  \times_{\textbf{k}}^{\h}  0\,] \; \simeq \; \textbf{k}^{\oplus m},$$
	from which the proof is completed.
\end{proof}
\end{lem}

\begin{rem}\label{r:stlp} The infinitesimal $\P$-bimodule structure on $\ovl{\rL}_\P$ is described as follows.
\begin{enumerate}
	\item As an infinitesimal left $\P$-module, $\ovl{\rL}_\P$ is freely generated by $\I_C$. We may therefore write $$\ovl{\rL}_\P \cong \P\circ_{(1)}\I_C$$ (see Remark~\ref{r:inffr}).

	\item For the (infinitesimal) right $\P$-action, we need to define maps of the form
	$$ \circ^{\ir} : \ovl{\rL}_\P(c_1,\cdots,c_m;c) \otimes \P(d_1,\cdots,d_n;c_j) \lrar \ovl{\rL}_\P(c_1,\cdots,c_{j-1},d_1,\cdots,d_n,c_{j+1},\cdots,c_m;c)$$
 (see Remark \ref{r:infright}). This consists of, for each $i\in\{1,\cdots,m\}$, a component map  
	$$ \circ^{\ir} : \P(c_1,\cdots,c_m;c)^{\{c_i\}} \otimes \P(d_1,\cdots,d_n;c_j) \lrar \ovl{\rL}_\P(c_1,\cdots,c_{j-1},d_1,\cdots,d_n,c_{j+1},\cdots,c_m;c)$$
where $\P(c_1,\cdots,c_m;c)^{\{c_i\}}$ denotes $\P(c_1,\cdots,c_m;c)$ itself as a summand of $\ovl{\rL}_\P(c_1,\cdots,c_m;c)$ corresponding to $c_i$. When $i \neq j$, the latter map is induced by the composition in $\P$: 
$$ \P(c_1,\cdots,c_m;c)^{\{c_i\}} \otimes \P(d_1,\cdots,d_n;c_j) \lrar \P(c_1,\cdots,c_{j-1},d_1,\cdots,d_n,c_{j+1},\cdots,c_m;c)^{\{c_i\}}.$$ 
When $i = j$, the map $\circ^{\ir}$ factors through an operadic composition followed by a diagonal map:
$$ \P(c_1,\cdots,c_m;c)^{\{c_j\}} \otimes \P(d_1,\cdots,d_n;c_j) \lrar  \bigoplus_{k=1}^{n}\P(c_1,\cdots,c_{j-1},d_1,\cdots,d_n,c_{j+1},\cdots,c_m;c)^{\{d_k\}}.$$
\end{enumerate}
\end{rem}

\begin{rem}\label{r:lp} From a set-theoretic perspective, the operation $\circ^{\ir}$ is given as follows. For $\mu\in\P(c_1,\cdots,c_m;c)$ and each $i\in\{1,\cdots,m\}$, we denote by $\mu^{\{c_i\}}\in\ovl{\rL}_\P(c_1,\cdots,c_m;c)$ the element 
$\mu$ regarded as lying in the summand $\P(c_1,\cdots,c_m;c)^{\{c_i\}}$. Then, for $\lambda\in\P(d_1,\cdots,d_n;c_j)$, we define
	$$ \mu^{\{c_i\}} \circ^{\ir} \lambda \; = \; \begin{cases}
		(\mu \circ \lambda)^{\{c_i\}} & \text{ if } \, j \neq i \\ 
		\sum_{k=1}^{n} (\mu \circ \lambda)^{\{d_k\}} & \text{ if } \, j = i \end{cases} $$
where, as usual, $-\circ-$ denotes the (partial) composition of operations. 
\end{rem}

We now state the main result of this section.
\begin{thm}\label{t:cotan} Suppose that $\P \in \Op_C(\C(\textbf{k}))$ is adequate (see Definition~\ref{d:ade}).
\begin{enumerate}
	\item Under the equivalence $\T_\P\Op(\C(\textbf{k})) \simeq \IbMod(\P)$, the cotangent complex $\rL_\P$ is identified with $\ovl{\rL}_\P[-1] \in \IbMod(\P)$, in which $\ovl{\rL}_\P = \P\circ_{(1)}\I_C$ is given at each level by
	$$ \ovl{\rL}_\P(c_1,\cdots,c_m;c) = \P(c_1,\cdots,c_m;c)^{\oplus \, m},$$
	with the infinitesimal $\P$-bimodule structure as described in Remark~\ref{r:stlp}.

	\item Consequently, the Quillen cohomology of $\P$ with coefficients in an object $M\in\IbMod(\P)$ is formulated as
	$$ \HHQ^\star(\P ; M)  \simeq \Map^{\h}_{\IbMod(\P)}(\ovl{\rL}_\P[-1],M),$$
	and the $n$-th Quillen cohomology group is given by
	$$  \HHQ^n(\P ; M)  \cong \pi_0\Map^{\h}_{\IbMod(\P)}(\ovl{\rL}_\P,M[n+1]).$$
\end{enumerate}	 
\end{thm}

In the rest of this subsection, we discuss some key features of the object $\ovl{\rL}_\P \in \IbMod(\P)$.

\begin{rem}\label{r:fres} The object $\ovl{\rL}_\P$ was studied in \cite[$\S$10.3]{Fresse1} and \cite[$\S$7.2]{Joan} as a right $\P$-module governing the \textbf{module of Kähler differentials} of $\P$-algebras. We will revisit this result in Corollary~\ref{co:keyas} below.
\end{rem}

\begin{define}\label{d:der} Let $A$ be a $\P$-algebra, and $N\in\Mod_{\P}^A$ an $A$-module.
	
	\begin{enumerate}[(1)] 
		\item A \textbf{derivation from $A$ to} $N$ is a map $\delta : A \lrar N$ in $\C(\textbf{k})^{\times C}$, given by a collection of maps $\delta = \{A(c) \overset{\delta_c}{\lrar} N(c)\}_{c\in C}$, that satisfies equations of the form
		\begin{equation}\label{eq:der}
			\delta_c(\mu\cdot(\alpha_1,\cdots,\alpha_m)) = \sum_{i=1}^m \mu\cdot(\alpha_1,\cdots,\delta_{c_i}(\alpha_i),\cdots,\alpha_m)
		\end{equation}
		in which $\mu \in \P(c_1,\cdots,c_m;c)$ and $\alpha_i \in A(c_i)$ for $i=1,\cdots,m$ (cf., e.g., \cite{Loday, Fresse1}). We denote by $\Der(A,N)$ the \textbf{set of derivations from $A$ to $N$}.

		\item The \textbf{module of K{\"a}hler differentials of} $A$, denoted by  $\Om_A$, is the $A$-module that represents the functor $N \mapsto \Der(A,N)$. That is, there are natural isomorphisms of the form
		$$ \Hom_{\Mod_{\P}^A}(\Om_A,N)  \cong \Der(A,N).$$
	\end{enumerate}
\end{define}

\begin{rem}\label{r:OmA} Alternatively, $\Om_A$ can be characterized as follows. The coproduct $A\oplus N$ (as objects in $\C(\textbf{k})^{\times C}$) admits the canonical structure of a $\P$-algebra over $A$, called the \textbf{square zero extension of} $A$ \textbf{by} $N$, and denoted by $A \ltimes N$  (see \cite[$\S$12.3.3]{Loday}). Moreover, this construction determines a right adjoint 
	$$ A \ltimes (-) : \Mod_{\P}^A \lrar \Alg_{\P}(\C(\textbf{k}))_{/A}.$$ We denote by $\Om^{/A}$ the corresponding left adjoint. One can then verify that $\Om_A \cong \Om^{/A}(\Id_A)$.
\end{rem}

We introduce the following concept.
\begin{define}\label{d:opder} Let $M \in \IbMod(\P)$ be an infinitesimal $\P$-bimodule. A \textbf{tangent structure on $M$} is a map $\varepsilon : \I_C \lrar M$ of $C$-collections, given by a collection $\varepsilon = \{\varepsilon_c \in M(c;c) \, | \, c\in C \}$ where each $\varepsilon_c$ is of degree $0$,  such that for every $\mu \in \P(c_1,\cdots,c_m;c)$ the following relation holds:
\begin{equation}\label{eq:datumm}
	\varepsilon_c \circ^{\ir} \mu  = \sum_{i=1}^{m}\mu\circ^{i\ell}\varepsilon_{c_i}
\end{equation}
(see Definition \ref{d:infbi} for notation). We denote by $\Tan_\P(M)$ the  \textbf{set of tangent structures on $M$}. 
\end{define}

The following statement illustrates the main interest in this notion.
\begin{pro}\label{p:keyas} Let $A$ be a $\P$-algebra, $N$ an $A$-module, and $M$ an infinitesimal $\P$-bimodule.
	\begin{enumerate}[(1)] 
	\item There is a natural isomorphism of the form
		 $$  \Hom_{\IbMod(\P)}(\ovl{\rL}_\P,M) \cong \Tan_\P(M).$$

		\item There is a natural isomorphism 
		$$ \Der(A,N)  \cong \Tan_\P(\End_{A,N})$$
identifying  derivations from $A$ to $N$ with tangent structures on $\End_{A,N} \in \IbMod(\P)$ (cf. $\S$\ref{s:endo}).   
	\end{enumerate}
\begin{proof} (1) Let $f :  \ovl{\rL}_\P \lrar M$ be a map in $\IbMod(\P)$. Recall that, as an infinitesimal left $\P$-module, $\ovl{\rL}_\P$ is freely generated by $\I_C$. Therefore, regarded as a map in $\ILMod(\P)$, $f$ is determined by a single map $\varepsilon : \I_C \lrar M$ of $C$-collections. As in Definition \ref{d:opder}, we write $\varepsilon = \{\varepsilon_c \in M(c;c) \, | \, c\in C \}$. Accordingly, for $\mu \in \P(c_1,\cdots,c_m;c)$, the map $f$ sends  each $\mu^{\{c_i\}}\in \ovl{\rL}_\P(c_1,\cdots,c_m;c)$ to $$f(\mu^{\{c_i\}}) = \mu \circ^{i\ell} \varepsilon_{c_i} \in M(c_1,\cdots,c_m;c).$$
This rule is compatible with the right $\P$-module structures, and one can verify that this compatibility condition is equivalent to the relation \eqref{eq:datumm}. It follows that giving $f$ amounts to specifying a tangent structure $\varepsilon : \I_C \lrar M$.   
	
(2) By definition, a tangent structure on $\End_{A,N}$ is represented by a map $\delta : \I_C \lrar \End_{A,N}$, which is equivalent to a collection $\delta = \{A(c) \x{\delta_c}{\lrar} N(c)\}_{c\in C}$ of maps in $\C(\textbf{k})$. Moreover, these are required to satisfy the relation \eqref{eq:datumm}:
$$ 	\delta_c \circ^{\ir} \mu  = \sum_{i=1}^{m}\mu\circ^{i\ell}\delta_{c_i}  $$
for every $\mu \in \P(c_1,\cdots,c_m;c)$.  This requirement is equivalent to the condition that $\delta$ is a derivation $A \lrar N$ \eqref{eq:der}.  
\end{proof}
\end{pro}

We give a more conceptual proof of the following result.
\begin{cor}\label{co:keyas}\textup{(Fresse \cite{Fresse1}, Millès \cite{Joan})} Let $A$ be a $\P$-algebra. There is a natural isomorphism of $A$-modules:
\begin{equation}\label{eq:egovv}
	\ovl{\rL}_\P\circ_\P A  \cong \Om_A
\end{equation}
(see $\S$\ref{s:endo} for notation).
\begin{proof}  For an $A$-module $N$, we have a chain of natural isomorphisms
	$$ \Hom_{\Mod_{\P}^A}(\ovl{\rL}_\P\circ_\P A,N) \, \cong \, \Hom_{\IbMod(\P)}(\ovl{\rL}_\P,\End_{A,N}) \, \cong \, \Der(A,N)$$
in which the first isomorphism is due to the adjunction of Proposition~\ref{p:ends}(iii), and the second follows by combining the two parts of Proposition \ref{p:keyas}. The assertion then follows directly from the definition of $\Om_A$.
\end{proof}
\end{cor}

\begin{rem}\label{r:control} The identification \eqref{eq:egovv} provides an explicit description of $\Om_A$, as presented, e.g., in \cite{Loday, Fresse1, Joan}. Moreover, when combined with \cite[Theorem 5.1]{Hoang1}, it allows us to recover the significant result that the module of Kähler differentials of $A$ is a model for the cotangent complex $\rL_A \in \T_A\Alg_\P(\C(\textbf{k}))$, provided that $A\in\Alg_\P(\C(\textbf{k}))$ is cofibrant (see also \cite{YonatanCotangent} for another conceptual proof).
\end{rem}

Finally, combining Theorem \ref{t:cotan} with \cite[Remark 5.28]{Hoang1}, we may reformulate the Quillen cohomology of $\P$-algebras as follows.
\begin{cor}\label{co:QPalg} Let $A\in\Alg_\P(\C(\textbf{k}))$ be a cofibrant $\P$-algebra, and $N\in\Mod_{\P}^A$ an $A$-module. The Quillen cohomology of $A$ with coefficients in $N$ can be computed by the formula
	$$  \HHQ^\star(A ; N) \simeq \Omega \HHQ^\star(\P;\End_{A,N}) \simeq \Map^{\h}_{\IbMod(\P)}(\ovl{\rL}_\P,\End_{A,N}),$$
and the $n$-th Quillen cohomology group is computed by
  $$  \HHQ^n(A ; N) \cong \pi_0\Map^{\h}_{\IbMod(\P)}(\ovl{\rL}_\P,\End_{A,N}[n]).$$
\end{cor}

\subsection{Operadic modules of K{\"a}hler differentials and reduced Quillen cohomology}\label{s:redcotandg}

As before, we let $\P\in\Op_C(\C(\textbf{k}))$ be a fixed $C$-colored operad in $\C(\textbf{k})$, for some set of colors $C$.

Recall that $\rL^{\red}_\P \in \T_\P\Op_C(\C(\textbf{k}))$ signifies the cotangent complex of $\P$ when considered as an object of $\Op_C(\cS)$ (see $\S$\ref{s:other}). As in $\S$\ref{s:cotandg}, we have a Quillen equivalence
\begin{equation}\label{eq:redqadj} 
	\adjunction*{\simeq}{\IbMod(\P)}{\T_\P\Op_C(\C(\textbf{k}))}{},
\end{equation}
provided that $\P$ is $\Sigma$-cofibrant. Accordingly, the \textbf{reduced Quillen cohomology} of $\P$ with coefficients in an object $M\in\IbMod(\P)$ is given by
\begin{equation}\label{eq:redqcohom}
	 \HHQ^\star_{\red}(\P; M) \simeq \Map^{\h}_{\IbMod(\P)}(\rL^{\red}_\P, M)
\end{equation}
where we use the same notation for the derived image of $\rL^{\red}_\P$ through the identification \eqref{eq:redqadj}.

\smallskip

We now describe $\rL^{\red}_\P$, beginning with the following observation.

\begin{rem}\label{r:redpkah} As in \cite[$\S$3]{Guti}, $S^C$ denotes the \textbf{operad of $C$-colored operads}, yielding a categorical isomorphism
	\begin{equation}\label{eq:scop}
		\Op_C(\C(\textbf{k})) \cong \Alg_{S^C}(\C(\textbf{k})).
	\end{equation}
	Thus $\P$ may be viewed as an $S^C$-algebra, and one readily verifies the categorical isomorphism
	$$ \IbMod(\P) \cong \Mod_{S^C}^\P $$
	between infinitesimal $\P$-bimodules and $\P$-modules over $S^C$. Moreover, under these identifications, an (\textbf{operadic}) \textbf{derivation} from $\P$ to  $M\in\IbMod(\P)$ is a map $d : \P \lrar M$ of $C$-collections satisfying equations of the form
	\begin{equation}\label{eq:derivative}
		d(\mu\circ\nu) = \mu\circ^{i\ell}d(\nu) + d(\mu)\circ^{\ir}\nu
	\end{equation}
	in which $\mu \in \P(c_1,\cdots,c_m;c)$ and $\nu \in \P(d_1,\cdots,d_n;c_i)$ for some $i \in \{1,\cdots,m\}$.  
\end{rem}

Accordingly, the operadic analogue of Definition \ref{d:der} is defined as follows.

\begin{define}\label{d:OmP} The \textbf{(operadic) module of K{\"a}hler differentials} $\Om_\P \in \IbMod(\P)$ is  the module of K{\"a}hler differentials of $\P$ regarded as an $S^C$-algebra. This is characterized by natural isomorphisms of the form $$ \Hom_{\IbMod(\P)}(\Om_\P,M) \cong \Der(\P,M)$$
where $\Der(\P,M)$ refers to the set of derivations from  $\P$ to $M \in \IbMod(\P)$. (See \cite{Vallette} for more details.)
\end{define}

\begin{rem}\label{r:OmP} As in Remark \ref{r:OmA}, we can define the \textbf{square-zero extension} $\P \ltimes M \in \Op_C(\C(\textbf{k}))_{/\P}$ of $\P$ by an object $M \in \IbMod(\P)$ as follows. As a $C$-collection, $\P \ltimes M$ is the levelwise coproduct $\P \oplus M$. The unit operations are given by $(\id_c, 0)$ for $c\in C$. Moreover, for $(\mu , p) \in (\P \oplus M)(c_1,\cdots,c_m;c)$ and $(\nu , q) \in (\P \oplus M)(d_1,\cdots,d_n;c_i)$, the composition is defined by
		$$ (\mu , p) \circ (\nu , q) := (\mu\circ\nu, \mu\circ^{i\ell}q + p\circ^{\ir}\nu).$$
Finally, this construction determines  a right adjoint
	$$  \P \ltimes (-) :  \IbMod(\P) \lrar \Op_C(\C(\textbf{k}))_{/\P},$$
and we then obtain $\Om_\P \cong \Om^{/\P}(\Id_\P)$ where $\Om^{/\P}$ denotes the left adjoint to $\P \ltimes (-)$.
\end{rem}

\begin{example}\label{ex:derf} Suppose we have a map $f : \P \lrar \Q$ between $C$-colored dg operads, which endows $\Q$ with an  infinitesimal $\P$-bimodule structure. There is a canonical derivation $d_f : \P \lrar \Q$ defined by sending each operation $\lambda \in \P$ of arity $k$ to $$d_f(\lambda) := (k-1)f(\lambda).$$
	In particular, $d_f$ is given by $-f$ on \textit{nullary} (i.e., $0$-ary) operations, maps every unary operation to zero, and  agrees with $f$ on binary operations. 
\end{example}

We now state the main result of this section.

\begin{prop}\label{t:redop} Suppose that $\P\in\Op_C(\C(\textbf{k}))$ is cofibrant. Under the identification $$\T_\P\Op_C(\C(\textbf{k}))\simeq\IbMod(\P),$$ the reduced cotangent complex $\rL^{\red}_\P\in\T_\P\Op_C(\C(\textbf{k}))$ corresponds to $\Om_\P \in \IbMod(\P)$. Consequently, the reduced Quillen cohomology of $\P$ with coefficients in an object $M \in \IbMod(\P)$ is computed by
	\begin{equation}
		\HHQ^\star_{\red}(\P; M) \simeq \Map^{\h}_{\IbMod(\P)}(\Om_\P,M).
	\end{equation} 
\begin{proof} The isomorphism \eqref{eq:scop} determines a Quillen equivalence $\adjunction*{\simeq}{\Op_C(\C(\textbf{k}))}{\Alg_{S^C}(\C(\textbf{k}))}{}$ between the transferred semi-model structures (see $\S$\ref{s:bases}), from which we further obtain a Quillen equivalence between tangent categories $$\adjunction*{\simeq}{\T_\P\Op_C(\C(\textbf{k}))}{\T_\P\Alg_{S^C}(\C(\textbf{k}))}{}.$$ Clearly, the latter preserves cotangent complexes; that is, it sends $\rL^{\red}_\P$ to the cotangent complex of $\P$ considered as an $S^C$-algebra. We complete the proof by combining this with \cite[Corollary 2.5.11]{YonatanCotangent} (see also Remark \ref{r:control}).
\end{proof}
\end{prop}

Next, we outline the relation between Quillen cohomology and reduced Quillen cohomology of $\P$.

\begin{lem}\label{l:fibred} Suppose that $\P\in\Op_C(\C(\textbf{k}))$ is $\Sigma$-cofibrant. Then there is a (homotopy) cofiber sequence in $\IbMod(\P)$ of the form
	\begin{equation}\label{eq:cofred}
		\rL^{\red}_\P  \lrar \P\circ_{(1)}\P \lrar \rL_\P[1]
	\end{equation}
	where, as usual, we use the same notation for the derived images of $\rL^{\red}_\P$  and $\rL_\P$  in $\IbMod(\P)$, and $\P\circ_{(1)}\P$ represents the free  infinitesimal $\P$-bimodule generated by $\I_C$ (cf. Remark \ref{r:inffr}). 
\begin{proof} Let $\rL_{\P/\I_C}\in\T_\P\Op(\C(\textbf{k}))$ denote the relative cotangent complex of the unit map $\eta : \I_C \lrar \P$ (see Definition \ref{d:cotan}). It arises from a cofiber sequence in $\T_\P\Op(\C(\textbf{k}))$ of the form
	$$ \eta_!(\rL_{\I_C}) \lrar \rL_\P \lrar \rL_{\P/\I_C}$$ 
where $\eta_!$ refers to the induced functor $\eta_! : \T_{\I_C}\Op(\C(\textbf{k})) \lrar  \T_\P\Op(\C(\textbf{k}))$ (see Remark \ref{r:relcotan}), which is (homotopically) equivalent to the free functor
$$ \Free_\P^{\si} : \Coll_C(\C(\textbf{k})) \cong  \IbMod(\I_C) \lrar \IbMod(\P).$$
By Theorem \ref{t:cotan}, the derived image of $\rL_{\I_C}$ in $\Coll_C(\C(\textbf{k}))$ is  $\ovl{\rL}_{\I_C}[-1] = \I_C[-1]$, whose image through $\Free_\P^{\si}$ is precisely $(\P\circ_{(1)}\P)[-1]$. The proof is completed by using \cite[Lemma 5.2.9]{Hoang}, which proves that $\rL^{\red}_\P$ and $\rL_{\P/\I_C}$ have the same derived image in $\IbMod(\P)$.
\end{proof}
\end{lem}

The above lemma has the following immediate consequence.

\begin{cor}\label{co:fibred} Suppose that $\P$ is $\Sigma$-cofibrant and suppose given an object $M\in\IbMod(\P)$. Then there is a fiber sequence of spaces of the form
	\begin{equation}\label{eq:fibfred}
		\Om\HHQ^\star(\P; M)  \lrar \prod_{c\in C} |M(c;c)| \lrar \HHQ^\star_{\red}(\P; M)
	\end{equation}
	in which $|M(c;c)|$ refers to the underlying space of $M(c;c)\in \C(\textbf{k})$ (see $\S$\ref{s:other}). 
\end{cor}

\begin{rem} If $\P$ is cofibrant and connective, then by Theorem \ref{t:cotan} and Proposition \ref{t:redop}, the sequence \eqref{eq:cofred} takes the form
	\begin{equation}\label{eq:strcofred}
		\Om_\P \x{\Psi_\P}{\lrar} \P\circ_{(1)}\P \x{\varphi_\P}{\lrar} \ovl{\rL}_\P.
	\end{equation}
	Let us explain these two maps in more detail. As discussed earlier, $\Psi_\P$ is classified by a single derivation $d : \P \lrar \P\circ_{(1)}\P$, defined by sending each $\mu \in \P(c_1,\cdots,c_m;c)$ to
	$$ d(\mu)  := (\mu \, ; \, \id_{c_1}) + \cdots + (\mu \, ; \, \id_{c_m}) - (\id_{c} \, ; \, \mu).$$
	On the other hand, the map $\varphi_\P : \P\circ_{(1)}\P \lrar \ovl{\rL}_\P$ is classified by the canonical map $\I_C \lrar \ovl{\rL}_\P$ of $C$-collections, defined by inserting the unit operation $\id_c$ into $\ovl{\rL}_\P(c;c) = \P(c;c)$ for every $c\in C$. 
\end{rem}

\subsection{Cotangent complex of the dg $\E_\infty$-operad}\label{s:coeinfty}

We will fix a $\Sigma$-cofibrant model for the operad $\Com \in \Op_*(\C(\textbf{k}))$, denoted $\E_\infty$, and refer to it as the \textbf{dg} $\E_\infty$-\textbf{operad}. The category $\Fun(\Fin_*^{\op} , \C(\textbf{k}))$ is equipped with the projective model structure. 

\begin{thm}\label{t:LEinfty} There is a zigzag of left Quillen equivalences  $$ \Fun(\Fin_*^{\op} , \C(\textbf{k}))    \x{\simeq}{\llar}     \IbMod(\E_\infty) \lrarsimeq   \T_{\E_\infty}\Op(\C(\textbf{k})).$$
	Moreover, the cotangent complex $\rL_{\E_\infty} \in \T_{\E_\infty}\Op(\C(\textbf{k}))$ is then identified with $t[-1] : \Fin_*^{\op} \lrar \C(\textbf{k})$, i.e., the desuspension of the  Pirashvili functor $t$ (see Example \ref{ex:tfunctor}). 
	\begin{proof} Clearly, the induced functor $\textbf{Ib}^{\E_\infty} \lrar \textbf{Ib}^{\Com} \cong \Fin_*^{\op}$ is a weak equivalence of dg categories (see $\S$\ref{s:opinfi}). Thus, we obtain a Quillen equivalence	$ \adjunction*{\simeq}{\IbMod(\E_\infty)}{\Fun(\Fin_*^{\op} , \C(\textbf{k}))}{}$. Combined with Theorem \ref{t:keylem}(1), this proves the first statement.
		
	For the second statement, we just need to show that the derived image of	$\ovl{\rL}_{\E_\infty}$ under the left Quillen equivalence  $\IbMod(\E_\infty) \lrarsimeq \Fun(\Fin_*^{\op} , \C(\textbf{k}))$ is weakly equivalent to $t$. To this end, it is enough to verify that $\ovl{\rL}_{\Com} \in \IbMod(\Com)$, regarded as a functor $\Fin_*^{\op} \lrar \C(\textbf{k})$, is isomorphic to $t$. Unwinding the definitions, $\ovl{\rL}_{\Com}$ sends each object $\l m \r$ to $\ovl{\rL}_{\Com}(\l m \r) = \textbf{k}^{\oplus \, m}$; and moreover, for each map $f : \l m \r \lrar \l n \r$ in $\Fin_*$, the corresponding map 
	$$ \ovl{\rL}_{\Com}(\l n \r) = \textbf{k}^{\oplus \, \{1,\cdots,n\}} \lrar  \textbf{k}^{\oplus \, \{1,\cdots,m\}}  =  \ovl{\rL}_{\Com}(\l m \r)$$ 
	is given by, for each $i\in\{1,\cdots,n\}$, mapping the summand $\textbf{k}^{\{i\}}$ to $\textbf{k}^{\oplus \, \{f^{-1}(i)\}} \subseteq \textbf{k}^{\oplus \, \{1,\cdots,m\}}$ via a diagonal map. This exhibits an equivalent way of defining $t$. 
	\end{proof}
\end{thm}

In light of the above theorem, Quillen cohomology of the dg $\E_\infty$-operad will take coefficients in functors $\Fin_*^{\op} \lrar \C(\textbf{k})$. 

\begin{cor}\label{co:QEinfty}  The Quillen cohomology of $\E_\infty$ with coefficients in a given functor $\F : \Fin_*^{\op} \lrar \C(\textbf{k})$ is computed by
	$$ \HHQ^\star(\E_\infty ; \F)  \simeq \Map^{\h}_{\Fun(\Fin_*^{\op} , \C(\textbf{k}))}(t[-1],\F).$$
Moreover, the $n$-th Quillen cohomology group is computed by the formula
$$ \HHQ^n(\E_\infty ; \F) \cong \pi_0\Map^{\h}_{\Fun(\Fin_*^{\op} , \C(\textbf{k}))}(t,\F[n+1]).$$
\end{cor}

\begin{rem} Accordingly, $\HHQ^\star(\E_\infty ; \F)$ is (up to a shift) equivalent to the \textbf{stable cohomotopy} of the functor $\F$ (cf. \cite{Pirash}). This shows that the obstruction theory for $\E_\infty$-\textbf{structures}, as initiated by Robinson \cite{Robinson}, can be encoded by Quillen cohomology of $\E_\infty$ itself (cf. \cite[Theorem 5]{Robinson}).
\end{rem}

Next, let $A$ be an $\E_\infty$-algebra and $N$ an $A$-module. As in $\S$\ref{s:endo}, the pair $(A,N)$ gives rise to an object $ \End_{A,N} \in  \IbMod(\E_\infty)$. We use the same notation for the derived image of the latter in $\Fun(\Fin_*^{\op} , \C(\textbf{k}))$. Combining Theorem \ref{t:LEinfty} with \cite[Remark 5.28]{Hoang1}, we obtain another noteworthy consequence.

\begin{cor}\label{co:QEinftyalg} Suppose that $A\in\Alg_{\E_\infty}(\C(\textbf{k}))$ is cofibrant.  Then the Quillen cohomology of $A$ with coefficients in an $A$-module $N$ is computed by 
	$$  \HHQ^\star(A ; N)  \simeq \Omega\HHQ^\star(\E_\infty ; \End_{A,N})   \simeq \Map^{\h}_{\Fun(\Fin_*^{\op} , \C(\textbf{k}))}(t,\End_{A,N}).$$
\end{cor}

\section{Cotangent complex of dg $\E_n$-operads}\label{s:Qprin}

Let us first recall the definition of the \textbf{topological little $n$-discs operad} $\E_n$. We  denote by $$\sD^n:=\{x\in\RR^n \, | \, \lVert x \rVert <1 \}$$ the \textbf{open (unit)  $n$-disc}. By definition, a map $\sD^n \lrar \sD^n$ is a \textbf{rectilinear embedding} if it takes the form $x \mapsto \lambda x +c$ for $\lambda>0$. Moreover, a map $\underset{k}{\bigsqcup}\sD^n \lrar \sD^n$ from $k$ disjoint open $n$-discs to another open $n$-disc is a \textit{rectilinear embedding} if it is an open embedding whose restriction to each component $\sD^n$ is rectilinear. The operad $\rE_n$ is then defined to be the single-colored topological operad with
$$ \rE_n(k) = \Rect(\underset{k}{\bigsqcup}\sD^n,\sD^n)$$ 
the space of rectilinear embeddings of $k$ disjoint open $n$-discs in another open $n$-disc, endowed with the subspace topology inherited from the mapping space in $\Top$. The operadic composition is induced in the evident way by the composition of rectilinear embeddings. 

\begin{notn}\label{no:En} We use the same notation $\rE_n$ to denote the simplicial version of the little $n$-discs operad, for every $n\geq0$. Moreover, we denote by $\mathbb{E}_n$ the \textbf{differential graded (dg) version of}  $\E_n$. Explicitly, $\mathbb{E}_n$ is obtained from $\rE_n$ via the composite functor
	$$ \Op_*(\Set_\Delta) \x{\textbf{k}[-]}{\lrar} \Op_*(\sMod(\textbf{k})) \x{\sN}{\lrar} \Op_*(\C(\textbf{k})) $$
where the first functor is defined by applying the free functor $\textbf{k}[-] : \Set_\Delta \lrar \sMod(\textbf{k})$ levelwise, and  $\sN$ is defined by applying the \textbf{normalization} (or \textbf{normalized chain complex}) \textbf{functor} levelwise (cf., e.g. \cite{Hoang2}). 
\end{notn}

Notice that $\mathbb{E}_n$ is adequate (cf. Definition \ref{d:ade}), and hence, by Theorem \ref{t:cotan} its cotangent complex is represented by $\ovl{\rL}_{\mathbb{E}_n}[-1] \in \IbMod(\mathbb{E}_n)$. The main goal of this section is to prove the following result.

\begin{thm}\label{t:QprinEn} There is a cofiber sequence in $\IbMod(\mathbb{E}_n)$ of the form
\begin{equation}\label{eq:QprinEn}
	  \Free_{\mathbb{E}_n}^{\si}(\mathcal{E}_*) \lrar \mathbb{E}_n^{\si} \lrar \ovl{\rL}_{\mathbb{E}_n}[n]
\end{equation}
in which the first map is induced by the identification $\textbf{k} \overset{\cong}{\lrar} \mathbb{E}_n(0)$ (see  Remark \ref{r:inffr} and $\S$\ref{s:other} for notations). 
\end{thm}

The proof will make use of the following auxiliary material.

\begin{cons}\label{con:M_n} According to Example \ref{ex:rightkan}, we consider the functor (or infinitesimal $\mathbb{E}_n$-bimodule) $$\underline{1}_*\left[\textbf{k}[n-1]\right] : \textbf{Ib}^{\mathbb{E}_n} \lrar \C(\textbf{k}), \, \underline{k} \; \mapsto \; \bigoplus_{i=0}^k \, T_k^i$$
in which $T_k^i := \Map_{\C(\textbf{k})}\left(\Map^{\rho_i}_{\textbf{Ib}^{\mathbb{E}_n}}(\underline{k},\underline{1}), \textbf{k}[n-1]\right)$ for each $i = 0,\cdots,k$. Notice that $\mathbb{E}_n(1) \simeq\mathbb{E}_n(0) \cong \textbf{k}$, and $\mathbb{E}_n(2) \simeq \sS^{n-1} \cong \textbf{k}\oplus\textbf{k}[n-1]$. Therefore, we further obtain 
$$  T_k^0  = \Map_{\C(\textbf{k})}(\mathbb{E}_n(2), \textbf{k}[n-1]) \simeq \sS^{n-1}  \;\;\; \text{and}  \;\;\; T_k^i  = \Map_{\C(\textbf{k})}(\mathbb{E}_n(1)^{\otimes 2}, \textbf{k}[n-1]) \simeq \textbf{k}[n-1]  $$
for $i = 1,\cdots,k$. Next, we denote by $M_n := \mathbb{E}_n^{\si} \, \otimes \, \underline{1}_*\left[\textbf{k}[n-1]\right]$, i.e., the levelwise tensor product of $\mathbb{E}_n^{\si}$ and $\underline{1}_*\left[\textbf{k}[n-1]\right]$. This  carries a canonical infinitesimal $\mathbb{E}_n$-bimodule structure induced by the diagonal action, as $\mathbb{E}_n$ arises from a simplicial operad. For each $k\geq0$, we have
\begin{align*}
	&M_n(k) = \bigoplus_{i=1}^k \, \left(\mathbb{E}_n(k)\otimes T_k^i\right) \, \oplus \, \left(\mathbb{E}_n(k)\otimes T_k^0\right) \\
	&\phantom{M_n(k)  {}} \simeq \, \underset{k}{\bigoplus}\,\mathbb{E}_n(k)[n-1]\, \oplus \, \left(\mathbb{E}_n(k) \oplus \mathbb{E}_n(k)[n-1]\right).
\end{align*}
\end{cons}

\begin{cons}\label{con:ovlM_n} Given $X\in\C(\textbf{k})$, following Example \ref{ex:rightkan1}, we consider the functor
$$\underline{0}_*[X] : \textbf{Ib}^{\mathbb{E}_n} \lrar \C(\textbf{k})  , \;\; \underline{k}  \mapsto \Map_{\C(\textbf{k})}(\mathbb{E}_n(1)\otimes\mathbb{E}_n(0)^{\otimes k}, \, X) \simeq X.$$
We regard $\underline{0}_*[X]$ as a \textit{weakly constant} \textit{functor}  $\textbf{Ib}^{\mathbb{E}_n} \lrar \C(\textbf{k})$ with  value $X$. For $X = \textbf{k}[n-1]$ and $X=\sS^{n-1}$,  set 
$$ P_n := \mathbb{E}_n^{\si} \otimes \, \underline{0}_*\left[\textbf{k}[n-1]\right]  \;\;\; \text{and} \;\;\; Q_n := \mathbb{E}_n^{\si} \otimes \, \underline{0}_*[\sS^{n-1}].$$
As in Construction \ref{con:M_n}, both of these carry the structure of an infinitesimal $\mathbb{E}_n$-bimodule.
\end{cons}

\begin{rem}\label{r:constant} Since $\mathbb{E}_n$ arises from a simplicial operad, we may consider the (strictly) constant functor  $\ovl{X} : \textbf{Ib}^{\mathbb{E}_n} \lrar \C(\textbf{k})$ with value $X$. The canonical map $\ovl{X} \lrar \underline{0}_*[X]$, classified by $\Id_X$, is a weak equivalence in $\IbMod(\mathbb{E}_n)$. We therefore obtain two weak equivalences in $\IbMod(\mathbb{E}_n)$:
	$$ \mathbb{E}_n^{\si}[n-1] =  \mathbb{E}_n^{\si}\otimes\ovl{\textbf{k}[n-1]}  \lrarsimeq P_n \;\;\;\;\; \text{and} \;\;\;\;\; \mathbb{E}_n^{\si} \oplus\mathbb{E}_n^{\si}[n-1] =  \mathbb{E}_n^{\si}\otimes\ovl{\sS^{n-1}}  \lrarsimeq Q_n.$$
\end{rem}

\begin{rem}\label{r:split} We let $\mu_0 \in \E_n(0)$ denote the unique nullary operation of $\E_n$, and consider the map $$ (-)\circ_1 \mu_0 : \rE_n(k+1) \lrar \rE_n(k).$$ Visually, this is obtained by deleting the first disc from each configuration of $k+1$ disjoint $n$-discs inside the unit $n$-disc, viewing such a configuration as a point of $\rE_n(k+1)$. The map $(-)\circ_1 \mu_0$ is a fiber bundle  with fiber homotopy equivalent to the $k$-fold wedge sum $ \underset{k}{\bigvee} \sS_{\st}^{n-1}$ of the topological $(n-1)$-sphere $\sS_{\st}^{n-1}$, and moreover, it splits after a single suspension:
	$$ \Sigma\rE_n(k+1) \, \simeq \, \Sigma\left(\rE_n(k) \times \underset{k}{\bigvee}\sS_{\st}^{n-1} \right).$$
In light of this, we obtain a canonical quasi-isomorphism:
\begin{equation}\label{eq:split}
\mathbb{E}_n(k+1) \, \simeq \, \underset{k}{\bigoplus}\,\mathbb{E}_n(k)[n-1]\, \oplus \, \mathbb{E}_n(k).
\end{equation}
\end{rem}

\begin{cons}\label{cons:split} This quasi-isomorphism can be described as follows. For $k=0$, it is given by the projection  $\mathbb{E}_n(1) \lrarsimeq \mathbb{E}_n(0) \cong \textbf{k}$. Assume now that $k\geq1$. Since $\mathbb{E}_n(2) \simeq \sS^{n-1}$, we can write
	$$ \underset{k}{\bigoplus}\,\mathbb{E}_n(k)[n-1]\, \oplus \, \mathbb{E}_n(k) \, \simeq \, (\mathbb{E}_n(k)\otimes\mathbb{E}_n(2))\underset{\mathbb{E}_n(k)}{\bigsqcup} \cdots \underset{\mathbb{E}_n(k)}{\bigsqcup}(\mathbb{E}_n(k)\otimes\mathbb{E}_n(2)) =: \mathfrak{B}_k$$
	where $\mathfrak{B}_k$ refers to the $k$-fold cofiber coproduct of the map $\mathbb{E}_n(k) \lrar \mathbb{E}_n(k)\otimes\mathbb{E}_n(2)$  induced by the inclusion $\textbf{k} \lrar \sS^{n-1} \simeq \mathbb{E}_n(2)$. We write $\id\in\rE_n(1)$ for the identity operation. For each $1 \leq i \leq k$, we define a map $\tau_i : \mathbb{E}_n(k+1) \lrar \mathbb{E}_n(k) \otimes \mathbb{E}_n(2)$ induced by the map
	$$  \rE_n(k+1) \lrar \rE_n(k) \times  \rE_n(2) , \; \nu  \mapsto (\nu\circ_1 \mu_0  , \, \nu\circ(\id,\mu_0,\cdots,\mu_0,\id,\mu_0,\cdots,\mu_0)) $$
	in which the latter $``\id$'' is placed in the $(i+1)$-th position. (The latter operation is visually obtained from $\nu$ by deleting all but the first and $(i+1)$-th discs.) The quasi-isomorphism \eqref{eq:split} is then modeled by the composite
	\begin{equation}\label{eq:split1}
\mathbb{E}_n(k+1) \; \overset{(\tau_1,\cdots,\tau_k)}{\lrar} \; \underset{k}{\bigoplus}\,\mathbb{E}_n(k)\otimes\mathbb{E}_n(2) \lrar \mathfrak{B}_k
	\end{equation}
	where the second map is the canonical projection. 
\end{cons}

\begin{lem}\label{l:firstcof} There is a cofiber sequence in $\IbMod(\mathbb{E}_n)$ of the form
	\begin{equation}\label{eq:firstcof}
		\ovl{\rL}_{\mathbb{E}_n}[n-1] \overset{\iota}{\lrar} M_n \overset{\rho}{\lrar} Q_n.
	\end{equation} 
\begin{proof} First, consider the functor (or infinitesimal $\Com$-bimodule) $\ovl{\rL}_{\Com} : \textbf{Ib}^{\Com} \lrar \C(\textbf{k})$ (see $\S$\ref{s:coeinfty}), and let $\ovl{\rL}^n_{\Com} : \textbf{Ib}^{\mathbb{E}_n} \lrar \C(\textbf{k})$ denote its restriction to $\textbf{Ib}^{\mathbb{E}_n}$. Then observe that $\ovl{\rL}_{\mathbb{E}_n} \cong \mathbb{E}_n^{\si} \otimes \ovl{\rL}^n_{\Com}$. By adjunction, we have a canonical map $$\ovl{\rL}^n_{\Com}[n-1] \lrar \underline{1}_*\left[\textbf{k}[n-1]\right]$$ classified by the identity on $\textbf{k}[n-1]$. Accordingly, the map $\iota$ is given by the  induced map
	$$ \iota : \ovl{\rL}_{\mathbb{E}_n}[n-1] = \mathbb{E}_n^{\si} \otimes \ovl{\rL}^n_{\Com}[n-1] \lrar \mathbb{E}_n^{\si} \otimes \underline{1}_*\left[\textbf{k}[n-1]\right] = M_n.$$
	Unwinding the definitions, at each level $k$ this map is weakly equivalent to the  embedding
	$$ \ovl{\rL}_{\mathbb{E}_n}[n-1](k) \simeq \underset{k}{\bigoplus}\,\mathbb{E}_n(k)[n-1] \; \lrar \; \underset{k}{\bigoplus}\,\mathbb{E}_n(k)[n-1]  \oplus  \left(\mathbb{E}_n(k) \oplus \mathbb{E}_n(k)[n-1]\right) \simeq M_n(k).$$
	
	Next, since $\underline{1}_*\left[\textbf{k}[n-1]\right]$ is given by $T_0^0$ at level $0$, we get a map $\underline{1}_*\left[\textbf{k}[n-1]\right] \lrar \underline{0}_*[\sS^{n-1}]$ classified by the identification $T_0^0 \simeq \sS^{n-1}$. Accordingly, the map $\rho$ is the induced map $$\rho : M_n = \mathbb{E}_n^{\si} \otimes \underline{1}_*\left[\textbf{k}[n-1]\right] \, \lrar \,  \mathbb{E}_n^{\si} \otimes \, \underline{0}_*[\sS^{n-1}] = Q_n.$$
	Moreover, at each level this map is (up to homotopy) given by the projection
	$$ M_n(k) \, \simeq \, \underset{k}{\bigoplus}\,\mathbb{E}_n(k)[n-1]\, \oplus \, \left(\mathbb{E}_n(k) \oplus \mathbb{E}_n(k)[n-1]\right) \; \lrar \; \mathbb{E}_n(k) \oplus \mathbb{E}_n(k)[n-1] \simeq Q_n(k).$$
	
	From the two paragraphs above, we indeed obtain the expected cofiber sequence \eqref{eq:firstcof}.
\end{proof}
\end{lem}

\begin{lem}\label{le:secondcof} There is a cofiber sequence in $\IbMod(\mathbb{E}_n)$ of the form
	\begin{equation}\label{eq:secondcof}
		\Free_{\mathbb{E}_n}^{\si}(\mathcal{E}_*) \overset{\varepsilon}{\lrar} M_n \overset{\pi}{\lrar} P_n.
	\end{equation} 
\begin{proof} First recall that $\Free_{\mathbb{E}_n}^{\si}(\mathcal{E}_*) \cong \underline{0}_![\textbf{k}]$ (cf. Example \ref{ex:inffree1}), and that $M_n(0) = \mathbb{E}_n(0)\otimes T_0^0$.  Let $\ovl{\varepsilon} : \textbf{k} \lrar M_n(0)$ be the map classified by a pair $(\mu_0, \kappa)$ in which $\mu_0$ represents the generator of $\mathbb{E}_n(0) \cong \textbf{k}$ (see Remark \ref{r:split}), and $\kappa$ denotes the projection $\mathbb{E}_n(2) \simeq \sS^{n-1} \lrar  \textbf{k}[n-1]$. (Therefore, $\ovl{\varepsilon}$ is  weakly equivalent to the inclusion $\textbf{k} \lrar \textbf{k}\oplus\textbf{k}[n-1]$.) Accordingly, $\varepsilon : \Free_{\mathbb{E}_n}^{\si}(\mathcal{E}_*) \lrar M_n$ is the map induced by $\ovl{\varepsilon}$ via adjunction, and at each level $k$ takes the form
	$$ \varepsilon(k)  : \Free_{\mathbb{E}_n}^{\si}(\mathcal{E}_*)(k) = \mathbb{E}_n(k+1) \; \lrar \; \bigoplus_{i=1}^k \, \left(\mathbb{E}_n(k)\otimes T_k^i\right) \, \oplus \, \left(\mathbb{E}_n(k)\otimes T_k^0\right) = M_n(k)$$
	(see also Remark \ref{r:inffr}). Concretely, the map $\varepsilon(k)$ consists, for each $0 \leq i \leq k$, of a component map  $\varepsilon_i(k) : \mathbb{E}_n(k+1) \lrar \mathbb{E}_n(k)\otimes T_k^i$, which is defined through the following steps:
	
	\begin{enumerate}
		
		\item First,  let $\upsilon_0 : \mathbb{E}_n(k+1) \lrar T_k^0$ be the map classified by the composite
		$$ \mathbb{E}_n(k+1)\otimes \mathbb{E}_n(2) \lrar \mathbb{E}_n(2) \overset{\kappa}{\lrar} \textbf{k}[n-1]$$
		where the first map is given by the rule $$\rE_n(k+1) \times \rE_n(2) \; \ni \; (\nu, \mu) \; \mapsto \; \mu \circ_1\nu \circ (\id,\mu_0,\cdots,\mu_0) \; \in \; \rE_n(2).$$
		
		\item Next, for $1 \leq i \leq k$, let $\upsilon_i : \mathbb{E}_n(k+1) \lrar T_k^i$ be the map induced by 
		$$ \mathbb{E}_n(k+1) \otimes \mathbb{E}_n(1)^{\otimes 2} \lrar \mathbb{E}_n(2) \overset{\kappa}{\lrar} \textbf{k}[n-1]$$
		in which the first map is determined by 
		$$\rE_n(k+1) \times \rE_n(1) \times \rE_n(1) \; \ni \; (\nu, \alpha, \beta) \; \mapsto \; \alpha \circ \nu \circ (\id,\mu_0,\cdots,\mu_0,\beta,\mu_0,\cdots,\mu_0) \; \in \; \rE_n(2)$$ 
		such that $\beta$ is placed in the $(i+1)$-th position.
		
		\item Finally, for each $0 \leq i \leq k$, the map  $\varepsilon_i(k)$ is the composed map
		$$  \mathbb{E}_n(k+1) \; \lrar \; \mathbb{E}_n(k+1)\otimes\mathbb{E}_n(k+1) \; \overset{(-\,\circ_1 \mu_0) \otimes \upsilon_i}{\lrar} \; \mathbb{E}_n(k)\otimes T_k^i $$
		where the first map is induced by the diagonal map.
		
	\end{enumerate}

	 Clearly $\varepsilon_0(k)$ maps $\mathbb{E}_n(k+1)$ onto the summand $\mathbb{E}_n(k) \subseteq \mathbb{E}_n(k)\otimes T_k^0$ via the map $(-)\circ_1 \mu_0$, and furthermore, one can show that the corestriction of $\varepsilon(k)$ onto the complement of the summand $\mathbb{E}_n(k)[n-1] \subseteq \mathbb{E}_n(k)\otimes T_k^0$ in $M_n(k)$ is weakly equivalent to the map \eqref{eq:split1}. It implies that $\varepsilon(k)$ is weakly equivalent to the inclusion
	$$ \underset{k}{\bigoplus}\,\mathbb{E}_n(k)[n-1]\, \oplus \, \mathbb{E}_n(k) \lrar  \underset{k}{\bigoplus}\,\mathbb{E}_n(k)[n-1]\, \oplus \, \left(\mathbb{E}_n(k) \oplus \mathbb{E}_n(k)[n-1]\right).$$

	On the other hand, the map $\pi : M_n \lrar P_n$ is given by the composite
	$$ M_n  \overset{\rho}{\lrar}  Q_n \lrar P_n.$$
	Here, the second map is induced by the map $\underline{0}_*\left[\sS^{n-1}\right] \lrar \underline{0}_*\left[\textbf{k}[n-1]\right]$, which arises from the projection $\sS^{n-1} \lrar \textbf{k}[n-1]$. We can verify that $\pi$ is (up to homotopy) given at each level by the projection onto the distinguished summand $\mathbb{E}_n(k)[n-1] \subseteq \mathbb{E}_n(k)\otimes T_k^0$ : 
	$$ M_n(k) \, \simeq \, \underset{k}{\bigoplus}\,\mathbb{E}_n(k)[n-1]\, \oplus \, \left(\mathbb{E}_n(k) \oplus \mathbb{E}_n(k)[n-1]\right) \; \lrar \; \mathbb{E}_n(k)[n-1] \simeq P_n(k).$$   
	Together with the previous paragraph, this yields the cofiber sequence \eqref{eq:secondcof}.
\end{proof}
\end{lem}

\begin{proof}[\underline{Proof of Theorem \ref{t:QprinEn}}] We denote by $K_n := \mathbb{E}_n^{\si} \otimes \, \underline{0}_*\left[\textbf{k}\right]$, and consider it as an infinitesimal $\mathbb{E}_n$-bimodule (see Construction \ref{con:ovlM_n}). As in Remark \ref{r:constant}, there is a weak equivalence $\mathbb{E}_n^{\si} \lrarsimeq K_n$ induced by the map $\ovl{\textbf{k}} \lrarsimeq \underline{0}_*[\textbf{k}]$. Moreover, since $\sS^{n-1} \cong \textbf{k} \oplus \textbf{k}[n-1]$, we have $Q_n \simeq K_n \oplus P_n$. Combining this with the cofiber sequences  \eqref{eq:firstcof} and \eqref{eq:secondcof}, we obtain a diagram of homotopy (co)Cartesian squares in $\IbMod(\mathbb{E}_n)$:
$$ \xymatrix{
	\ovl{\rL}_{\mathbb{E}_n}[n-1] \ar[r]\ar[d] &	\Free_{\mathbb{E}_n}^{\si}(\mathcal{E}_*) \ar[r]^{ \;\;\;\;\;\;\;\; \varepsilon}\ar[d] & M_n \ar[d]^\rho \\
	0 \ar[r] &	K_n \ar[r] & Q_n \\
}$$
in which the top horizontal composed map is weakly equivalent to $\ovl{\rL}_{\mathbb{E}_n}[n-1] \overset{\iota}{\lrar} M_n$, and the middle vertical arrow is given by the composite $\Free_{\mathbb{E}_n}^{\si}(\mathcal{E}_*) \lrar \mathbb{E}_n^{\si} \lrarsimeq K_n$. Accordingly, we get a cofiber sequence of the form
\begin{equation}\label{eq:phi}
	\ovl{\rL}_{\mathbb{E}_n}[n-1] \lrar \Free_{\mathbb{E}_n}^{\si}(\mathcal{E}_*) \lrar \mathbb{E}_n^{\si},
\end{equation}
from which the proof is completed. 
\end{proof}

\begin{notn}\label{no:varphi} We denote by $\varphi : \ovl{\rL}_{\mathbb{E}_n}[n-1] \lrar \Free_{\mathbb{E}_n}^{\si}(\mathcal{E}_*)$ the map that emerged in the cofiber sequence \eqref{eq:phi}.
\end{notn}

\begin{rem} At each level $k$, the map $\varphi$ is weakly equivalent to the inclusion
	$$ \ovl{\rL}_{\mathbb{E}_n}[n-1](k) \, \simeq \, \underset{k}{\bigoplus}\,\mathbb{E}_n(k)[n-1] \; \lrar \; \underset{k}{\bigoplus}\,\mathbb{E}_n(k)[n-1]  \oplus  \mathbb{E}_n(k) \, \simeq \, \Free_{\mathbb{E}_n}^{\si}(\mathcal{E}_*)(k).$$
\end{rem}

\begin{rem} The cofiber sequence \eqref{eq:QprinEn} is the dg analogue of the so-called \textit{Quillen principle for topological $\E_n$-operads} (cf. \cite[$\S$6.3]{Hoang} and \cite[$\S$6.3]{Hoang1}).
\end{rem}

As in \cite[Example 5.10]{Hoang1}, $\mathbb{E}_n^{\si} \in \IbMod(\mathbb{E}_n)$ represents the \textbf{Hochschild complex} of the operad $\mathbb{E}_n$, and for $M\in\IbMod(\mathbb{E}_n)$, the space  $$\HHH^\star(\mathbb{E}_n ; M) := \Map^{\h}_{\IbMod(\mathbb{E}_n)}(\mathbb{E}_n^{\si}, M)$$  is the \textbf{Hochschild cohomology} of $\mathbb{E}_n$ with coefficients in $M$. The following is an immediate  consequence of Theorem \ref{t:QprinEn}.

\begin{cor}\label{co:QprinEn} Suppose given an object $M\in\IbMod(\mathbb{E}_n)$. Then there is a fiber sequence of spaces of the form
	$$ \Om^{n+1}\HHQ^\star(\mathbb{E}_n ; M) \lrar \HHH^\star(\mathbb{E}_n ; M) \lrar |M(0)|$$
(see $\S$\ref{s:other}(i) for notation). Consequently, there is a long exact sequence of $\textbf{k}$-modules:
$$   \cdots \lrar \HHQ^{-k-n-1}(\mathbb{E}_n ; M) \lrar \HHH^{-k}(\mathbb{E}_n ; M) \lrar \sH_k(M(0))    $$ 
$$ \lrar \HHQ^{-k-n}(\mathbb{E}_n ; M) \lrar \HHH^{-k+1}(\mathbb{E}_n ; M) \lrar \sH_{k-1}(M(0)) \lrar \cdots \, .$$
\end{cor}

The statement below provides an interesting example. Recall that for each integer $m\geq n$, there is an embedding $\mathbb{E}_n \lrar \mathbb{E}_m$ of dg operads, which endows $\mathbb{E}_m$ with the structure of an infinitesimal $\mathbb{E}_n$-bimodule.

\begin{cor}\label{co:QprinEnm} For each integer $k \neq 0$, there is a canonical isomorphism 
	$$  \HHQ^{-k-n-1}(\mathbb{E}_n ; \mathbb{E}_m) \cong \HHH^{-k}(\mathbb{E}_n ; \mathbb{E}_m).$$
Moreover, the zeroth Hochschild cohomology group factors as
$$ \HHH^{0}(\mathbb{E}_n ; \mathbb{E}_m) \cong \HHQ^{-n-1}(\mathbb{E}_n ; \mathbb{E}_m) \oplus \textbf{k}.$$
\begin{proof} Note first that $\sH_\bullet(\mathbb{E}_m(0)) \cong \textbf{k}$, concentrated in degree $0$. Together with the long exact sequence of Corollary \ref{co:QprinEn}, this yields $$\HHQ^{-k-n-1}(\mathbb{E}_n ; \mathbb{E}_m) \cong \HHH^{-k}(\mathbb{E}_n ; \mathbb{E}_m)$$ for all $k \in \ZZ\setminus\{-1,0\}$. Moreover, we obtain an exact sequence of $\textbf{k}$-modules:
\begin{equation}\label{eq:QprinEnm}
0 \lrar  \HHQ^{-n-1}(\mathbb{E}_n ; \mathbb{E}_m) \lrar \HHH^{0}(\mathbb{E}_n ; \mathbb{E}_m) \lrar \textbf{k} \x{\ovl{\varphi}}{\lrar} \HHQ^{-n}(\mathbb{E}_n ; \mathbb{E}_m) \lrar \HHH^{1}(\mathbb{E}_n ; \mathbb{E}_m) \lrar 0
\end{equation}
where $\ovl{\varphi}$ is the map $$\textbf{k} \cong \pi_0\Map^{\h}_{\IbMod(\mathbb{E}_n)}(\Free_{\mathbb{E}_n}^{\si}(\mathcal{E}_*),\mathbb{E}_m) \x{\ovl{\varphi}}{\lrar} \pi_0\Map^{\h}_{\IbMod(\mathbb{E}_n)}(\ovl{\rL}_{\mathbb{E}_n}[n-1],\mathbb{E}_m) \cong \HHQ^{-n}(\mathbb{E}_n ; \mathbb{E}_m)$$
induced by $\varphi : \ovl{\rL}_{\mathbb{E}_n}[n-1] \lrar \Free_{\mathbb{E}_n}^{\si}(\mathcal{E}_*)$ (see Notation \ref{no:varphi}). Hence, by the exactness of \eqref{eq:QprinEnm}, it suffices to prove that $\ovl{\varphi}$ is trivial.

We write $\xi_m : \Free_{\mathbb{E}_n}^{\si}(\mathcal{E}_*) \lrar \mathbb{E}_m$ for the map induced by the identification $\textbf{k} \x{\cong}{\lrar} \mathbb{E}_m(0)$, and let $\xi_m$ represent the generator of the domain of $\ovl{\varphi}$. Thus, $\ovl{\varphi}$ is classified by the composition $\xi_m\circ\varphi$. Moreover, observe that $\xi_m$ coincides with the composed map
$$ \Free_{\mathbb{E}_n}^{\si}(\mathcal{E}_*) \x{\xi_n}{\lrar} \mathbb{E}_n \lrar \mathbb{E}_m,$$
and hence, $\xi_m\circ\varphi$ is the same as the composed map
$ \ovl{\rL}_{\mathbb{E}_n}[n-1] \x{\varphi}{\lrar} \Free_{\mathbb{E}_n}^{\si}(\mathcal{E}_*) \x{\xi_n}{\lrar} \mathbb{E}_n \lrar \mathbb{E}_m$. But $\xi_n\circ\varphi$ is weakly equivalent to the zero map (as shown in the proof of Theorem \ref{t:QprinEn}). Accordingly, $\xi_m\circ\varphi$ indeed represents the trivial class in the codomain of $\ovl{\varphi}$. The proof is therefore completed.
\end{proof}
\end{cor}

\section{Deformation theory and Quillen cohomology}\label{s:defandQcohom}

Our main purpose in this section is to establish a precise relationship between \textbf{deformation theory} and Quillen cohomology. We illustrate the main theorem by examining the case of dg ($\E_n$-)operads.

\subsection{From formal moduli contexts to formal moduli problems}\label{sub:fmc}

We introduce the notion of a \textbf{formal moduli context}, designed to capture \textbf{formal moduli problems} associated with the \textbf{deformation spaces} of objects in a model category of interest.

We will use the following notations and conventions.

\begin{enumerate}\label{enu:defQ}
	
	\item Throughout this section, we assume that $\textbf{k}$ is a field of characteristic $0$.
	
	\item We denote by $\CAlg^{\aug} := \Alg^{\aug}_{\Com}(\C_{\geq 0}(\textbf{k}))$ the category of augmented commutative algebras in $\C_{\geq 0}(\textbf{k})$. Explicitly, an object of $\CAlg^{\aug}$ is a commutative dg $\textbf{k}$-algebra $R$, concentrated in non-negative degrees and equipped with a map $R \lrar \textbf{k}$ of dg $\textbf{k}$-algebras.
	
	\item We endow the category $\CAlg^{\aug}$ with the transferred model structure (see $\S$\ref{s:bases}). This is a pointed model category with zero object $\textbf{k}$ itself. We then consider the stabilization $\Sp(\CAlg^{\aug})$, which exists as a semi-model category.   
	
	\item For $R \in \CAlg^{\aug}$ and $n\in\NN$, we let $\pi_n(R)$ denote the $n$-th homology group of the underlying complex of $R$. 
	
	\item By convention, an augmented commutative algebra $\eps:R \lrar \textbf{k}$ is  \textbf{artinian} if the underlying complex of $R$ is finite dimensional and the map $\pi_0(\eps):\pi_0(R) \lrar \pi_0(\textbf{k}) = \textbf{k}$ exhibits $\pi_0(R)$ as a \textit{local} $\textbf{k}$-\textit{algebra} (i.e., the kernel of $\pi_0(\eps)$ is the unique maximal ideal of $\pi_0(R)$).
	
	\item We will denote by $\CAlg^{\art} \subseteq \CAlg^{\aug}$ the full subcategory spanned by artinian algebras.
	
	\item Furthermore, we abuse the notation to write $\Sp(\CAlg^{\art})$ for the full subcategory of $\Sp(\CAlg^{\aug})$ spanned by the spectrum objects $X : \NN \times \NN \lrar \CAlg^{\aug}$ such that $X_{m,n}$ is artinian for every pair $(m,n) \in \NN \times \NN$. 
	
	\item We let $\ModCat$ denote the category whose objects are model categories and whose morphisms are Quillen adjunctions with the sources and targets being those of the left Quillen functors.  
	
	\item We are interested in functors $\F : \CAlg^{\art} \lrar \ModCat$. For each map $f:R\lrar S$ in $\CAlg^{\art}$, we denote by $f_! : \adjunction*{}{\F(R)}{\F(S)}{} : f^{*}$  the image of $f$ under the functor $\F$.

	\item Finally, we denote by $\textbf{S} := (\Set_\Delta)_\infty$ the $\infty$-category of spaces and by $\CAlg^{\art}_\infty$ the full $\infty$-subcategory of $\CAlg^{\aug}_\infty$ spanned by artinian $\textbf{k}$-algebras.
\end{enumerate}

Let us first recall the following definition from \cite{Luriefmp}.

\begin{dfn}\label{d:fmp} A \textbf{formal moduli problem} is a functor $\G:\CAlg^{\art}_\infty\longrightarrow \textbf{S}$ such that the space $\G(\textbf{k})$ is contractible, and if $\sigma$ is a Cartesian square in $\CAlg^{\art}_\infty$ of the form
	$$ \xymatrix{
		R \ar[r]\ar[d] & S \ar[d] \\
		T \ar[r] & U \\
	}$$
	with the induced maps $\pi_0(S) \longrightarrow \pi_0(U)$ and $\pi_0(T) \longrightarrow \pi_0(U)$ surjective, then $\G(\sigma)$ is a Cartesian square in $\textbf{S}$ as well.
\end{dfn}

We arrive at the expected definitions.

\begin{define}\label{d:moduli}
	A \textbf{formal moduli context} is a functor $\F:\CAlg^{\art} \lrar \ModCat$ satisfying the following properties: 
	\begin{enumerate}[(1)]
		\item
		$\F$ sends weak equivalences in $\CAlg^{\art}$ to Quillen equivalences, and for every morphism $f: R \lrar S$ in $\CAlg^{\art}$ the right adjoint $f^*: \F(S) \lrar \F(R)$ preserves weak equivalences.
		\item
		For every homotopy Cartesian square
		$$ \xymatrix{
			R \ar[r]\ar[d] & S \ar[d] \\
			T \ar[r] & U \\
		}$$
		in $\CAlg^{\art}$ such that the maps $\pi_0(S) \lrar \pi_0(U)$ and $\pi_0(T) \lrar \pi_0(U)$ are surjective, the corresponding diagram
		$$ \xymatrix{
			\F(R) \ar[r]\ar[d] & \F(S) \ar[d] \\
			\F(T) \ar[r] & \F(U) \\
		}$$
		of model categories is homotopy Cartesian (in the sense of Definition~\ref{d:car}).
	\end{enumerate}
	
\end{define}

We will see interesting examples of formal moduli contexts in $\S$\ref{sub:examples}. For now, let us proceed with the abstract theory.


\begin{define}\label{d:defor}
	Suppose given a formal moduli context $\F: \CAlg^{\art} \lrar \ModCat$. Let $X \in \F(\textbf{k})$ be a fibrant object, and $R \in\CAlg^{\art}$ an artinian $\textbf{k}$-algebra with augmentation $\eps: R \lrar \textbf{k}$. 
	\begin{enumerate}
		\item A \textbf{deformation of $X$ over} $R$ is a pair $(Y,\eta)$ with $Y\in\F(R)$ cofibrant and $\eta: \eps_!Y \lrarsimeq X$ a weak equivalence in $\F(\textbf{k})$. 
		
		\item Given two deformations $(Y,\eta)$ and $(Y',\eta')$ of $X$ over $R$, an \textbf{equivalence} from $(Y,\eta)$ to $(Y',\eta')$ is a weak equivalence $Y \lrarsimeq Y'$ in $\F(R)$ that is compatible with the structure maps $\eta$ and $\eta'$ in the evident way.
		
		\item The \textbf{space of deformations of $X$ over $R$}, denoted $\Def(X,R)$, is the Kan replacement of the nerve of the category whose
		objects are deformations of $X$ over $R$ and whose morphisms are equivalences of deformations. We regard $\Def(X,R)$ as a pointed space, with base point given by the \textbf{trivial deformation} $(u_!X,\eta) \in \Def(X,R)$ where $u: \textbf{k} \lrar R$ is the unit of $R$ and $\eta: \eps_!u_!X \x{\cong}{\lrar} X$ is the canonical isomorphism.
	\end{enumerate}
\end{define}

\begin{notn}\label{no:maxi} For an $\infty$-category $\C$, we denote by $\C^{\simeq}$ the \textbf{maximal $\infty$-subgroupoid} of $\C$. It is known that the assignment $\C \mapsto \C^{\simeq}$ determines an $\infty$-categorical right adjoint from $\infty$-categories to $\infty$-groupoids (cf. \cite{Luriehtt}).
\end{notn}

\begin{rem}\label{r:defor} Under the setting of Definition \ref{d:defor}, let $\widetilde{\eps}_! : \F(R)_{\infty}^{\simeq} \lrar \F(\textbf{k})_{\infty}^{\simeq}$  denote the map induced by $\eps: R \lrar \textbf{k}$. The space $\Def(X,R)$ can then be identified with the homotopy pullback
	$$ (\F(R)_{\infty}^{\simeq})_{/X} := \F(R)_{\infty}^{\simeq} \times^{\h}_{\F(\textbf{k})_{\infty}^{\simeq}} (\F(\textbf{k})_{\infty}^{\simeq})_{/X}.$$
(This follows from the fact that section model categories have the right type, as discussed in Appendix~\ref{s:hocartmodel}.) In particular, $\Def(X,R)$ is weakly equivalent to the homotopy fiber over $X$ of the map $\widetilde{\eps}_! : \F(R)_{\infty}^{\simeq} \lrar \F(\textbf{k})_{\infty}^{\simeq}$. Combining this with Definition \ref{d:moduli}(1), we deduce that the assignment $R \mapsto \Def(X,R)$ determines an  $\infty$-functor  $\CAlg^{\art}_\infty\longrightarrow \textbf{S}$. 
\end{rem}

The main interest in the notion of a formal moduli context is that it naturally gives rise to formal moduli problems.
\begin{prop}\label{p:moduli}
	Let $\F: \CAlg^{\art} \lrar \ModCat$ be a formal moduli context and let $X \in \F(\textbf{k})$ be a fibrant object. Then the functor $$ \Def(X,-) : \CAlg^{\art}_\infty\longrightarrow \textbf{S}, \,  R \mapsto \Def(X,R) $$ is a formal moduli problem.
\end{prop}
\begin{proof} We need to verify that $\Def(X,\textbf{k})$ is weakly contractible, and that $\Def(X,-)$ preserves the type of Cartesian squares appearing in Definition~\ref{d:fmp}. The first claim immediately follows from Remark~\ref{r:defor}.

	We now verify the second claim, again using the homotopy pullback representing $\Def(X,-)$ from  Remark~\ref{r:defor}. Suppose given a homotopy Cartesian square in $\CAlg^{\art}$ of the form
	$$ \xymatrix{
		R \ar[r]\ar[d] & S \ar[d] \\
		T \ar[r] & U \\
	}$$
	such that the two maps $\pi_0(S) \lrar \pi_0(U)$ and $\pi_0(T) \lrar \pi_0(U)$ are surjective. By assumption, there is a homotopy Cartesian square of model categories
	$$ \xymatrix{
		\F(R) \ar[r]\ar[d] & \F(S) \ar[d] \\
		\F(T) \ar[r] & \F(U).  \\
	}$$
	By Remark~\ref{r:right}, this induces a homotopy Cartesian square of the underlying $\infty$-categories, and hence, we further obtain a homotopy Cartesian square of $\infty$-groupoids
	\begin{equation}\label{e:car}
		\xymatrix{
			\F(R)_{\infty}^{\simeq} \ar[r]\ar[d] & \F(S)_{\infty}^{\simeq} \ar[d] \\
			\F(T)_{\infty}^{\simeq} \ar[r] & \F(U)_{\infty}^{\simeq}.  \\
		}
	\end{equation}
	By the commutation of homotopy limits with homotopy limits, we conclude that the square of $\infty$-groupoids
	$$
	\xymatrix{
		(\F(R)_{\infty}^{\simeq})_{/X} \ar[r]\ar[d] & (\F(S)_{\infty}^{\simeq})_{/X} \ar[d] \\
		(\F(T)_{\infty}^{\simeq})_{/X} \ar[r] & (\F(U)_{\infty}^{\simeq})_{/X}  \\
	}
	$$
	is homotopy Cartesian as well. The proof is therefore completed.
\end{proof}

\subsection{Main theorem}\label{sub:fmpandQcohom}

We begin by introducing the following significant concept.

\begin{define}\label{d:defor-first-order} Let $\F: \CAlg^{\art} \lrar \ModCat$ be a formal moduli context,  $X \in \F(\textbf{k})$ a fibrant object, and $M \in \Sp(\CAlg^{\art})$ an $\Om$-spectrum. (Then the structure map $\Om^{\infty}M \lrar \textbf{k}$ exhibits $\Om^{\infty}M$ as an artinian $\textbf{k}$-algebra.) We  denote by $$\Def(X,M) := \Def(X,\Om^{\infty}M),$$ and refer to it as the \textbf{space of first order deformations of $X$ in direction} $M$. 
\end{define}

\begin{conv}\label{conv:tmain} In the rest of this subsection, we let $\F: \CAlg^{\art} \lrar \ModCat$ be a fixed formal moduli context, and assume that $X \in \F(\textbf{k})$ is such that the tangent category $\T_X \F(\textbf{k})$ exists (at least) as a semi-model category.
\end{conv}

\begin{rem} Let $R \in \CAlg^{\art}$ with unit $u : \textbf{k} \lrar R$ and augmentation $\varepsilon : R \lrar \textbf{k}$. The canonical natural transformation $u_! \lrar \varepsilon^*$ induces a natural transformation $u^*u_! \lrar \Id \cong u^*\varepsilon^*$, which forms a retraction of the unit $\Id \lrar u^*u_!$. Consequently, we obtain a natural factorization $X \lrar u^*u_!X \lrar X$ of $\Id_X$. Accordingly, we make the following definition.
\end{rem}

\begin{define}\label{d:M(X)}
Given a spectrum object $M: \NN \times \NN \lrar \CAlg^{\art}$, we denote by $M(X) \in \T_X \F(\textbf{k})$ the spectrum object with
	$$ M(X)_{n,m} := u_{n,m}^*(u_{n,m})_!X \in \F(\textbf{k})_{X//X} $$ 
	where $u_{n,m}: \textbf{k} \lrar M_{n,m}$ refers to the unit of $M_{n,m} \in \CAlg^{\art}$.   
\end{define}

\begin{lem}\label{l:specMX}
	Suppose further that $X \in \F(\textbf{k})$ is cofibrant and that $M \in \Sp(\CAlg^{\art})$ is an $\Om$-spectrum. Then $M(X)\in \T_X \F(\textbf{k})$ is an $\Om$-spectrum as well.
\end{lem}
\begin{proof}
	First, we show that $M(X)$ is a prespectrum. For $(n,m) \in \NN \times \NN$ with $n \neq m$, we need to prove that the unit map $X \lrar u_{n,m}^*(u_{n,m})_!X$ is a weak equivalence. Since $M$ is in particular a prespectrum, the unit map $u_{n,m}: \textbf{k} \lrar M_{n,m}$ is a weak equivalence. Therefore, the adjunction $$(u_{n,m})_! : \adjunction*{}{\F(\textbf{k})}{\F(M_{n,m})}{} : (u_{n,m})^*$$ is a Quillen equivalence, and $(u_{n,m})^*$ preserves weak equivalences (cf. Definition \ref{d:moduli}(1)). These facts, together with the cofibrancy of $X$, prove that $X \lrar u_{n,m}^*(u_{n,m})_!X$ is indeed a weak equivalence.

	It remains to show that for each $n \geq 0$, the following square in $\F(\textbf{k})$ is homotopy Cartesian:
	\begin{equation}\label{e:diag}
		\xymatrix{
			u_{n,n}^*(u_{n,n})_!X \ar[r]\ar[d] & u_{n,n+1}^*(u_{n,n+1})_!X \ar[d] \\
			u_{n+1,n}^*(u_{n+1,n})_!X \ar[r] & u_{n+1,n+1}^*(u_{n+1,n+1})_!X.  \\
		}
	\end{equation}
 For $i, j \in \{n,n+1\}$, we write $\vphi_{i,j} : M_{n,n} \lrar M_{i,j}$ for the map induced by $(n,n) \lrar (i,j)$ in $\NN \times \NN$. Since $u_{i,j} = \vphi_{i,j} \circ u_{n,n}$, we have $u_{i,j}^{*}=u_{n,n}^{*}\circ \vphi_{i,j}^{*}$ and $(u_{i,j})_! = (\vphi_{i,j})_! \circ (u_{n,n})_!$. Hence, the square \eqref{e:diag} agrees with the image under $u_{n,n}^{*}$ of the following square in $\F(M_{n,n})$:
	
	\begin{equation}\label{e:diag-2}
		\xymatrix{
			X' \ar[r]\ar[d] & \vphi_{n,n+1}^*(\vphi_{n,n+1})_!X' \ar[d] \\
			\vphi_{n+1,n}^*(\vphi_{n+1,n})_!X' \ar[r] & \vphi_{n+1,n+1}^*(\vphi_{n+1,n+1})_!X'  \\
		}
	\end{equation}
	where $X':=(u_{n,n})_!X$. Therefore, it suffices to verify that this square is homotopy Cartesian. By assumption, the following square in $\CAlg^{\art}$
	\begin{equation}\label{e:diag-3}
		\xymatrix{
			M_{n,n} \ar[r]\ar[d] & M_{n,n+1} \ar[d] \\
			M_{n+1,n} \ar[r] & M_{n+1,n+1}  \\
		}
	\end{equation}
	is homotopy Cartesian, with $M_{n,n+1}\simeq M_{n+1,n} \simeq \textbf{k}$. Besides, note that $$\Sp(\CAlg^{\aug}) \x{\defi}{=}  \Sp(\Alg^{\aug}_{\Com}(\C_{\geq 0}(\textbf{k}))) \simeq \Sp(\Alg^{\aug}_{\Com}(\C(\textbf{k})))$$ (both are Quillen equivalent to $\C(\textbf{k})$). Accordingly, we may assume that $M$ arises from an $\Om$-spectrum in $\Sp(\Alg^{\aug}_{\Com}(\C(\textbf{k})))$, so that the underlying square of chain complexes of \eqref{e:diag-3} is also homotopy coCartesian. Moreover, by definition, there exists  $\ovl{M}_{n,n}\in\C_{\geq 0}(\textbf{k})$ such that $M_{n,n} \cong \textbf{k} \oplus \ovl{M}_{n,n}$ as chain complexes. Consequently,  $M_{n+1,n+1} \simeq \textbf{k} \oplus \ovl{M}_{n,n}[1]$ as chain complexes. Thus, the maps $M_{n,n+1}\lrar M_{n+1,n+1}$ and $M_{n+1,n}\lrar M_{n+1,n+1}$ induce  isomorphisms on $\pi_0$. Due to this, we obtain a homotopy Cartesian square of model categories:
	$$ \xymatrix{
		\F(M_{n,n}) \ar[r]\ar[d] & \F(M_{n,n+1}) \ar[d] \\
		\F(M_{n+1,n}) \ar[r] & \F(M_{n+1,n+1})  \\
	}$$
(cf. Definition \ref{d:moduli}(2)). Finally, by Remark~\ref{r:car} we deduce that the square \eqref{e:diag-2} is indeed homotopy Cartesian.
\end{proof}

\begin{rem}\label{r:specMX} Suppose $\F$ further satisfies that, for every $R\in\CAlg^{\art}$ with unit  $u : \textbf{k} \lrar R$, the induced functor $u_! : \F(\textbf{k}) \lrar \F(R)$ preserves weak equivalences. Then, following the proof of Lemma~\ref{l:specMX}, $M(X)$ is already an $\Omega$-spectrum for every $X \in \F(\textbf{k})$ and every $\Om$-spectrum $M \in \Sp(\CAlg^{\art})$. For instance, this condition is satisfied by the functors considered in Propositions \ref{p:example-1}, \ref{p:example-2}, \ref{p:example-3}, and \ref{p:example-4}.
\end{rem}

We are now in position to prove the main theorem of this section.
\begin{thm}\label{t:main} Suppose further that $X \in \F(\textbf{k})$ is \textit{bifibrant} (i.e., both fibrant and cofibrant), and let $M \in \Sp(\CAlg^{\art})$ be an $\Om$-spectrum. Then there is a weak equivalence of spaces
	$$  \Def(X,M) \simeq \HHQ^\star(X ; M(X)[1]) $$
(see Definition~\ref{d:defor-first-order} for notation). Consequently, there is a canonical isomorphism 
$$ \pi_0 \Def(X,M) \cong \HHQ^1(X;M(X)),$$ 
and for every $n\geq1$, there is a canonical isomorphism 
$$ \pi_n(\Def(X,M),*) \cong \HHQ^{1-n}(X;M(X))$$
where $*$ refers to any base point.
\end{thm}
\begin{proof} By Remark~\ref{r:modelQ}, it suffices to prove that $\Def(X,M)$ is weakly equivalent to the derived mapping space $\Map^{\h}_{\F(\textbf{k})_{/X}}(X,\Om_+^{\infty-1}M(X))$. 
	
	 First, by assumption, we have $M_{1,0} \simeq M_{0,1} \simeq \textbf{k}$, $M_{0,0} \simeq \Om^{\infty}M$ and $M_{1,1} \simeq \Om^{\infty}M[1]$. Moreover, the maps $M_{1,0} \lrar M_{1,1}$ and $M_{0,1} \lrar M_{1,1}$ induce isomorphisms on $\pi_0$ (see the proof of Lemma~\ref{l:specMX}). We simply denote by $u_M : \textbf{k} \lrar \Om^{\infty}M[1]$ the unit, and by $\eps_M:\Om^{\infty}M[1] \lrar \textbf{k}$ the augmentation of $\Om^{\infty}M[1]$. 
	As in the proof of Proposition~\ref{p:moduli}, there is a homotopy Cartesian square of  $\infty$-groupoids:
	$$ \xymatrix{
		\F(\Om^{\infty}M)_{\infty}^{\simeq} \ar[r]\ar[d] & \F(\textbf{k})_{\infty}^{\simeq} \ar^{(\wtl{u}_M)_!}[d] \\
		\F(\textbf{k})_{\infty}^{\simeq} \ar[r]_-{(\wtl{u}_M)_!} & \F(\Om^{\infty}M[1])_{\infty}^{\simeq}.  \\
	}$$
Combining this with Remark~\ref{r:defor}, we deduce that $\Def(X,M)$ is weakly equivalent to the homotopy fiber of the right vertical map over $(\wtl{u}_M)_!X$. Next, consider a diagram of $\infty$-groupoids of the form
	$$ \xymatrix{
		\ast \ar^-{\{X\}}[r]\ar^{\iota}[d] & \F(\textbf{k})_{\infty}^{\simeq} \ar^{(\wtl{u}_M)_!}[d] \\
		Z \ar[r]\ar[d] & \F(\Om^{\infty}M[1])_{\infty}^{\simeq} \ar^{(\wtl{\eps}_M)_!}[d]  \\
		\ast \ar^-{\{X\}}[r] & \F(\textbf{k})_{\infty}^{\simeq}  \\
	}$$
	with $Z$ chosen in such a way that the bottom square is homotopy Cartesian. In particular, the top square is homotopy Cartesian as well. Hence, we may identify $\Def(X,M)$ with the loop space $\Om_{\iota(\ast)}Z$. It follows that $\Def(X,M)$ is weakly equivalent to the homotopy fiber (over the constant loop at $X$) of the map between loop spaces
	\begin{equation}\label{eq:hofibdef}
		\Om_{(\wtl{u}_M)_!X} \, \F(\Om^{\infty}M[1])_{\infty}^{\simeq} \lrar \Om_{X} \, \F(\textbf{k})_{\infty}^{\simeq}
	\end{equation}
	induced by $(\wtl{\eps}_M)_!$. Let $\K(X,M)$ denote this homotopy fiber. It remains to show that $\K(X,M)$ is weakly equivalent to  $\Map^{\h}_{\F(\textbf{k})_{/X}}(X,\Om_+^{\infty-1}M(X))$.

	 It is known that the map \eqref{eq:hofibdef} is weakly  equivalent to the map $$\Aut^{\h}_{\F(\Om^{\infty}M[1])}((u_{M})_!X,(u_{M})_!X) \lrar \Aut^{\h}_{\F(\textbf{k})}(X,X)$$ where $\Aut^{\h}_{\F(\textbf{k})}(X,X) \subseteq \Map^{\der}_{\F(\textbf{k})}(X,X)$ denotes the subspace of self-equivalences on $X$, and the domain is defined analogously (cf., e.g. \cite{DK}). Moreover, by the same  argument as in the proof of Lemma~\ref{l:specMX}, the induced square of $\infty$-categories
	  $$ \xymatrix{
	  	\F(\Om^{\infty}M[1])_{\infty} \ar[r]\ar[d] & \F(\textbf{k})_{\infty} \ar[d] \\
	  	\F(\textbf{k})_{\infty} \ar[r] & \F(\Om^{\infty}M[2])_{\infty}  \\
	  }$$
  is Cartesian. This implies that the functor $\F(\Om^{\infty}M[1])_{\infty}  \lrar \F(\textbf{k})_{\infty}$ detects equivalences. We therefore obtain a homotopy Cartesian square of spaces:  
	 \begin{equation}\label{eq:KXM}
	 	\xymatrix{
	 		\Aut^{\h}_{\F(\Om^{\infty}M[1])}((u_{M})_!X,(u_{M})_!X) \ar[r]\ar[d] &  \Map^{\der}_{\F(\Om^{\infty}M[1])}((u_M)_!X,(u_M)_!X) \ar[d] \\
	 		\Aut^{\h}_{\F(\textbf{k})}(X,X) \ar[r] & \Map^{\der}_{\F(\textbf{k})}(X,X).  \\
	 	}
	 \end{equation}
Thus, $\K(X,M)$ is weakly equivalent to the homotopy fiber over $\Id_X$ of the right vertical map. Furthermore, by the Quillen adjunction $(u_M)_! \dashv u^{*}_{M}$, this map is weakly equivalent to the map
\begin{equation}\label{eq:tmain1}
	\Map^{\der}_{\F(\textbf{k})}(X,u^{*}_{M}(u_M)_!X) \lrar \Map^{\der}_{\F(\textbf{k})}(X,X)
\end{equation}
induced by $u^{*}_{M}(u_M)_!X \lrar X$. Finally, since $X$ is fibrant, the homotopy fiber of the map \eqref{eq:tmain1} over $\Id_X$  can be identified with the derived mapping space 
$$ \Map^{\der}_{\F(\textbf{k})_{/X}}(X,u^{*}_{M}(u_M)_!X) \; \x{\defi}{=} \;  \Map^{\der}_{\F(\textbf{k})_{/X}}(X,\Om^{\infty-1}_+M(X)).$$  
The proof is therefore completed.
\end{proof}

\begin{rem} In practice, Theorem~\ref{t:main} usually remains valid for $X$ even without assuming (co)fibrancy, as discussed in Remark \ref{r:refinemain}.
\end{rem}

\subsection{Examples of formal moduli contexts}\label{sub:examples}

In this subsection, we describe some interesting examples of formal moduli contexts. Throughout this process, we need to verify the two conditions of Definition~\ref{d:moduli}.

\begin{notn} For any $R \in \CAlg^{\art}$, we denote by $\Mod_R := \Mod_R(\C_{\geq 0}(\textbf{k}))$ the category of $R$-modules in $\C_{\geq 0}(\textbf{k})$.
\end{notn}

 The following lemma will be useful.

\begin{lem}\label{l:examlem} Suppose given a map $R \lrar S$ in $\CAlg^{\art}$ that is surjective on $\pi_0$. Then any object $K\in\Mod_R$ is acyclic whenever $K\otimes_R S$ is.
\begin{proof} Assume that $K\otimes_R S$ is acyclic, and suppose for contradiction that $K$ is not. Let $k\geq0$ be the smallest integer such that $\pi_k K$ is nontrivial, and set $\mathfrak{a} := \ker(\pi_0(R) \lrar \pi_0(S))$. We then obtain $$\pi_k K/\mathfrak{a}(\pi_k K) \cong \pi_k K \otimes_{\pi_0R} \pi_0S \cong \pi_k(K\otimes_{R} S) = 0.$$
Now, the equation $\pi_k K = \mathfrak{a}(\pi_k K)$ leads to the contradiction that $\pi_k K =0$, since every ideal of a local artinian algebra is nilpotent. 
\end{proof}
\end{lem}

\begin{pro}\label{p:example-1} The functor
	\begin{equation}\label{FMod}
		\F_{\Mod}: \CAlg^{\art} \lrar \ModCat, \; R \mapsto \Mod_R
	\end{equation}
is a formal moduli context.
\end{pro}
\begin{proof} (1) We begin by verifying the first condition of Definition~\ref{d:moduli}.  Let $f : R \lrar S$ be a map in $\CAlg^{\art}$. Then  $f_! : \adjunction*{}{\Mod_R}{\Mod_S}{} : f^*$ is the adjunction of induction-restriction functors.  Clearly, $f^*$ preserves weak equivalences. Moreover, since $\textbf{k}$ is a field, the underlying complexes of $R$ and $S$ are automatically cofibrant in $\C_{\geq 0}(\textbf{k})$. Therefore, if $f$ is a weak equivalence, the adjunction $f_! \dashv f^*$ is indeed a Quillen equivalence.  
	
	\smallskip
	
(2)	We will now verify the second condition of Definition~\ref{d:moduli}. Let
	\begin{equation}\label{e:square}
		\xymatrix{
			R \ar[r]\ar[d] & S \ar^{p}[d] \\
			T \ar^{q}[r] & U \\
		}
	\end{equation}
	be a homotopy Cartesian square in $\CAlg^{\art}$ with $p$ and $q$ surjective on $\pi_0$. We need to show that the corresponding square of model categories
	\begin{equation}\label{e:square-2}
		\xymatrix{
			\Mod_R \ar[r]\ar[d] & \Mod_S \ar[d] \\
			\Mod_T \ar[r] & \Mod_U \\
		}
	\end{equation}
	is homotopy Cartesian in the sense of Definition \ref{d:car}. Let $\F': [1] \times [1] \lrar \ModCat$ represent the square \eqref{e:square-2}, and let $\I \subseteq [1] \times [1]$ denote the full subcategory spanned by the vertices $(0,1), (1,0)$ and $(1,1)$. We need to prove that the adjunction $ \L^{\coc} : \adjunction*{}{\Mod_R}{\Sec^{\coc}(\F'|_{\I})}{} : \R^{\coc}$ is a Quillen equivalence. For this, it suffices to show that $\RR\R^{\coc}$ detects weak equivalences (between cofibrant objects) and that, for every cofibrant object $X \in \Mod_R$, the derived unit  $X \lrar \RR \R^{\coc} \L^{\coc}(X)$ is a weak equivalence. 
	
	\smallskip
	
(2a) We first prove the second claim. By Remark~\ref{r:car}, we just need to show that the square
	\begin{equation}\label{e:X}
		\xymatrix{
			X \ar[r]\ar[d] & X \otimes_R S \ar[d] \\
			X \otimes_R T \ar[r] & X \otimes_R U \\
		}
	\end{equation}
	is homotopy Cartesian in $\Mod_R$. In fact, it suffices to prove that this square is homotopy coCartesian in $\Mod_R$. Viewing the square \eqref{e:X} as the image of \eqref{e:square} under the left Quillen functor $X \otimes_R (-) : \Mod_R \lrar \Mod_R$, it  remains to prove that \eqref{e:square} is  homotopy coCartesian in $\Mod_R$. (Since $X$ is cofibrant, the functor $X \otimes_R (-)$ preserves weak equivalences.) Observe that the square \eqref{e:square} is already homotopy Cartesian in $\Mod_R(\C(\textbf{k}))$, due to the fact that a map in $\Mod_R(\C_{\geq 0}(\textbf{k}))$ is an epimorphism whenever it is surjective at every positive degree and surjective on $\pi_0$. Thus, the square \eqref{e:square} is indeed homotopy coCartesian in $\Mod_R$, since the category $\Mod_R(\C(\textbf{k}))$ is stable and the embedding $ \Mod_R \lrar \Mod_R(\C(\textbf{k}))$ detects homotopy coCartesian squares.

\smallskip

(2b) We now verify that $\RR\R^{\coc}: \Sec^{\coc}(\F'|_{\I}) \lrar \Mod_R$ detects weak equivalences between cofibrant objects. Note first that a map in $\Sec^{\coc}(\F'|_{\I})$ is a weak equivalence if and only if its cofiber is a \textbf{weak} \textbf{zero object} (i.e., a section whose entries are weak zero objects). Indeed, the property of being a weak equivalence and the formation of homotopy cofibers are both jointly created by the projections $\Sec^{\coc}(\F'|_{\I}) \lrar \Mod_S$, $\Sec^{\coc}(\F'|_{\I}) \lrar \Mod_T$ and $\Sec^{\coc}(\F'|_{\I}) \lrar \Mod_U$, and the characterization of weak equivalences via cofibers holds in any module category in $\C_{\geq 0}(\textbf{k})$.  It will hence suffice to show that $\RR\R^{\coc}$ detects weak zero cofibrant objects. By definition, a cofibrant section $s\in\Sec^{\coc}(\F'|_{\I})$ consists of a triple $(X_S,X_T,X_U)\in \Mod_S\times\Mod_T\times\Mod_U$ of cofibrant objects together with maps $\vphi: X_T \otimes_T U \lrar X_U$ and $\psi: X_S \otimes_S U \lrar X_U$ in $\Mod_U$. When $s$ is moreover fibrant, these maps are weak equivalences. To summarize, it suffices to show that if $s=(X_S,X_T,X_U)$ is bifibrant and the homotopy pullback $X_T \times^{\der}_{X_U} X_S$ is acyclic, then each of $X_S,X_T$, and $X_U$ is acyclic as well.

Assume henceforth that $X_T \times^{\der}_{X_U} X_S$ is acyclic. If $X_T$ is acyclic, then the weak equivalence $\vphi : X_T \otimes_T U \lrarsimeq X_U$, together with the cofibrancy of $X_T\in\Mod_T$, implies that $X_U$ is acyclic. Hence, by the weak equivalence $\psi: X_S \otimes_S U \lrarsimeq X_U$ and  Lemma \ref{l:examlem},  $X_S$ is also acyclic. Similarly, if $X_S$ is acyclic then so are $X_T$ and $X_U$. Therefore, it suffices to verify that either $X_T$ or $X_S$ is acyclic. Let us assume by contradiction that neither of them is acyclic, and let $l \geq0$ be the smallest integer such that at least one of $\pi_{l}(X_T)$ or $\pi_{l}(X_S)$ is nontrivial. Then $\pi_{l}(X_T \otimes_T U) \cong \pi_{l}(X_T) \otimes_{\pi_0(T)} \pi_0(U)$. Thus, the map $\pi_{l}(X_T) \lrar \pi_{l}(X_U)$, which coincides with the composite map
	$$ \pi_{l}(X_T) \lrar \pi_{l}(X_T) \otimes_{\pi_0(T)} \pi_0(U) \x{\pi_{l}(\vphi)}{\underset{\cong}{\lrar}}  \pi_{l}(X_U),$$
	is surjective. Similarly, the map $\pi_l(X_S) \lrar \pi_l(X_U)$ is surjective as well. On the other hand, since $X_T \times^{\der}_{X_U} X_S$ is acyclic, the induced map $f_l:\pi_l(X_T) \oplus \pi_l(X_S) \lrar \pi_l(X_U)$ is injective. These observations imply that $\pi_l(X_T)=\pi_l(X_S)=0$, a contradiction.
\end{proof}

For a map $f : R \lrar R'$ in $\CAlg^{\art}$, the induced adjunction  $\adjunction*{}{\Mod_R}{\Mod_{R'}}{}$ satisfies that the left (resp. right) adjoint is symmetric monoidal (resp. lax symmetric monoidal). It hence gives rise to a Quillen adjunction $f_! : \adjunction*{}{\Alg_{\P}(\Mod_R)}{\Alg_{\P}(\Mod_{R'})}{} : f^*$. 
\begin{pro}\label{p:example-2}
	Let $\P$ be an operad in $\C_{\geq 0}(\textbf{k})$. Then the functor 
	\begin{equation}\label{FAlg}
		\F_{\Alg}: \CAlg^{\art} \lrar \ModCat, \; R \mapsto \Alg_{\P}(\Mod_R)
	\end{equation}
is a formal moduli context.
\end{pro}
\begin{proof} Let $f : R \lrar R'$ be a map in $\CAlg^{\art}$. The right adjoint $f^*$ is simply the restriction functor, and therefore preserves weak equivalences. Moreover, as in the above proof, if $f : R \lrarsimeq R'$ is a weak equivalence, then $\adjunction*{}{\Mod_R}{\Mod_{R'}}{}$ is a Quillen equivalence, and hence the adjunction $f_! : \adjunction*{}{\Alg_{\P}(\Mod_R)}{\Alg_{\P}(\Mod_{R'})}{} : f^*$ is also a Quillen equivalence. Thus, we have verified the first condition of Definition~\ref{d:moduli}.
	
Moreover, note that the induced adjunction $f_! : \adjunction*{}{\Alg_{\P}(\Mod_R)}{\Alg_{\P}(\Mod_{R'})}{} : f^*$ is given on the underlying modules by the same adjunction $\adjunction*{}{\Mod_R}{\Mod_{R'}}{}$. This  allows us to apply the proof of Proposition~\ref{p:example-1} to verify the second condition of  Definition~\ref{d:moduli}. 
\end{proof}

Let $f : R \lrar R'$ be a map in $\CAlg^{\art}$. The induced adjunction $\adjunction*{}{\Mod_R}{\Mod_{R'}}{}$ gives rise to an adjunction between enriched categories:
\begin{equation}\label{eq:adjexample-3}
	f_! : \adjunction*{}{\Cat({\Mod_R})}{\Cat({\Mod_{R'}})}{} : f^*.
\end{equation}
Concretely, for each $\C\in\Cat({\Mod_R})$, the category $\C\otimes_{R} R' :=  f_!(\C) \in\Cat({\Mod_{R'}})$ has the same objects as $\C$, while its mapping objects are obtained by applying the functor $(-)\otimes_R R'$ to those of $\C$. Similarly, for  $\C'\in\Cat({\Mod_{R'}})$, $f^*(\C')\in\Cat({\Mod_{R}})$ has the same objects as $\C'$, with mapping objects induced by  restriction along $f$. For brevity, we again denote this category by $\C'$.

\begin{pro}\label{p:example-3} The functor
	\begin{equation}\label{FCat}
\F_{\Cat} : \CAlg^{\art} \lrar \ModCat, \; R \mapsto  \Cat({\Mod_R})
	\end{equation}
is a formal moduli context.
\end{pro}

The proof requires a technical lemma. A map $\alpha:\C\lrar\D$ in $\Cat({\Mod_R})$ is called an \textbf{isofibration} if the induced functor $\Ho(\alpha):\Ho(\C)\lrar \Ho(\D)$ between homotopy categories (cf. $\S$\ref{s:bases}) is an isofibration in the usual sense. 
\begin{lem}\label{l:iso}
	Let $p: R \lrar S$ be a map in $\CAlg^{\art}$ that is surjective on $\pi_0$, and let $\C\in\Cat({\Mod_{R}})$ be a levelwise cofibrant $\Mod_R$-enriched category. Then the induced map $\C \lrar \C\otimes_{R}S$ is an isofibration in $\Cat({\Mod_{R}})$. More generally, for a chain $R\lrar S \lrar U$ of maps in $\CAlg^{\art}$ such that the map $S \lrar U$ is surjective on $\pi_0$, the induced map $\C \otimes_R S\lrar\C \otimes_R U $ is an isofibration in $\Cat({\Mod_{R}})$.
\end{lem}
\begin{proof} First, by the free-forgetful adjunction $\adjunction*{}{\C_{\geq 0}(\textbf{k})}{\Mod_R}{}$,  the homotopy category of a $\Mod_R$-enriched category coincides with that of its underlying $\C_{\geq 0}(\textbf{k})$-enriched category. Unwinding the definitions, the functor $\Ho(\C) \lrar \Ho(\C \otimes_R S)$ is the identity on objects, and for each pair $(x,y)$ of objects, the induced map $\Hom_{\Ho(\C)}(x,y) \lrar \Hom_{\Ho(\C \otimes_R S)}(x,y)$ is identified with the projection
	$$ \pi_0\Map_\C(x,y) \lrar \pi_0\Map_\C(x,y)\otimes_{\pi_0R} \pi_0S \cong \pi_0\Map_\C(x,y)/\mathfrak{a}(\pi_0\Map_\C(x,y))$$
where  $\mathfrak{a} := \ker(\pi_0(R) \lrar \pi_0(S))$. Therefore, it suffices to show that any map $f: x \lrar y$ in $\Ho(\C)$  is an isomorphism whenever the corresponding map $\ovl{f} : x \lrar y$ in $\Ho(\C \otimes_R S)$ is one. Moreover, $f$ is an isomorphism if and only if, for every object $z\in\Ob(\C)$, the induced map $f\circ (-):\Map_\C(z,x) \lrar \Map_\C(z,y)$ is a weak equivalence. (An analogous statement holds for $\ovl{f}$.) Therefore, it remains to show that $f\circ (-)$ is a weak equivalence as soon as the induced map $$\Map_\C(z,x) \otimes_{R} S \lrar \Map_\C(z,y)\otimes_{R} S$$ is one. Now, since the functor $(-)\otimes_{R} S$ preserves homotopy cofibers of maps between cofibrant $R$-modules, the problem reduces to proving that any $M \in \Mod_R$ is acyclic whenever $M\otimes_{R} S$ is. This  follows from Lemma \ref{l:examlem}. The second claim is verified in the same manner.
\end{proof}

\begin{proof}[\underline{Proof of Proposition~\ref{p:example-3}}] 
	
	As before, the first condition of Definition~\ref{d:moduli} is readily verified. We next verify the second by showing that for a square of the type \eqref{e:square}, the induced square of model categories
	\begin{equation}\label{e:square-4}
		\xymatrix{
			\Cat({\Mod_R}) \ar[r]\ar[d] & \Cat({\Mod_S}) \ar[d]^{p_!} \\
			\Cat({\Mod_T}) \ar[r]_{q_!} & \Cat({\Mod_U}) \\
		}
	\end{equation}
	is homotopy Cartesian. Let $\F': [1] \times [1] \lrar \ModCat$ represent the square~\eqref{e:square-4} and let $\I \subseteq [1] \times [1]$ denote the full subcategory spanned by $(0,1),(1,0)$ and $(1,1)$. We need to show that the adjunction $\L^{\coc} : \adjunction*{}{\Cat({\Mod_R})}{\Sec^{\coc}(\F'|_{\I})}{} : \R^{\coc}$ is a Quillen equivalence. We will prove that $\R^{\coc}$ detects weak equivalences between bifibrant objects, and that the derived unit $\C \lrar \RR \R^{\coc} \L^{\coc}(\C)$ is a weak equivalence for every cofibrant object $\C \in \Cat({\Mod_R})$.

	Let us start with the second claim. By Remark~\ref{r:car}, it suffices to show that the induced square
	\begin{equation}\label{e:C}
		\xymatrix{
			\C \ar[r]\ar[d] & \C \otimes_R S \ar[d] \\
			\C \otimes_R T \ar[r] & \C \otimes_R U \\
		}
	\end{equation}
	is homotopy Cartesian in $\Cat({\Mod_R})$. By \cite[Lemma 3.1.11]{YonatanCotangent}, it remains to check that for all $x,y \in \Ob(\C)$ the induced square of mapping objects 
	$$ \xymatrix{
		\Map_\C(x,y) \ar[r]\ar[d] & \Map_\C(x,y) \otimes_R S \ar[d] \\
		\Map_\C(x,y) \otimes_R T \ar[r] & \Map_\C(x,y) \otimes_R U \\
	} $$
	is homotopy Cartesian, and that the induced maps $\C \otimes_R T \lrar \C \otimes_R U$ and $\C \otimes_R S \lrar \C \otimes_R U$ are isofibrations. The former is included in the proof of Proposition~\ref{p:example-1}, while the latter follows from Lemma~\ref{l:iso}.

	We next show that $\R^{\coc}$ detects weak equivalences between bifibrant objects. Let $f: s \rar s'$ be a map between bifibrant objects in $\Sec^{\coc}(\F'|_{\I})$ such that $\R^{\coc}(s) \lrar \R^{\coc}(s')$ is a weak equivalence in $\Cat({\Mod_R})$. We need to show that $f$ itself is a weak equivalence. Since $s$ is in particular coCartesian, the homotopy type of $s(1,1)$ is determined by either $s(0,1)$ or $s(1,0)$ via (derived) tensoring; the same holds for $s'$. Thus it remains to show that $f_{0,1}: s(0,1) \lrar s'(0,1)$ and $f_{1,0} : s(1,0) \lrar s'(1,0)$ are weak equivalences (in $\Cat({\Mod_T})$ and $\Cat({\Mod_S})$, respectively). By symmetry, it suffices to treat $f_{0,1}$. We first show that $f_{0,1}$ is fully faithful. Let $x,y$ be two objects of $s(0,1) \in \Cat({\Mod_T})$ and let $x',y'$ denote their images in $s(1,1) \in \Cat({\Mod_U})$. The map $s(1,0) \lrar p^*s(1,1)$ is a fibration (cf. Observation \ref{ob:injfib}), hence in particular an isofibration; while its adjoint $p_!s(1,0) \lrar s(1,1)$ is a weak equivalence and therefore essentially surjective. These facts together imply that $s(1,0) \lrar p^*s(1,1)$ is surjective on objects. Let $x'',y'' \in s(1,0)$ be objects which map to $x',y' \in s(1,1)$ respectively. Since $\R^{\coc}(s) \lrar \R^{\coc}(s')$ is a weak equivalence, the map  
	$$ \Map_{s(0,1)}(x,y) \times^{\der}_{\Map_{s(1,1)}(x',y')} \Map_{s(1,0)}(x'',y'') \lrarsimeq \Map_{s'(0,1)}(x,y) \times^{\der}_{\Map_{s'(1,1)}(x',y')} \Map_{s'(1,0)}(x'',y'') $$
	is a weak equivalence in $\Mod_R$. Following the  proof of Proposition~\ref{p:example-1}, the map $\Map_{s(0,1)}(x,y) \lrar \Map_{s'(0,1)}(x,y)$ is a weak equivalence. Therefore  $f_{0,1}$ is fully faithful. It remains to prove that it is essentially surjective. Let $z$ be an object of $s'(0,1) \in \Cat({\Mod_T})$ and let $z'$ denote its image in $s'(1,1)$. As above, there exists $z'' \in s'(1,0)$ mapping to $z'$. Hence we obtain an object $\omega:=(z,z',z'')$ of  $\R^{\coc}(s')=s'(0,1)\times^{h}_{s'(1,1)}s'(1,0)$ in $\Cat({\Mod_R})$. Since $\R^{\coc}(s) \lrar \R^{\coc}(s')$ is essentially surjective, there exists an object $\ovl{\omega} \in \R^{\coc}(s)$ whose image in $\R^{\coc}(s')$ is equivalent to $\omega$. Now, we have a commutative square 
	\begin{equation}
		\xymatrix{
			\R^{\coc}(s) \ar[r]^{\varphi}\ar[d] & s(1,0) \ar[d]^{f_{0,1}} \\
			\R^{\coc}(s') \ar[r]_{\varphi'} & s'(1,0) \\
		}
	\end{equation}
	where the left vertical map takes $\ovl{\omega}\in \R^{\coc}(s)$ to $\omega\in\R^{\coc}(s')$ (up to equivalence), and the bottom horizontal map takes $\omega$ to $z$. The commutativity of this square implies that $z$ is equivalent to the image under $f_{0,1}$ of $\varphi(\ovl{\omega})\in s(1,0)$. Thus $f_{0,1}: s(0,1) \lrar s'(0,1)$ is essentially surjective, as required.
\end{proof}

Finally, we arrive at the main example of interest. Let $f : R \lrar R'$ be a map in $\CAlg^{\art}$. This induces an adjunction $f_! : \adjunction*{}{\Op({\Mod_R})}{\Op({\Mod_{R'}})}{} : f^*$ between enriched operads, extending \eqref{eq:adjexample-3}. As in that case, we set $(-)\otimes_R R' :=f_!$, and for each  $\O'\in\Op({\Mod_{R'}})$, we use the same notation for its image under $f^*$.
\begin{pro}\label{p:example-4} The functor
	\begin{equation}\label{FOp}
		\F_{\Op} : \CAlg^{\art} \lrar \ModCat, \; R \mapsto  \Op({\Mod_R})
	\end{equation}
 is a formal moduli context.
	\begin{proof} Verifying the first condition of Definition~\ref{d:moduli} is straightforward. The second will mainly rely on Propositions~\ref{p:example-1} and~\ref{p:example-3}. Suppose given a square of the type \eqref{e:square}. We need to show that the induced square of model categories
		\begin{equation}\label{e:square-5}
			\xymatrix{
				\Op(\Mod_R) \ar[r]\ar[d] & \Op(\Mod_S) \ar[d]^{p_!} \\
				\Op(\Mod_T) \ar[r]_{q_!} & \Op(\Mod_U) \\
			}
		\end{equation}
		is homotopy Cartesian. It is required to prove that the adjunction 
		$$ \L^{\coc} :  \adjunction*{}{\Op(\Mod_R)}{\Sec^{\coc}(\F'|_{\I})}{}: \R^{\coc}$$
		is a Quillen equivalence, where as usual $\I \subseteq [1] \times [1]$ denotes the full subcategory spanned by $(0,1),(1,0)$ and $(1,1)$; and  $\F':[1] \times [1] \lrar \ModCat$ is the square~\eqref{e:square-5}.

		We first show that the derived unit $\O \lrar \RR \R^{\coc}\L^{\coc}(\O)$ is a weak equivalence for every cofibrant object $\O \in \Op(\Mod_R)$. By Remark~\ref{r:car}, this is equivalent to saying that the square
		\begin{equation}\label{e:op}
			\xymatrix{
				\O \ar[r]\ar[d] & \O \otimes_R S \ar[d] \\
				\O \otimes_R T \ar[r] & \O \otimes_R U \\
			}
		\end{equation}
		is homotopy Cartesian in $\Op(\Mod_R)$. Following the proof of \cite[Lemma 4.1.3]{Hoang}, it suffices to verify that the induced square of underlying categories (cf. Remark \ref{r:catop})
		$$ \xymatrix{
			\O_1 \ar[r]\ar[d] & \O_1 \otimes_R S \ar[d] \\
			\O_1 \otimes_R T \ar[r] & \O_1 \otimes_R U \\
		} $$
		is homotopy Cartesian in $\Cat(\Mod_R)$, and that the induced square of spaces of operations 
		$$ \xymatrix{
			\O(c_1,\cdots,c_n ; c) \ar[r]\ar[d] & \O(c_1,\cdots,c_n ; c) \otimes_R S \ar[d] \\
			\O (c_1,\cdots,c_n ; c)\otimes_R T \ar[r] & \O(c_1,\cdots,c_n ; c) \otimes_R U \\
		} $$
		is homotopy Cartesian in $\Mod_R$ for every tuple $(c_1,\cdots,c_n;c)$ of colors of $\O$. The former is addressed in the proof of Proposition \ref{p:example-3}, while the latter follows from that of Proposition \ref{p:example-1}.

		It remains to check that $\R^{\coc}$ detects weak equivalences between bifibrant objects. Note that a map $f : \O \lrar \O'$ in $\Op({\Mod_R})$ is a weak equivalence precisely when the induced map $\O_1 \lrar \O'_1$ of underlying categories is a weak equivalence in $\Cat({\Mod_R})$, and for every tuple $(c_1,\cdots,c_n;c)$ the corresponding map $$\O(c_1,\cdots,c_n ; c) \lrar \O'(f(c_1),\cdots,f(c_n) ; f(c))$$ is one in $\Mod_R$. 
	Therefore, the expected property of $\R^{\coc}$ follows directly from that of the functors $\R^{\coc} : \Sec^{\coc}(\F'|_{\I}) \lrar \Cat(\Mod_R)$ (Proposition \ref{p:example-3}) and $\R^{\coc} : \Sec^{\coc}(\F'|_{\I}) \lrar \Mod_R$ (Proposition \ref{p:example-1}).
\end{proof}
\end{pro}

\begin{rem}\label{r:refinemain} Let $\F$ be any of the functors $\F_{\Alg}$, $\F_{\Cat}$, or $\F_{\Op}$. Then every object $X \in \F(\textbf{k})$ is automatically fibrant. Combining this  with Remark \ref{r:specMX}, we conclude that Theorem~\ref{t:main} holds for every $X \in \F(\textbf{k})$.
\end{rem}

\subsection{Deformation theory of dg $\E_n$-operads}\label{s:defEn}

Combining the main results of $\S$$\S$\ref{s:Qprin}-\ref{s:defandQcohom}, we may now relate the deformation theory of the operad $\mathbb{E}_n \in \Op(\C(\textbf{k}))$ (Notation \ref{no:En}) to its Hochschild cohomology. 

Consider the formal moduli context
$$ 		\F_{\Op} : \CAlg^{\art} \lrar \ModCat, \; R \mapsto  \Op({\Mod_R})$$
(cf. Proposition \ref{p:example-4}), and let $\P \in \F_{\Op}(\textbf{k}) \overset{\defi}{=} \Op(\C_{\geq 0}(\textbf{k}))$ be given.

\begin{cons}\label{con:PotimesA} Let $A \in \C_{\geq 0}(\textbf{k})$ be a finite-dimensional (connective) dg $\textbf{k}$-module. For each $n \in \NN$, the square-zero extension $\textbf{k} \ltimes A[n]$  is an artinian $\textbf{k}$-algebra. Consider the $\Omega$-spectrum $M \in \Sp(\CAlg^{\art})$ defined by $M_{n,n} = \textbf{k} \ltimes A[n]$ and $M_{n,m}=\textbf{k}$ for $n\neq m$. This gives rise to an $\Omega$-spectrum $M(\P) \in \T_\P\Op(\C_{\geq 0}(\textbf{k}))$  (cf. Lemma~\ref{l:specMX} and Remark \ref{r:specMX}), in which $$M(\P)_{n,n} = \P \ltimes (\P \otimes A[n])$$ the square-zero extension of $\P$ by $\P \otimes A[n] \in \IbMod(\P)$ (see Remark \ref{r:OmP}). More explicitly, the latter is given at each tuple $(c_1,\cdots,c_r ; c)$ of colors of $\P$ by $$(\P \otimes A[n])(c_1,\cdots,c_r ; c) := \P(c_1,\cdots,c_r ; c)\otimes A[n],$$ 
with the infinitesimal $\P$-bimodule structure naturally induced by the composition in $\P$.
\end{cons}

We obtain a consequence of Theorem \ref{t:main} as follows. 
\begin{cor}\label{co:defop} Suppose further that $\P$ is $\Sigma$-cofibrant. There is a weak equivalence of spaces
	$$  \Def(\P,\textbf{k} \ltimes A) \simeq \HHQ^\star(\P; \P \otimes A[1])$$
where the space on the right is the Quillen cohomology of $\P \in\Op(\C(\textbf{k}))$ with coefficients in $\P \otimes A[1] \in \IbMod(\P)$ (cf. $\S$\ref{s:cotandg}).   	
\begin{proof} By Remark \ref{r:refinemain}, Theorem \ref{t:main} applies to $\P$ without (co)fibrancy requirement. Moreover, if we regard $\P$ as an object of $\Op(\C(\textbf{k}))$ (rather than $\Op(\C_{\geq 0}(\textbf{k}))$), then the $\Sigma$-cofibrancy assumption yields a right Quillen equivalence $$\T_\P\Op(\C(\textbf{k})) \lrarsimeq \IbMod(\P)$$ 
(cf. Theorem \ref{t:keylem}). We complete the proof by observing that, under this functor, the $\Omega$-spectrum $M(\P) \in \T_\P\Op(\C(\textbf{k}))$ corresponds to $\P \otimes A \in \IbMod(\P)$.
\end{proof}
\end{cor}

\begin{example} In the case $A = \textbf{k}$, $\Def(\P,\textbf{k} \ltimes A)$ represents the space of deformations of $\P$ over the \textbf{algebra of dual numbers} $\textbf{k} \ltimes \textbf{k} \cong \textbf{k}[t]/(t^2)$. Consequently, the space $\Def(\P,\textbf{k}[t]/(t^2))$ is  weakly equivalent to $\HHQ^\star(\P; \P^{\si}[1])$ the Quillen cohomology of $\P$ with coefficients in $\P^{\si}[1]\in \IbMod(\P)$.  
\end{example}

Corollaries \ref{co:defop}, \ref{co:QprinEn}, and \ref{co:QprinEnm} together yield the following conclusion.
\begin{thm}\label{co:defEn} There is a fiber sequence of spaces
	$$ \Om^{n+2}\Def(\mathbb{E}_n,\textbf{k} \ltimes A) \lrar \HHH^\star(\mathbb{E}_n ; \mathbb{E}_n \otimes A) \lrar |A|.$$
Moreover, for $k\geq -n-1$ and $k \neq 0$, there is a canonical isomorphism
$$ \pi_{k+n+2}\Def(\mathbb{E}_n,\textbf{k}[t]/(t^2))  \cong \HHH^{-k}(\mathbb{E}_n ; \mathbb{E}_n^{\si}),$$
and, in addition,
$$  \HHH^{0}(\mathbb{E}_n ; \mathbb{E}_n^{\si}) \cong \pi_{n+2}\Def(\mathbb{E}_n,\textbf{k}[t]/(t^2))  \oplus \textbf{k}.$$
\end{thm}

\medskip

\appendix

\section{Homotopy Cartesian squares of model categories}\label{s:hocartmodel}

This appendix provides material for $\S$\ref{s:defandQcohom}. Recall that $\ModCat$ denotes the category whose objects are model categories and whose morphisms are Quillen adjunctions, with  sources and targets being those of the left Quillen functors. Let $\F: \I \lrar \ModCat$ be a diagram of model categories indexed by a small category $\I$. 

\begin{define}\label{d:section} A \textbf{section} of $\F$ is a section of the natural projection $\int_{\I}\F \lrar\I$ from the Grothendieck construction of $\F$ to $\I$. More explicitly, a section consists of an object $s(i) \in \F(i)$ for each $i \in \I$, together with a morphism $f_\alp: \alp_!s(i) \lrar s(j)$ for each map $\alp: i \rar j$ in $\I$, subject to a natural compatibility constraint for every composable pair of morphisms in $\I$. We denote by $\Sec(\F)$ the \textbf{category of sections} of $\F$.
\end{define}

Suppose that each $\F(i)$ is a combinatorial model category. According to \cite{Barwick}, one can then endow $\Sec(\F)$ with either the projective or  injective model structure (see also \cite{Harpazlax} and \cite{Balzin}).  Here, we are interested in the injective model structure, denoted $\Sec(\F)^{\inj}$, in which a map $s \rar s'$ is a cofibration (resp. weak equivalence) if and only if $s(i) \lrar s'(i)$ is a cofibration (resp. weak equivalence) for every $i \in \I$.

\begin{define}\label{d:cocartsec} A section $s: \I \lrar \int_{\I}\F$ is called \textbf{coCartesian} if the composite map 
	$$ \alp_!(s(i)^{\cof}) \lrar \alp_!s(i) \x{f_\alp}{\lrar} s(j) $$ 
	is a weak equivalence for every $\alp: i \rar j$ in $\I$, where $s(i)^{\cof} \lrar s(i)$ is a cofibrant replacement of $s(i)\in\F(i)$. 
\end{define}

\begin{rem} When all model categories $\F(i)$ are in addition left proper, $\Sec(\F)^{\inj}$ is also left proper. One can then left Bousfield localize $\Sec(\F)^{\inj}$, so that new fibrant objects are precisely the injective fibrant coCartesian sections. We denote the resulting model structure by $\Sec(\F)^{\coc}$. If the $\F(i)$ are not all left proper, we may still define the localization $\Sec(\F)^{\coc}$ as a semi-model category (cf. \cite{White}). In either case, $\Sec(\F)^{\coc}$ has the right type, in the sense that its underlying $\infty$-category $\Sec(\F)^{\coc}_\infty$ is a model for the limit of the diagram $$\I \ni i \mapsto \F(i)_\infty \in \Cat_\infty$$ where for each  $\alp: i \rar j$ the functor $\F(i)_\infty \lrar \F(j)_\infty$ is given by $(\alp_!)_\infty$ (cf. \cite[Corollary 3.45]{Balzin}).
\end{rem}
 
\begin{cons}\label{con:squarecat} Let $\F: [1] \times [1] \lrar \ModCat$ be a diagram of combinatorial model categories on the \textit{square category}. Denote by $\I \subseteq [1] \times [1]$ the full subcategory spanned by $(0,1),(1,0)$, and $(1,1)$; and denote by $\alp_i: (0,0) \lrar i$ the unique map from $(0,0)$ to an object $i \in \I$. We then obtain a functor $\L: \F(0,0) \lrar \Sec(\F|_{\I})$ which sends an object $X \in \F(0,0)$ to the section $$\L(X)(i) := (\alp_i)_!X \in \F(i).$$ This functor admits a right adjoint $\R: \Sec(\F|_{\I}) \lrar \F(0,0)$ taking a section $s: \I \lrar \int_{\I} \F$ to the pullback 
	$$ \R(s) := \lim_{i \in \I}\alp_i^*s(i)  \in \F(0,0).$$
	Finally, since $\L$ sends (trivial) cofibrations to injective (trivial) cofibrations, we deduce that $\L\dashv \R$ is a Quillen adjunction.
\end{cons}

\begin{define}\label{d:car}
	Let $\F: [1] \times [1] \lrar \ModCat$ be a diagram of combinatorial model categories. We will say that $\F$ is \textbf{homotopy Cartesian} if the composed left Quillen functor $$\L^{\coc}:= [\F(0,0) \x{\L}{\lrar} \Sec(\F|_{\I})^{\inj} \lrar \Sec(\F|_{\I})^{\coc}]$$ is a left Quillen equivalence, in which $\Sec(\F|_{\I})^{\inj} \lrar \Sec(\F|_{\I})^{\coc}$ is the identity functor, which is a left Quillen functor by construction. We write $\R^{\coc}: \Sec(\F|_{\I})^{\coc} \lrar \F(0,0)$ for the corresponding right Quillen functor.
\end{define}

\begin{rem}\label{r:right}
	Let $\F: [1] \times [1] \lrar \ModCat$ be a diagram of combinatorial model categories depicted as the square of left Quillen functors
	$$ \xymatrix{
		\F(0,0) \ar[r]\ar[d] & \F(0,1) \ar[d] \\
		\F(1,0)  \ar[r] & \F(1,1).  \\
	}$$
	If $\F$ is homotopy Cartesian in the sense of Definition~\ref{d:car}, then the corresponding square of underlying $\infty$-categories
	$$ \xymatrix{
		\F(0,0)_\infty \ar[r]\ar[d] & \F(0,1)_\infty \ar[d] \\
		\F(1,0)_\infty  \ar[r] & \F(1,1)_\infty  \\
	}$$
	is homotopy Cartesian as well. This follows from the fact that the model category $\Sec(\F|_{\I})^{\coc}$ has the right type, as mentioned above.
\end{rem}

\begin{obs}\label{ob:injfib} Let $\G : \I \lrar \ModCat$ be a diagram of combinatorial model categories, with $\I \subseteq [1] \times [1]$ the category described in Construction \ref{con:squarecat}. Let $s \in \Sec(\G)$ be a section of $\G$. If $s$ is injective fibrant, then the structure maps $s(0,1) \lrar \beta^{*}s(1,1)$ and $s(1,0) \lrar \gamma^{*}s(1,1)$ are fibrations, where $\beta : (0,1) \lrar (1,1)$ and $\gamma : (1,0) \lrar (1,1)$ are   morphisms in $\I$.
\end{obs}
\begin{proof} By symmetry, it suffices to show that $s(0,1) \lrar \beta^{*}s(1,1)$ is a fibration in $\G(0,1)$. Suppose given a commutative square in $\G(0,1)$ of the form 
	\begin{equation}\label{eq:injfib}
	\xymatrix{
		A \ar[r]\ar[d] & s(0,1) \ar[d] \\
		B  \ar[r] & \beta^{*}s(1,1),  \\
	}
	\end{equation}
	where $A \lrar B$ is a trivial cofibration. We canonically define two sections $s', s''$ by setting $s'(0,1) = A$, $s'(1,1) = \beta_! A$, $s'(1,0) = \emptyset$; and 
	$s''(0,1) = B$, $s''(1,1) = \beta_! B$, $s''(1,0) = \emptyset$. The map $A \lrar B$ induces a map of sections $s' \lrar s''$, which is clearly an injective trivial cofibration. Moreover, there is a canonical map of sections $s' \lrar s$ with structure maps given by $$s'(0,1) = A \lrar s(0,1), \;\; s'(1,1) = \beta_! A \lrar \beta_! s(0,1) \lrar  s(1,1) \;\; \text{and} \;\; s'(1,0) = \emptyset \lrar s(1,0).$$ Now, since $s$ is injective fibrant, there exists a map $s'' \lrar s$ lifting $s' \lrar s$ along the injective trivial cofibration $s' \lrar s''$. In particular, we obtain a map $B=s''(0,1) \lrar s(0,1)$, which provides a lift for the square \eqref{eq:injfib}.
\end{proof}

\begin{rem}\label{r:car}
	Let $\F: [1] \times [1] \lrar \ModCat$ be a diagram of combinatorial model categories such that for each morphism $\alp$ in $[1] \times [1]$, the right adjoint $\alp^*$ preserves weak equivalences. Since $\L: \F(0,0) \lrar \Sec(\F|_{\I})^{\inj}$ sends cofibrant objects to coCartesian sections, it follows that for any cofibrant object $X \in \F(0,0)$, the derived unit $X \lrar \RR\R^{\coc}\LL \L^{\coc}(X)$ is a weak equivalence if and only if $X \lrar \RR \R \LL \L(X)$ is a weak equivalence. Now, for a section $s: \I \lrar \int_{\I} \F$, the value $\RR\R(s)$ can be identified with the homotopy fiber product $\holim_{i \in \I}\alp_j^*s(i)$, by Observation \ref{ob:injfib}. In conclusion, if $X \in \F(0,0)$ is  cofibrant, then the derived unit map $X \lrar \RR \R^{\coc} \LL \L^{\coc}(X)$ is a weak equivalence if and only if the square
	$$ \xymatrix{
		X \ar[r]\ar[d] & \alp_{(0,1)}^*(\alp_{(0,1)})_!X \ar[d] \\
		\alp_{(1,0)}^*(\alp_{(1,0)})_!X  \ar[r] & \alp_{(1,1)}^*(\alp_{(1,1)})_!X  \\
	}$$
	is homotopy Cartesian in $\F(0,0)$.
\end{rem}

\newpage

\bibliographystyle{amsplain}

\end{document}